\documentclass[10pt]{amsart}
\usepackage{texmaX,semmaX,semtkzX}
\usepackage{dsfont}
\usepackage{makecell}
\usepackage{multirow}
\usepackage{verbatim}

\def\mod{\quad\textup{mod }}
\def\JX{\Jac_X}
\def\JY{\Jac_Y}
\def\cX{{\mathcal{X}}}
\def\triv{\mathds 1}
\def\ct{\cC_{\Theta}^{\textup{sf}}}
\def\dcg{\mathfrak{X}}
\def\chicyc{\chi_{\rm{cyc}}}
\def\iso{\simeq}
\def\mod{\quad\textup{mod }} 
\def\L{L}

\def\CT{\cC_{-\Theta}^{\textup{sf}}}
\def\Aut{\textup{Aut}}
\def\K{\mathcal{K}}

\def\Cluster #1 = #2;{\node[cluster=#2] (#1) {};}
\def\ClusterL #1[#2] = #3;{
  \node[cluster=#3] (#1t) {}; \node[clabelL=#1] (#1l) {$#2$}; \node[clouter=(#1t)(#1l)] (#1) {};}
\def\ClusterD #1[#2] = #3;{
  \node[cluster=#3] (#1t) {}; \node[clabelD=#1] (#1d) {$#2$}; \node[clouter=(#1t)(#1d)] (#1) {};}
\def\ClusterLD #1[#2][#3] = #4;{
  \node[cluster=#4] (#1t) {}; \node[clabelL=#1] (#1l) {$#2$}; 
  \node[clabelD=#1] (#1d) {$#3$}; \node[clouter=(#1t)(#1l)(#1d)] (#1) {};}
\def\ClusterLDName #1[#2][#3][#4] = #5;{
  \node[cluster=#5] (#1t) {}; \node[clabelL=#1] (#1l) {$#2$}; 
  \node[clabelD=#1] (#1d) {$#3$}; 
  \node[scale=\cnamescale,above=\clustersep/3 of #1t,inner sep=0, outer sep=0] (#1n) {$#4$}; 
  \node[clouter=(#1l)(#1d)(#1t)] (#1) {};}

\newcommand{\Langle}{\langle\!\langle}
\newcommand{\Rangle}{\rangle\!\rangle}

\theoremstyle{theorem}
\newtheorem*{conjecture*}{Conjecture}
\newtheorem{convention}{Convention}[section]

\advance\oddsidemargin -1cm
\advance\evensidemargin -1cm
\advance\textwidth 2cm
\advance\textheight 0.3cm

\begin{document}

\title{Parity of ranks of Jacobians of curves}
\author{Vladimir Dokchitser, Holly Green, Alexandros Konstantinou, Adam Morgan}

\address{University College London, London WC1H 0AY, UK}
\email{v.dokchitser@ucl.ac.uk}
\email{alexandros.konstantinou.16@ucl.ac.uk}

\address{University of Bristol, Bristol, BS8 1UG, UK}
\email{holly.green@bristol.ac.uk}

\address{University of Glasgow, Glasgow, G12 8QQ, UK.}
\email{adam.morgan@glasgow.ac.uk}


\marginpar{}

\begin{abstract}
We investigate Selmer groups of Jacobians of curves that admit an action of a non-trivial group of automorphisms, and give applications to the study of the parity of Selmer ranks. Under the Shafarevich--Tate conjecture, we give an expression for the parity of the Mordell--Weil rank of an arbitrary Jacobian in terms of purely local invariants; the latter can be seen as an arithmetic analogue of local root numbers, which, under the Birch--Swinnerton-Dyer conjecture, similarly control parities of ranks of abelian varieties. As an application, we give a new proof of the parity conjecture for elliptic curves. The core of the paper is devoted to developing the arithmetic theory of Jacobians for Galois covers of curves, including decomposition of their $L$-functions, 
and the interplay between Brauer relations and~Selmer~groups.
\end{abstract}

\maketitle

\setcounter{tocdepth}{1}

\tableofcontents

\section{Introduction}\label{Sec_intro}

Tate's proof of the analytic continuation and the functional equation of Hecke $L$-functions interpreted the sign in the functional equation as a product of certain local constants. 
The existence of such an expression for Artin representations was proved by Langlands and Deligne, who, moreover, extended the definition of the local constants to more general Galois representations. The local Langlands correspondence tells us that these local constants are compatible with those arising for $L$-functions of automorphic forms. However, the compatibility of this theory with the Birch--Swinnerton-Dyer conjecture and its generalisations remains an open problem. The sign in the functional equation of the $L$-function of a curve $X$ over a number field $K$ determines the parity of the order of vanishing of the $L$-function at the central point $s=1$, and hence should match the parity of the Mordell--Weil rank of its Jacobian $\JX$. In other words, one expects the {\em parity conjecture}
$$
(-1)^{\rk \JX} =  \prod_{v\textup{ place of }K}w(X/K_v), 
$$
where the $w(X/K_v)$ are the local constants (local root numbers) given by Langlands and Deligne. The aim of the present article is to define local constants $w_{\text{arith}}(X/K_v)$ for curves\footnote{strictly speaking, curves together with a choice of map to $\P^1$}
over local fields $X/K_v$ that determine the parity of ranks of Jacobians of curves over number fields:

\begin{theorem}[=Theorem \ref{thm:PCanalogue}(i)]\label{thm:introjacobian}
If the Shafarevich--Tate conjecture holds for Jacobians of all curves over a number field $K$, then for all curves $X/K$
$$   
(-1)^{\rk \JX} = \prod_{v\textup{ place of }K} w_{\rm{arith}}(X/K_v).
$$
\end{theorem}

We will not address the compatibility of our arithmetic local constants $w_{\text{arith}}(X/K_v)$ with local root numbers $w(X/K_v)$ in general. 
We will, however, illustrate their use by giving a new proof of the parity conjecture for elliptic curves assuming the finiteness of $\sha$.
(A similar result was proved in \cite[Theorem 1.2]{kurast}, with different assumptions on the Tate--Shafarevich group.)

\begin{theorem}[=Theorem \ref{thm:mainthm}]\label{thm:introEPC} 
Let  $E: y^2 = x^3+ax+b$ be an elliptic curve over a number field~$K$ with $a\neq 0$. Let $E'$ be the elliptic curve given by $y^2=x^3-27bx^2-27a^3x$. If $\Sha(E)$ has finite $3$-primary part and $\Sha(E')$ has finite $2$- and $3$-primary parts, then the parity conjecture holds for $E$.
\end{theorem}

\begin{remark}
When $a=0$, the elliptic curve $E$ has a $3$-isogeny over $K$. Thus if $\Sha(E)$ has finite $3$-primary part, then the parity conjecture holds for $E$ by \cite[Theorem 1.8]{kurast}. 
\end{remark}

In the context of understanding the parity of ranks of Jacobians, the key new feature of our approach is the use of higher genus curves to control the arithmetic of curves of lower genus. For instance, the proof of the above theorem for elliptic curves makes use of auxiliary curves of genus 3. 

The role played by these higher genus curves is similar to that of extensions of the base number field, which one makes when studying Heegner points, Iwasawa theory, the parity conjecture, etc. Smooth projective curves over number fields are in 1--to--1 correspondence with their function fields (finitely generated extensions of $\Q$ of transcendence degree 1). For example, the elliptic curve $E: y^2=x^3\!+\!ax\!+\!b$ over $\Q$ corresponds to $\Q(x,\sqrt{x^3+ax+b})$, and its base change to $\Q(i)$ corresponds to the quadratic extension $\Q(x,\sqrt{x^3\!+\!ax\!+\!b},i)$. The quadratic extension $\Q(\sqrt{x},\sqrt{x^3+ax+b})$ instead corresponds to the double cover of $E$ given by $y^2=t^6\!+\!at^2\!+\!b$. As we shall explain, the theory of Jacobians of such covers of curves mimics that of extensions of the base number field extremely well.

To fix ideas, consider a (smooth, projective) curve $X$ over a number field $K$, let $G$ a be finite group of $K$-automorphisms of $X$, and let $X/G$ denote the quotient curve. In the language of function fields, $G$ is simply the Galois group $\Gal(K(X)/K(X/G))$.

First of all, rational points on Jacobians satisfy {\em ``Galois descent''}  (see \cite[Theorem 1.3]{Etale_paper}, or Theorem \ref{thm:introgaloisdescent}): $G$ acts naturally on $\Jac_X(K)$ and 
\begin{equation*}\label{eq:rationalpts}
  \Jac_X(K)_\C^G \simeq \Jac_{X/G}(K)_\C,\tag{$\dagger$}
\end{equation*} 
where ``${}_\C$'' is shorthand for $\otimes_\Z\C$.

Secondly, $L$-functions of covers of curves can be decomposed analogously to the factorisation of Dedekind zeta functions into Artin $L$-functions. For representations $\tau$ of $G$, we will define $L$-functions $L(X^\tau,s)$ that in particular satisfy {\em ``Artin formalism"} (Definition \ref{def:twistedLfn}, Proposition \ref{prop:artinformalism}):
\begin{equation*}\label{eq:Linductivity}
  L(X^{\tau_1\oplus\tau_2}\!,s)=L(X^{\tau_1}\!,s)L(X^{\tau_2}\!,s)\quad\textup{ and }\quad L(X^{\Ind_H^G\triv},s)=L({X/H},s).  \tag{$\ddagger$}
\end{equation*} 
We will also define the root number $w(X^\tau)$ that controls the sign in the conjectural functional equation of $L(X^\tau,s)$ (Definition \ref{def:twistedroot},  Conjecture \ref{conj:twistedlfunctions}(2)).

Finally, these $L$-functions are related to arithmetic by the following generalisation of the Birch and Swinnerton-Dyer conjecture and the parity conjecture
(see Conjecture \ref{conj:wishful} and Theorem \ref{thm:wishful}). 
It is the analogue of the equivariant versions (or that for ``Artin twists of elliptic curves'') of these conjectures in the context of extensions of the base number field, see e.g. \cite[\S 2]{RohV}  or \cite[Conjecture 1.1]{tamroot}.

\begin{conjecture}\label{conj:wishful_intro}
Let $X/K$ be a curve over a number field and $G$ a finite group of $K$-automorphisms of $X$. Let $\tau$ be a representation of $G$.
\begin{enumerate}
\item $\ord_{s=1}L(X^\tau,s) = \langle \tau,  \Jac_X(K)_\C \rangle$.
\item If $\tau$ is self-dual, then $w(X^\tau)=(-1)^{\langle \tau,  \Jac_X(K)_\C \rangle}$.
\end{enumerate}
\end{conjecture}

Thus, the situation is like in the good old days. In particular, these properties allow us to adapt the method of Brauer relations and regulator constants of \cite{tamroot,MR2680426} from the setting of extensions of number fields to covers of curves. This is the key input for controlling the parity of ranks of Jacobians.
Recall that a Brauer relation in a finite group $G$ is a formal linear combination of subgroups (up to conjugacy) $\Theta = \sum_i H_i - \sum_j H_j'$, such that the associated permutation representations $\bigoplus_i \C[G/H_i]$ and $\bigoplus_j \C[G/H_j']$ are isomorphic. 
The theory of ``regulator constants''  associates to a Brauer relation $\Theta$ and a prime $p|\#G$ a 
representation
$\tau_{\Theta,p}$ (Definition \ref{Def: p-adically nontrivial reps} or \cite{tamroot}, Section 2). 
The multiplicity of these representations in the group of points $\Jac_X(K)_\C$ can be controlled using local data:

\begin{theorem}[see Theorems \ref{isogeny_from_pseudo_brauer} \& \ref{Theorem: Local Formula}]
\label{LocalFormula:Intro}
Let $X$ be a curve over a number field $K$, let $G$ be a finite group of $K$-automorphisms of $X$, let $\Theta=\sum_i H_i - \sum_j H_j'$ be a Brauer relation for $G$ and $p$ a prime. 
\begin{enumerate}[leftmargin=*]
\item There is an isogeny 
$\prod_j \Jac_{X/H_j'} \to \prod_i \Jac_{X/H_i}$ 
inducing an equality of $L$-functions of curves
$
\prod_j L({X/H_j'},s) = \prod_i L({X/H_i},s).
$
\item 
Suppose that $\Omega^{1}(\JX)$ is self-dual as a $G$-representation.\footnote
{This is automatic if $G$ is either a symmetric or a dihedral group, or if $K$ has a real place} If $\sha(\Jac_X)$ is finite, then
$$ 
\langle \tau_{\Theta,p}, \Jac_X(K)_\C\rangle \equiv
 \sum_{v\textup{ place of }K}   \ord_p  \Lambda_{\Theta}(X/K_{v})  \ \mod  2.
$$
\end{enumerate}
\end{theorem}

The precise expression for the local invariant $\Lambda_\Theta$ is given in Definition \ref{def:lemma:local_invariants}. When $\mathcal{K}$ is an extension of $\Q_p$, it is of the form
$$
 \Lambda_\Theta(X/\mathcal{K}) = \frac{\prod_i c(\Jac_{X/H_i})}{\prod_j c(\Jac_{X/H_j'})} \cdot \text{powers of 2 and }p,
$$
where $c$ denotes the local Tamagawa number over $\mathcal{K}$. For the present discussion, the crucial point is that it is a purely local invariant. 

To deduce Theorem \ref{thm:introjacobian}, we will show that the supply of curves, Brauer relations and representations $\tau_{\Theta,p}$ is large enough to determine the parity of the rank of every Jacobian. There is, in fact, a further trick up our sleeve, which puts us at an advantage compared to the theory of regulator constants for extensions of number fields. Specifically, if $H\le G$ is such that the quotient curve $X/H$ has genus $0$, then we have full control of the multiplicity of $\Ind_{H}^G\triv$ in $\Jac_X(K)_\C$: by Frobenius reciprocity and Galois descent ($\dagger$), $\langle \Ind_{H}^G\triv, \Jac_X(K)_\C\rangle = \dim \Jac_X(K)_\C^H = \dim \Jac_{\P^1}(K)_\C=0$. We will illustrate this process in Example \ref{ex:S3part1} below in the case of elliptic curves.

Conjecture \ref{conj:wishful_intro}(2) tells us that the parity of $\langle\tau,\Jac_X(K)_\C\rangle$ is controlled by the root number $w(X^\tau)$, which is defined as the product of local root numbers $w(X^\tau/K_v)$. Compatibility with Theorem \ref{LocalFormula:Intro}(2) requires the following ``product formula'' for $w(X^{\tau_{\Theta,p}}/K_v)(-1)^{\ord_p\Lambda_{\Theta}(X/K_v)}$.

\begin{conjecture}\label{conj:brauerwishful}
Let $X$ be a curve defined over a number field $K$, and $G$ a finite group of $K$-automorphisms of $X$. 
Suppose that $\Omega^{1}(\JX)$ is self-dual as a $G$-representation. 
Then for every Brauer relation $\Theta$ in $G$ and every prime $p$, 
$$
 \prod_{v\textup{ place of }K} w(X^{\tau_{\Theta,p}}/K_v) (-1)^{\ord_p \Lambda_\Theta(X/K_v)} = 1.
$$
\end{conjecture}

Our expression for the parity of the rank of a Jacobian in Theorem \ref{thm:introjacobian} is built up from Brauer relations. Thus we can also deduce:

\begin{theorem}[=Corollary \ref{cor:howtoproveparity}]\label{thm:PCconj}
If the Shafarevich--Tate conjecture and Conjecture \ref{conj:brauerwishful} hold for all Jacobians of curves over a fixed number field $K$, then so does the parity conjecture, that is
$$
  (-1)^{\rk \Jac_X} = w(\Jac_X)
$$
for every curve $X$ defined over $K$.
\end{theorem}

Conjecture \ref{conj:brauerwishful} thus gives an approach to the parity conjecture for general Jacobians. For example, if for curves over {\em local} fields $\mathcal K$ we always had the identity $w(X^{\tau_{\Theta,p}}/\mathcal K)=(-1)^{\ord_p  \Lambda_{\Theta}(X/\mathcal K)}$, then, assuming the finiteness of $\sha$, we would deduce the parity conjecture for all Jacobians over number fields. 
Unfortunately, the identity is false in general, and the correct relation between $ w(X^{\tau_{\Theta,p}}/\mathcal K)$ and $\Lambda_{\Theta}(X/\mathcal K)$ is yet to be found. 
However, we understand it well enough to prove the parity conjecture for elliptic curves (see Theorem \ref{thm:introEPC}), as well as Conjecture \ref{conj:brauerwishful} for general semistable curves when $p$ is odd:

\begin{theorem}[forthcoming work \cite{morganlattice}] \label{thm:termcompat}
Conjecture \ref{conj:brauerwishful} holds for all semistable curves $X$ and odd primes $p$, such that $\Jac_X$ has good ordinary reduction at all primes above $p$. Moreover, if $\sha(\Jac_X)$ is finite, then for every finite group $G$ of $K$-automorphisms of $X$ and Brauer relation $\Theta$ for $G$,
$$
  w(X^{\tau_{\Theta,p}})=(-1)^{\langle \tau_{\Theta,p},\Jac_X(K)_\C\rangle},
$$
so that Conjecture \ref{conj:wishful_intro}(2) holds for $X^{\tau_{\Theta,p}}$.
\end{theorem}

Let us illustrate the above ideas with an extended example, and sketch the proof of Theorem \ref{thm:introjacobian} in the setting of elliptic curves.

\begin{example}\label{ex:S3part1} 
Let $E$ be an elliptic curve over a number field $K$ given by $E: y^2 = x^3+ax+b$ with $a\neq 0$. 
We view $E$ as a degree 3 cover of $\mathbb P^1$ (with parameter $y$). 
The condition $a\neq 0$ ensures that the discriminant $h(y)=\text{Disc}(x^3+ax+b-y^2)=-27y^4 + 54by^2 - (4a^3 + 27b^2)$ has no repeated roots. In particular, $K(y,x)/K(y)$ is non-Galois and its Galois closure is an $S_3$-extension of $K(y)$. It has the following field diagram, with the corresponding covers of curves given on the right.

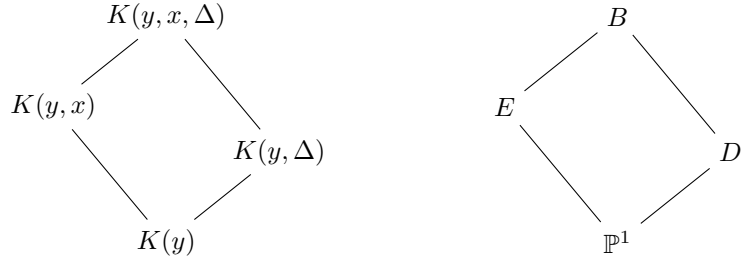
\begin{figure}[h!]
\begin{center}
\begin{tikzpicture}
    \node (Q1) at (0,0) {$K(y)$};
    \node (Q2) at (1.5,1.2) {$K(y, \Delta)$};
    \node (Q3) at (-1.5,1.8) {$K(y,x)$};
    \node (Q4) at (0,3) {$K(y,x,\Delta)$};
    \draw (Q4)--(Q2);
    \draw (Q4)--(Q3);
    \draw (Q3)--(Q1);
    \draw (Q2)--(Q1);
    \node (Q5) at (6,0) {$\P^1$};
    \node (Q6) at (7.5,1.2) {$D$};
    \node (Q7) at (4.5,1.8) {$E$};
    \node (Q8) at (6,3) {$B$};
    \draw (Q8)--(Q6);
    \draw (Q8)--(Q7);
    \draw (Q7)--(Q5);
    \draw (Q6)--(Q5);
\end{tikzpicture} 
\end{center}
\end{figure}

\noindent Here
$D$ and $B$ are the 
curves\footnote{when giving a curve by affine equations, we always mean the unique smooth projective curve birational to it}
given by $\Delta^2 = h(y)$ and by $\{y^2 = f(x),\;\Delta^2 = h(y)\}$. The Jacobian of $D$ is the elliptic curve denoted $E'$ in the statement of Theorem \ref{thm:introEPC} (cf. \cite[\S4]{Cre01}).

Now, $S_3$ acts on $B$ by automorphisms, and hence on the $K$-rational points of its Jacobian. 

Consider the decomposition $\Jac_B(K)_\C=\triv^{\oplus n}\oplus\epsilon^{\oplus m}\oplus \rho^{\oplus k}$ into irreducible representations, where $\triv $, $\epsilon$, $\rho$ denote the trivial, sign and 2-dimensional irreducible representation of $S_3$. 
Galois descent (\ref{eq:rationalpts}) for $G = S_3$ tells us that $n=\rk\Jac_{\P^1}=0$. 
Repeating with $G = C_2$ and $C_3$ gives $\rk E= k$ and $\rk\Jac_D = m$, respectively. In particular, we see that $\rk\Jac_B = 2\rk E + \rk\Jac_D$.

Assuming the Birch and Swinnerton-Dyer conjecture, we can also deduce this from decomposing $L$-functions. Indeed, by Artin formalism (\ref{eq:Linductivity}),
\begin{equation*}\label{eq:Lindentity}
L(B,s)L(\mathbb P^1,s)^2 = L(B^{\mathds 1 \oplus \mathds 1 \oplus\mathds 1 \oplus\epsilon\oplus\rho\oplus \rho},s)  = L(E,s)^2L(D,s),\tag{$*$}
\end{equation*} 
and taking $s\to 1$ recovers $\rk\Jac_B = 2\rk E + \rk\Jac_D$.
It is also easy to verify Conjecture \ref{conj:wishful_intro}(1): for example, $\ord_{s=1}L(B^{\rho},s)=\ord_{s=1}\frac{L(E,s)}{L(\P^1,s)} = k = \langle \rho, \Jac_B(K)_\C\rangle$.

Note that the identity ($*$) also indicates that $B$ has genus $3$ (comparing degrees of Euler factors) and that $E\times E\times \Jac_D$ is isogenous to $\Jac_B$ (by Faltings' theorem). 

The Birch--Swinnerton-Dyer conjecture is compatible with isogenies (Cassels--Tate \cite{cassels,tateanalog}), so, considering the leading terms of the $L$-functions at $s=1$, we obtain a corresponding identity of the form
$$
 \frac{\Reg_{E}^2\Reg_{\Jac_D}}{\Reg_{\Jac_B}} = \frac{|\sha({\Jac B})|}{|\sha(E)|^2|\sha({\Jac D})|} \frac{|E(K)_{\text{tors}}|^4|\Jac_D(K)_{\text{tors}}|^2}{|\Jac_B(K)_{\text{tors}}|^2} \cdot \prod_v \text{(local  terms)},
$$
where the local terms account for the periods, local Tamagawa numbers and other local ``fudge factors'' in the BSD formula. A computation with heights shows that the ratio of regulators is of the form $3^{n+m+k}\cdot (\textup{square})$, so the above expression  leads to
$3^{n+m+k} =   \prod_v \Lambda_v\cdot (\textup{square})$,
for a suitable (purely local) term $\Lambda_v$ involving invariants of $E/K_v$, $D/K_v$ and $B/K_v$.
Finally, taking $3$-adic valuations gives a formula for the parity of the sum of the  ranks of $E$ and $\Jac_D$ in terms of local data:
$$
 \rk E + \rk \Jac_D \equiv \sum_v \ord_3 \Lambda_v \mod 2.
$$

Theorem \ref{LocalFormula:Intro} automates this process. Formula ($*$) came from the representation-theoretic identity $\Ind_{\{1\}}^{G}\triv \oplus \triv^{\oplus 2}\simeq (\Ind_{C_2}^G\triv)^{\oplus 2}\oplus \Ind_{C_3}^G\triv$, which corresponds to the Brauer relation $\Theta = 2C_2 + C_3 - 2S_3 - \{1\}$. The theorem immediately tells of the existence of an isogeny $\Jac_B\to E\times E\times\Jac_D$ and the $L$-function identity ($*$). Moreover, this Brauer relation has $\tau_{\Theta,3}=\triv\oplus\epsilon\oplus\rho$ (see e.g. Example \ref{rem:tautable} or \cite[Ex. 2.53]{tamroot}), so, assuming the finiteness of $\sha(\Jac_B)$, Theorem \ref{LocalFormula:Intro}(2) together with Galois descent ($\dagger$) give
$$
 \rk E + \rk \Jac_D = 
  \langle \tau_{\Theta,3}, \Jac_B(K)_\C\rangle \equiv
 \sum_{v} \ord_3 \Lambda_\Theta(B/K_v) \mod 2.
$$
\end{example}

\begin{example}
\label{ex:C2C2rankrecipe}
Consider a hyperelliptic curve of the form $X:w^2 = g(z^2)$ over a number field~$K$. It has a natural action of $C_2\times C_2$, which gives the following diagrams of function fields 
and of the corresponding covers of curves.
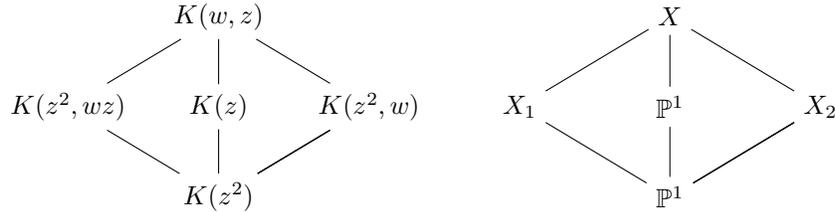
\begin{figure}[h!]
\centering\begin{tikzpicture}
    \node (Q1) at (0,0) {$K(z^2)$};
    \node (Q2) at (0,1.2) {$K(z)$};
    \node (Q3) at (-2, 1.2) {$K(z^2,wz)$};
    \node (Q4) at (0, 2.4) {$K(w,z)$};
	\node (Q5) at (2, 1.2) {$K(z^2, w)$};
	    \draw (Q4)--(Q3);
	    \draw (Q4)--(Q5);
    \draw (Q4)--(Q2);
    \draw (Q2)--(Q1);
    \draw (Q3)--(Q1);
    \draw (Q5)--(Q1);
        \draw (Q5)--(Q1);
            \draw (Q5)--(Q1);
    \node (Q6) at (6,0) {$\P^1$};
    \node (Q7) at (6,1.2) {$\P^1$};
    \node (Q8) at (4, 1.2) {$X_1$};
    \node (Q9) at (6, 2.4) {$X$};
	\node (Q10) at (8, 1.2) {$X_2$};
	    \draw (Q9)--(Q8);
	    \draw (Q9)--(Q10);
    \draw (Q9)--(Q7);
    \draw (Q7)--(Q6);
    \draw (Q8)--(Q6);
    \draw (Q10)--(Q6);
        \draw (Q10)--(Q6);
            \draw (Q10)--(Q6);
\end{tikzpicture}\end{figure}

\noindent The curves $X_1$ and $X_2$ are explicitly given by $t^2=rg(r)$ and by $s^2=g(r)$, respectively.

The group $C_2\times C_2$ has the Brauer relation $\Psi=C_2^a + C_2^b + C_2^c - 2C_2\times C_2 - \{1\}$, which has $\tau_{\Psi,2} = \Ind_{\{1\}}^{C_2\times C_2} {\triv}$; here the $C_2^{i}$ are the three subgroups of order two.
Theorem~ \ref{LocalFormula:Intro}(1) tells us that there is a $K$-isogeny $
\Jac_X\to \Jac_{X_1}\times\Jac_{X_2}.
$
(When $g$ is a square-free cubic polynomial, this recovers the classical fact that the Jacobian of $w^2 = g(z^2)$ is isogenous to a product of two elliptic curves.)
Assuming finiteness of $\Sha(\Jac_X)$,
Theorem \ref{LocalFormula:Intro}(2) tells us
\begin{equation*}
  \rk \Jac_{X} = \langle \tau_{\Psi,2}, \Jac_X(K)_\C\rangle \equiv \sum_{v \textup{ place of }K}\ord_2\Lambda_\Psi(X/K_v) \mod 2.
\end{equation*}
In other words, we can control the parity of the rank of the Jacobians of such curves.
\end{example}

Combining Examples \ref{ex:S3part1}  and \ref{ex:C2C2rankrecipe} (for the curve $D$) we immediately deduce the following expression for the parity of ranks of elliptic curves. 

\begin{theorem}\label{thm:recipe} 
Let $E$ be an elliptic curve over a number field $K$ given by $E: y^2 = f(x)= x^3+ax+b$ with $a\neq 0$. 
Let  $g(y^2)\in K[y]$ be the discriminant of $f(x)-y^2$, and define curves $D:\Delta^2 = g(y^2)$ and $B: \{y^2 = f(x),\;\Delta^2 = g(y^2)\}$. 
If $\Sha(E)$ and $\Sha(\Jac_D)$ are finite then
 \begin{equation*}\rk E\equiv\sum_{v \textup{ place of }K}\big(\ord_3\Lambda_\Theta(B/K_v)+\ord_2\Lambda_\Psi(D/K_v)\big)\mod 2,\end{equation*}
where $\Theta$ and $\Psi$ are the Brauer relations from Examples \ref{ex:S3part1} and \ref{ex:C2C2rankrecipe} for the groups $S_3$ and $C_2\!\times\! C_2$ acting on $B$ and $D$, respectively. 
\end{theorem}

In order to prove the parity conjecture for elliptic curves (Theorem \ref{thm:introEPC})  we need to compare the local terms $\Lambda_\Theta$ and $\Lambda_\Psi$ to the corresponding local root numbers. We will, in fact, prove Conjecture \ref{conj:brauerwishful} for the two Brauer relations $\Theta$ and $\Psi$ (Theorem \ref{thm:3PC} \& Remark \ref{rem:psicompat}). Theorem~\ref{thm:introEPC} will then follow from a mild strengthening of Theorem \ref{thm:recipe}.

\begin{remark}\label{Instances_of_isogenies}

As stated, Theorem \ref{LocalFormula:Intro}(1) is a result of Kani and Rosen (see \cite[Theorem 3]{MR1000113}). We will generalise the notion of Brauer relations to “pseudo Brauer relations” (Definition \ref{def:pseudo Brauer_relations}), which also give rise to isogenies and can be used to obtain data on parities of ranks. 

Brauer relations appear to be a very rich source of isogenies. 
One can verify 
that isogenies in all of the following cases can be obtained by applying Theorem \ref{LocalFormula:Intro}(1) to suitable pseudo Brauer relations \cite{Alexpaper} (with certain assumptions in (3)--(5)):
\begin{enumerate}
\item Isogenies between elliptic curves,
\item Isogenies $\Jac_X \times \Jac_{X_d} \to \Res_{K(\sqrt{d})/K} \Jac_X$,
where $X$ is a hyperelliptic curve, $X_d$ is its quadratic twist by $d\in K^\times$, and $\Res_{K(\sqrt{d})/K}$ denotes Weil restriction from $K(\sqrt{d})$ to $K$,
\item Richelot isogenies between Jacobians of genus $2$ curves, 
\item Isogenies from Jacobians of genus 2 curves to products of elliptic curves,
\item 
Isogenies $\Jac_X \to \Jac_Y\times \Jac_Z$, where $X$ is a curve that admits an unramified double cover to a trigonal curve $Y$, and $\Jac_Z$ is the associated Prym,
\item Isogenies between products of Weil restrictions, $\prod_i \Res_{F^{H_i}/K}\Jac_Y \to \prod_i \Res_{F^{H_j'}/K}\Jac_Y$, where $Y$ is a curve over $K$, $F/K$ is a Galois extension and $\sum H_i- \sum H_j'$ is a Brauer relation in its Galois group (see \cite[Proof of Theorem 2.3]{MR2680426}).
\end{enumerate}

Isogenies have been extensively used to derive formulae for the parities of various ranks in terms of local data, including all of those listed above:
\begin{enumerate}
\item Cassels (see \cite[Appendix]{MR2167089}),
\item Kramer for $X$ elliptic \cite{MR597871} and Morgan for $X$ hyperelliptic \cite{morgan},
\item Dokchitser--Maistret \cite{DM2019},
\item Coates--Fukaya--Kato--Sujatha for $p\neq2$ \cite{CFKS} and Green--Maistret for $p=2$ \cite{HollyCeline},
\item Docking \cite{docking20212},
\item Mazur--Rubin for dihedral groups \cite{MazurRubin} and Dokchitser--Dokchitser \cite{MR2680426}.
\end{enumerate}
Theorem \ref{LocalFormula:Intro}(2) provides the means for deriving local formulae for all of these rank expressions in a uniform way \cite{Alexpaper}. 
\end{remark}

\subsection{Overview of the paper} \label{ss:overview}

The core of this paper relies on the study of Galois covers of curves. One subtlety is that we adhere to the following convention:

\begin{convention}  \label{convention_curves}
Throughout this paper, curves are assumed to be  smooth and proper, but are not assumed to be connected, nor are their connected components assumed to be geometrically connected.
\end{convention}

The reason for considering this broader notion of a `curve' is that they arise naturally in the context of Galois covers. For example, considering Example \ref{ex:S3part1} with $ a = 0 $ and $ K = \mathbb{Q} $, the discriminant curve coincides with the normalization of $ D: \Delta^2 = -27(y^2 - b)^2 $, which is not geometrically connected. After base changing to $ \mathbb{Q}_7 $, $ D $ fails to be connected. For a thorough treatment of the arithmetic of curves and Jacobians in the context of Convention \ref{convention_curves}, see \cite{Etale_paper}. 

In \S \ref{sec:motivic}, we construct the $L$-functions $L(X^{\tau}, s)$ and local root numbers $w(X^{\tau})$ and show that they satisfy the Artin formalism \eqref{eq:Linductivity} mentioned above (Proposition \ref{prop:artinformalism}). These are constructed from the 
$\ell$-adic representations $\Hom_G(\tau\otimes_{\overline\Q}\overline{\Q}_\ell, H_\ell^1(X))$ where $H_{\ell}^1(X)=H^1_{\acute{e}t}((\Jac_{X})_{\overline \cK}, \overline \Q_{\ell})$. We show that these form a compatible system of $\ell$-adic representations and possess other desired properties (Theorem \ref{thm:indep_l} \& Corollary \ref{cor:indep_l}). We conclude by showing that Conjecture \ref{conj:wishful_intro} follows from the Birch--Swinnerton-Dyer conjecture and other standard conjectures on $L$-functions (Conjecture~\ref{conj:wishful} \& Theorem \ref{thm:wishful}). 

\S \ref{Sec_rep_theory}--\S \ref{sec:Rank_Parity} are devoted to proving Theorem \ref{LocalFormula:Intro}(2) for Selmer groups.

In \S\ref{Sec_rep_theory}, we introduce pseudo Brauer relations, a generalisation of Brauer relations, and we verify that the theory of regulator constants remains applicable in this setting (Theorem \ref{Theorem: Independence of pairing}). 

In \S \ref{sec:isogenies}, we discuss isogenies arising from pseudo Brauer relations (as in Theorem \ref{LocalFormula:Intro}(1)).

In \S \ref{sec:Rank_Parity}, we prove Theorem \ref{LocalFormula:Intro}(2) and its generalisation to Selmer groups (Theorem \ref{Theorem: Local Formula}). In particular, we define $\Lambda_{\Theta}(X/\mathcal K)$, an explicit invariant associated to curves over local fields $\mathcal{K}$ of characteristic $0$ and pseudo Brauer relations for their automorphism groups (Definition \ref{def:lemma:local_invariants}). We also discuss variations of Theorem \ref{Theorem: Local Formula} which use alternative local invariants (Theorems \ref{Theorem: Local_Formula_dependent_on_differnetials} \& \ref{thm:lambdatilde}).

\S \ref{Sec_ellcurves}--\S \ref{Sec_ss_parity} focus on applications of these results to the parity conjecture.

In \S \ref{Sec_ellcurves}, we prove Theorem \ref{thm:introEPC} via the $3$-parity conjecture for $E\times \textup{Jac}_D$ (Theorem \ref{thm:3PC}) and the $2$-parity conjecture for elliptic curves admitting a $2$-isogeny.

In \S \ref{Sec_applications}, we prove that there are enough Brauer relations and representations of the form $\tau_{\Theta, p}$ to 
control the parity of the rank of arbitrary Jacobians, and to deduce Theorem \ref{thm:introjacobian}.

In \S \ref{Sec_ss_parity}, we describe an approach to the parity conjecture for arbitrary Jacobians from Theorem~\ref{thm:introjacobian}. This relies on a conjectural relationship between the local invariant $\Lambda_{\Theta}(X/K_v)$ and the local root number $w(X^{\tau_{\Theta,p}}/K_v)$ (Conjecture \ref{conj:brauerwishful}). We end by presenting some evidence towards this conjecture (Theorem \ref{thm:compatibility} \& Corollary \ref{cor:mostgeneralPC}).

\subsection*{Acknowledgements} We would like to thank Alex Bartel for many helpful discussions and suggestions.  The first author was partially supported by a Royal Society University Research Fellowship. The second author was supported by University College London and the EPSRC studentship grant EP/R513143/1 during part of this work. The final author is supported by the Engineering and Physical Sciences Research Council (EPSRC) grant EP/V006541/1 ‘Selmer groups, Arithmetic Statistics and Parity Conjectures’; parts of this work were aditionally carried out when they were supported by the Leibniz fellowship programme at the Mathematisches Forschungsinstitut Oberwolfach.

\subsection{Notation } \label{General_notation} Throughout this paper, we adhere to the following notation. We write $K$ for a number field and $v$ for a finite place. We write $L$ for any field and $\mathcal K$ for a local field, usually of characteristic 0.

\medskip
\begin{tabular}{p{0.135\textwidth}p{0.81\textwidth}}
$X$ & a curve (see Convention \ref{convention_curves})\\
$\JX$ & the Jacobian variety of $X$ \\
$X/H$ & quotient of $X$ defined over $L$ by a finite group $H\leq \Aut_L(X)$\\
$A$ & an abelian variety \\
$\mathcal{X}_{p}(A)$ & $\text{Hom}_{\mathbb{Z}_{p}}(\varinjlim \text{Sel}_{p^n}(A), \mathbb{Q}_{p}/\mathbb{Z}_{p})\otimes \mathbb{Q}_{p}$, the dual $p^{\infty}$-Selmer group of $A$\\
$\rk_p A$ & the $p^{\infty}$-Selmer rank of $A$, that is the $\Z_p$-corank of $\mathcal{X}_{p}(A)$\\
$w(A)$ & the global root number of $A$ defined over number field\\
$w(A/\mathcal K)$ & the local root number of $A$ defined over a local field $\mathcal K$\\
$V_\ell(A)$& $T_\ell(A)\otimes \mathbb{Q}_\ell$, where $T_\ell(A)$ is the $\ell$-adic Tate module of $A$\\
$\Omega^{1}(A)$ & the $L$-vector space of regular differentials on $A$ defined over a field $L$\\
$c(A)$ & the Tamagawa number of $A$ defined over a non-archimedean local field\\
$|\cdot|_{\mathcal K}$, $|\cdot|_{v}$ & the normalised absolute value on $\mathcal K$ (resp. $K_v$), extended to $\overline{\mathcal K}$ (resp. $\overline{K_v}$)\\
$\Frob_\cK$ & a (choice of) arithmetic Frobenius element in the absolute Galois group of $\cK$ \\
$I_\cK$ & inertia group of $\cK$ \\
$\langle\cdot,\cdot\rangle$ & the inner product of characters of $G$-representations \\
$\rho^*$, $\rho^H$ & the dual (resp. $H$-invariant vectors, for $H\leq G$) of a $G$-representation $\rho$ \\
$\Theta $ & a (pseudo) Brauer relation, see Definition \ref{def:pseudo Brauer_relations}
\\
$\mathcal{C}_{\Theta}(\mathcal{V})$ & regulator constant for $\mathcal{V}$, see Definition \ref{def:regulator_constants_for_pseudo_brauer}\\
$\mathcal C_{\Theta}^{\mathcal{B}_1, \mathcal{B}_2}(\mathcal{V})$ & $\mathcal C_{\Theta}(\mathcal{V})$ computed with respect to the bases $\mathcal{B}_{1}$, $\mathcal{B}_{2}$, see Definition \ref{def:regulator_constants_for_pseudo_brauer}\\
$\tau_{\Theta, p}$& self-dual $\mathbb{C}[G]$-representation encoding regulator constants, see Definition \ref{Def: p-adically nontrivial reps}
 \\
$\Lambda_{\Theta}(X/{\mathcal{K}})$ & an explicit local invariant of $X$, see Definition \ref{def:lemma:local_invariants}\\
 $D_{2n}$& the dihedral group of order $2n$\\
  $\mathds 1$& the trivial representation
\end{tabular}\medskip

We remind the reader that the parity conjecture has the following analogue for Selmer groups.

\begin{conjecture}[The $p$-parity conjecture] Let $A$ be an abelian variety over a number field $K$ and $p$ a prime. Then
$$(-1)^{\rk_pA} = w(A).$$
\end{conjecture}

We will frequently use the following result regarding the action of automorphisms of a curve $X$ on its Jacobian, including the ``Galois descent'' \eqref{eq:rationalpts} mentioned above.

\begin{theorem}
\label{thm:introgaloisdescent}
Let $X/K$ be a curve over a field of characteristic 0, let $G$ be a finite subgroup of $\textup{Aut}_K(X)$. Then, 
\begin{enumerate} 
    \item $V_\ell(\Jac_X)^G \simeq V_\ell(\Jac_{X/G})$,
    \item $\Omega^1(\Jac_X)^G \simeq \Omega^1(\Jac_{X/G})$.
\end{enumerate}
When $K$ is a number field,
\begin{enumerate} 
  \item[(3)] $(\Jac_X(K)\otimes\Q)^G \simeq \Jac_{X/G}(K)\otimes\Q$, 
  \item[(4)] $\cX_p(\Jac_X)^G \simeq \cX_p(\Jac_{X/G})$,
\end{enumerate}
and moreover,
\begin{enumerate}
  \item[(5)] $\cX_p(\JX)$ is self-dual as a $G$-representation,
  \item[(6)] if $\Sha(\Jac_X)[p^\infty]$ is finite then $\Sha(\Jac_{X/G})[p^\infty]$ is finite,
    \item[(7)] if a representation $\rho$ does not appear in the $\ell$-adic Tate module $V_\ell(\JX)$ then it does not appear in rational points or in $p^\infty$-Selmer groups: 
$$
\langle \rho,V_\ell(\JX) \rangle =0 \quad\implies\quad
\langle\rho, \JX(K)\otimes\Q \rangle =\langle\rho, \cX_p(\JX) \rangle=0. 
$$
Similarly, if $\langle \rho,V_\ell(\JX) \rangle =0$ then  $\langle\rho, \Omega^{1}(\JX)\rangle=0$ for general characteristic $0$ fields.
\end{enumerate}

\end{theorem}

\begin{proof} This is proved in \cite{Etale_paper}: for (1), (3), (4) see Theorem 1.3, for (2) see Remark 4.29 with $F = \Omega^1(-)$, for (5) see Theorem 1.2, for (6) see Theorem 5.2(3) and for (7) see Proposition~1.4 and Remark 5.6.
\end{proof} 

\section{Motivic pieces of curves}\label{sec:motivic}

In the introduction we factorised the $L$-function of a curve $X$ with a group of automorphisms $G$ as a product of the $L$-function ``pieces'' $L(X^\tau,s)$. 
Decompositions of this kind were already mentioned by Serre \cite{SerreZeta} as coming from ``Artin $L$-functions in the setting of schemes''.
In the present section we formalise this construction and justify the generalisation of the Birch--Swinnerton-Dyer conjecture for these $L$-functions given in Conjecture \ref{conj:wishful_intro}. The reader who is willing to take the formulation of the latter conjecture on trust (and its Selmer group analogue given in Conjecture \ref{conj:wishful}) can fairly safely skip this section; the only other results that will be used later are the existence and basic properties of the  local and global root numbers $w(X^\tau/\cK)$ and $w(X^\tau)$ (Propositions \ref{prop:artinformalism}, \ref{prop:rootnumber}) in \S\S\ref{Sec_applications}--\ref{Sec_ss_parity}.

Formally, $X^\tau$ should be constructed as a motive --- it comes from the 1\textsuperscript{st} cohomology groups of the variety $X$ and the idempotent corresponding to $\tau$ in $\C[G]$. We will not set up the entire motivic machinery here and instead will just concentrate of the system of $\ell$-adic representations attached to $X^\tau$.
We will show that these are independent of $\ell$ and form a compatible system and satisfy other desired properties (Theorem \ref{thm:indep_l}, Corollary \ref{cor:indep_l}).

Throughout this section we will phrase everything in terms of principally polarised abelian varieties $A$ with an action of a finite group $G$ by automorphisms. The case of interest for the rest of the paper is $A=\Jac_X$ for a curve $X/K$ with $G$ coming from $K$-automorphisms of $X$. Thus we will simply write
$$
 L(X^\tau,s) = L(A^{\tau},s) \qquad \text{and} \qquad w(X^\tau)=w(A^\tau).
$$
While much of what follows remains true without assuming the existence of a principal polarisation, several statements are cleaner in the presence of this assumption. Since our main application is to Jacobians of curves, we elect to work in the principally polarised setting throughout.

In this section we will adopt the following conventions and notation.

\medskip

\noindent {\bf{\em{Convention.}}}
For the purpose of working with $\ell$-adic representations, we fix embeddings $\overline{\Q}\subset \overline{\Q}_\ell \subset\C$ for all $\ell$ (with the resulting embedding $\overline{\Q}\subset \C$ independent of $\ell$).

\begin{notation}
When a prime $\ell$ is fixed, we write $\chicyc$ for the $\ell$-adic cyclotomic character.
\end{notation}

\begin{notation}
We take all $d$-dimensional representations $\tau$ of finite groups to be valued in $\GL_d(\overline{\Q})$, and implicitly extend scalars to $\GL_d(\overline{\Q}_\ell)$ or $\GL_d(\C)$ whenever necessary.

For a finite group $G$ and a $G$-representation $\tau$, we write $\Q(\tau)$ for the (abelian) number field generated by the values of the character of $\tau$, i.e. by $\Tr \tau(g)$ for all $g\in G$. For $\alpha\in\Gal(\Q(\tau)/\Q)$ we write $\tau^\alpha$ for the representation that is $\alpha$-conjugate to $\tau$, that is $\Tr \tau^\alpha(g) = \alpha(\Tr\tau(g))$ for all $g\in G$.
\end{notation}

\noindent {\bf{\em{Convention.}}}
By an automorphism $\sigma$ of a principally polarised abelian variety, we always mean one that respects the polarisation. That is, such that $\sigma^\dagger \circ  \sigma=1$, where $\dagger$ denotes the Rosati involution.

\begin{definition}
For an 
abelian variety $A$ over a field $L$ we write
$$
H_{\ell}^1(A)= (T_\ell(A)\otimes_{\Z_\ell} \overline \Q_{\ell})^*
$$
for its associated $\ell$-adic Galois representation. 

If a finite group $G$ acts on $A$ by $L$-automorphisms, then $H_{\ell}^1(A)$ carries an induced action of $G$ that commutes with the action of $\Gal(\overline{L}/L)$. For a representation $\tau$ of $G$, we define the $\ell$-adic Galois representation
$$
H_{\ell}^1(A^\tau)= \Hom_G(\tau, H_\ell^1(A)).
$$
We caution the reader that this notation depends on the underlying group $G$, e.g. $H^1_\ell(A^{\mathds{1}})$ is the subspace of $G$-invariants of $H^1_\ell(A)$.
\end{definition}

The following abelian variety $A_G$ plays the role of the (Jacobian of the) quotient curve $X/G$ in \S\ref{Sec_intro}.

\begin{definition}\label{def:ag}
Let $A$ be an abelian variety over a field $L$ and let $G$ be a finite group acting on $A$ by $L$-automorphisms. We write
$$
A_G= \text{connected component of identity of } A^G, \text{ where } A^G=\bigcap_{g\in G}\ker\big(g-1\big).
$$
$A_G$ is an abelian subvariety of $A$, and $A^G$ is a subgroup variety.
\end{definition}

\begin{remark}\label{rmk:agGD}
The abelian variety $A_G$ satisfies the analogue of Galois descent in \S\ref{Sec_intro}($\dagger$):
$$(A(L)\otimes \mathbb{Q})^G\iso A_G(L)\otimes \mathbb{Q}.$$
Indeed, the natural inclusion $A_G\rightarrow A^G$ induces an injection $A_G(L)\rightarrow A^G(L)=A(L)^G$, whose cokernel is annihilated by the order of the component group of $A^G$.
\end{remark}

\begin{remark}[see \cite{Etale_paper}, \S\S3,4] 
Let $X$ be a curve over a field $L$ of characteristic~$0$.
The Jacobian $A\!=\!\Jac_X$ carries a canonical principal polarisation. 
Moreover, if $G$ is a finite group of $L$-automorphisms of $X$, then $G$ acts naturally (but possibly not faithfully) on $\Jac_X$ by $L$-automorphisms. 
In this setting, $\Jac_{X/G}$ is isogenous to $(\Jac_X)_G$, 
and so they can be used interchangeably in many arithmetic situations.
\end{remark}

\subsection{Galois representation of $A^\tau$}

Before discussing the $\ell$-adic representation for $A^\tau$ we first record the following  property of $H^1_\ell(A_G)$, analogous to Theorem \ref{thm:introgaloisdescent}.

\begin{proposition}\label{prop:quotientavrep}
Let $A$ be a principally polarised abelian variety over a field $L$ 
and $G$ a finite group acting on $A$ by $L$-automorphisms. 
For $H\le G$,
$$
 H^1_\ell(A_H) \simeq H^1_\ell(A)^H
$$
as $\Gal(\overline{L}/L)$-representations, and as $G/H \times \Gal(\overline{L}/L)$-representations if $H$ is normal in $G$.
\end{proposition}

\begin{proof}
Write $N_H=\sum_{h\in H}h$, which we view as an endomorphism of $A$. Note that $N_H(A)\subseteq A_H$. Denoting by $i_H$ the inclusion of $A_H$ into $A$,  we thus have maps 
$N_H:A\rightarrow A_H$ and $i_H:A_H \rightarrow A.$
Noting that the image of $N_H^*:H^1_{\ell}(A_H)\rightarrow H^1_\ell(A)$ is contained in $H^1_{\ell}(A)^H$,  we obtain maps 
$$
N_H^*:H^1_{\ell}(A_H)\rightarrow H^1_\ell(A)^H\quad \textup{ and }\quad i_H^*:H^1_\ell(A)^H\rightarrow H^1_{\ell}(A_H).
$$
The composition $N_H\circ i_H$ is multiplication by $|H|$ on $A_H$, hence $i_H^*\circ N_H^*$ is multiplication by $|H|$ on $H^1_{\ell}(A_H)$. Similarly,  $ N_H^*\circ i_H^* $ is multiplication   by $|H|$ on $H^1_{\ell}(A)^H$. 
We conclude that both $N_H^*$ and $i_H^*$ are  isomorphisms as in the statement.
\end{proof}

\begin{proposition}\label{prop:artinh1}
Let $A$ be a principally polarised 
abelian variety over a field $L$ of characteristic 0, let $G$ be a finite group acting on $A$ by $L$-automorphisms, and let $\ell$ be a prime.
There are isomorphisms of $\Gal(\overline{L}/L)$-representations as follows:
\begin{enumerate}[leftmargin=*]
\item For $G$-representations $\tau, \tau'$, $H^1_\ell(A^{\tau\oplus\tau'})\simeq H^1_\ell(A^\tau)\oplus H^1_\ell(A^{\tau'})$;
\item If $H\le G$ and $\rho$ is a representation of $H$, then $H^1_\ell(A^{\Ind_H^G\rho})\simeq H^1_\ell(A^\rho)$;
\item If $\sigma$ factors through $G/N$ for $N\triangleleft G$, then $H^1_\ell(A^\sigma)\simeq H^1_\ell((A_N)^\sigma)$;
\item For $H\le G$, $H^1_\ell(A^{\Ind_H^G\mathds{1}}) \simeq H^1_\ell(A_H)$;
\item There is a non-degenerate $\Gal(\overline{L}/L)$-equivariant bilinear pairing $$H^1_\ell(A^\tau)\times H^1_\ell(A^{\tau^*}) \to \chicyc^*.$$ If $\tau$ is self-dual and orthogonal, this pairing is alternating and $\dim H^1_\ell(A^\tau)$ is even;
\item $H_{\ell}^1(A^\tau)^*\simeq H_{\ell}^1(A^{\tau^{*}})\otimes_{\overline\Q_\ell} \chicyc$; 
\item  If $\tau$ is self-dual, then $(\det H_{\ell}^1(A^\tau))^{\otimes 2} \simeq \chicyc^{*\otimes\dim H_{\ell}^1(A^\tau)}$. If $\tau$ is orthogonal, then $\det H_{\ell}^1(A^\tau) \simeq \chicyc^{*\otimes\frac{1}{2}\dim H_{\ell}^1(A^\tau)}$.
\end{enumerate}
\end{proposition}

\begin{proof} 
(1) Follows from the fact that $\textup{Hom}_G$ commutes with finite direct sums in the first variable. 

(2) Since induction is left adjoint to restriction, for any $G$-representation $V$ there is an isomorphism 
$$
\Hom_G(\Ind_H^G\rho, V)\iso \Hom_H(\rho, V)
$$
which is functorial in $V$. Taking $V=H^1_\ell(X)$ gives the result. 

(3) By Proposition \ref{prop:quotientavrep} $H^1_\ell(A_N)\simeq H^1_\ell(A)^N$. Since $\sigma$ factors through $G/N$, the natural map
$$
\Hom_G\big(\sigma,H^1_\ell(A)^N\big) \rightarrow \Hom_G\big(\sigma,H^1_\ell(A)\big)
$$
is an isomorphism, from which the claim follows.

(4) By (2) and Proposition \ref{prop:quotientavrep} , $H^1_\ell(A^{\Ind_H^G\mathds{1}}) \simeq H^1_\ell(A)^H \simeq H^1_\ell(A_H)$.

(5) It suffices to construct a $G$-invariant pairing
\begin{equation*} \label{eq:pairing_without_G_inv}
  \textup{Hom}(\tau,H^1_\ell(A))\times \textup{Hom}(\tau^*,H^1_\ell(A)) \longrightarrow \chicyc^*   \tag{$*$}
\end{equation*}
with the same properties. Restricting this to $G$-invariants then gives the sought pairing on $H^1_\ell(A^\tau)$ (the restriction remains non-degenerate by \cite[Lemma 2.15]{tamroot}). Note that the standard isomorphisms of vector spaces  
\[\textup{Hom}(\tau,H^1_\ell(A))\iso H^1_\ell(A)\otimes \tau^*\quad \textup{ and }\quad \textup{Hom}(\tau^*,H^1_\ell(A))\iso H^1_\ell(A)\otimes \tau\]
 are $G\times \textup{Gal}(\overline{L}/L)$-equivariant,  where $G$ acts diagonally on the tensor products and $\textup{Gal}(\overline{L}/L)$ acts via the first factors. Having made these identifications, the tensor product of the Weil pairing on $H^1_\ell(A)$ with the natural `evaluation' pairing between $\tau$ and $\tau^*$ gives the pairing ($*$). 
That this is alternating when $\tau$ is self-dual and orthogonal follows from the fact that the tensor product of an antisymmetric pairing with a symmetric pairing is antisymmetric. 

(6) Follows from (5).

(7) The first claim follows from (6) on taking determinants. The second claim follows from the orthogonal case of (5) and the Pfaffian identity $\textup{pf}(MTM^t)=\det(M)\textup{pf}(T),$ where $T$ is an antisymmetric matrix and $M$ an arbitrary square matrix (see e.g. \cite[\S5.2, Proposition 1]{MR0107661}).
\end{proof}

\subsection{Local Galois representation of $A^\tau$}

We now turn to the properties of $H^1_\ell(A^\tau)$ over local fields. We begin by recalling some standard results on the $\ell$-adic representation $H^1_\ell(A)$ for semistable abelian varieties, except that we additionally keep track of the action of a finite group of automorphisms.

\begin{notation}\label{not:ssav}
Let $\cK$ be a non-archimedean local field with ring of integers $\mathcal{O}_\cK$ and residue field $k$.
Let $A/\cK$ be a principally polarised abelian variety, and let $G$ be a finite group acting on $A$.
Recall (e.g. from \cite[Definition 3.4]{grothendieck1971modeles}) that if $A$ is semistable, then by definition there is a short exact sequence of $k$-group schemes
$$
0\longrightarrow T \longrightarrow \mathcal{A}^0_k \longrightarrow B \longrightarrow 0,
$$
where $\mathcal{A}^0_k$ is the identity component of the special fibre of the  N\'{e}ron model $\mathcal{A}/\cO_\cK$ of $A$, $T$ is a torus and $B$ an abelian variety. 

We denote by $\mathfrak{X}_A$ the character lattice of $T$. 

The action of $G$ extends uniquely to $\mathcal{A}$, inducing actions on  $T$, $B$ and  $\mathfrak{X}_A$.
\end{notation}

\begin{proposition}\label{thm:grothendieckadam}
Continuing with the setup of Notation \ref{not:ssav}, let $\ell$ be a prime distinct from the characteristic of $k$.
\begin{enumerate}[leftmargin=*]
\item   We have 
$$
H^1_\ell(A)/H^1_\ell(A)^{I_\cK}\simeq (\mathfrak{X}_A\otimes_{\mathbb{Z}}\overline{\Q}_\ell)\otimes \chicyc^*
$$  
as $G\times \Gal(\bar{k}/k)$-modules.
\item Let $g=(g_0,\sigma)\in G\times W_\cK$, where $W_\cK$ is the Weil group of $\cK$. Then $\Tr (g|H^1_\ell(A)^{I_\cK})$ is a rational number independent of $\ell$. If $\sigma$ acts on $\overline{k}$ as a non-negative power of the geometric Frobenius, then $\Tr (g|H^1_\ell(A)^{I_\cK})\in \mathbb{Z}$.
\end{enumerate}
\end{proposition}

\begin{proof}  
The Weil pairing $T_\ell(A)\otimes T_\ell(A)\rightarrow \mathbb{Z}_\ell(1)$ induces an isomorphism  
\begin{equation*} \label{Weil_pairing_isom}
H^1_{\ell}(A)\iso T_\ell(A)(-1)\otimes_{\mathbb{Z}_{\ell}} \overline{\mathbb{Q}}_\ell \tag{$*$}
\end{equation*}
of $G\times \textup{Gal}(\overline{K}/K)$-modules, where the $-1$ denotes a Tate twist.
The  description of the $\ell$-adic Tate module of a semistable abelian variety given in \cite{grothendieck1971modeles} (see also \cite[Section 3]{papikian13} for a summary) is readily checked to be compatible with the action of $G$. In particular, we have a $G\times \textup{Gal}(\overline{K}/K)$-stable filtration 
$$
0\subseteq T_\ell(A)^t\subseteq T_\ell(A)^{I_\cK}\subseteq T_\ell(A),
$$
whose graded pieces are unramified and are, respectively, 
$$
\mathfrak{X}_*(T)\otimes \mathbb{Z}_\ell(1),\quad T_\ell(B)\quad \textup{ and }\quad \mathfrak{X}_A\otimes \mathbb{Z}_\ell.
$$
Here $\mathfrak{X}_*(T)\simeq \textup{Hom}(\mathfrak{X}_A,\mathbb{Z})$ is the cocharacter lattice of $T$. Twisting by $(-1)$ and using ($*$) we deduce part (1). 

For part (2), the above discussion gives
\begin{equation*} \label{eq:trace_sum}
\textup{Tr}\big(g\mid H^1_{\ell}(A)^{I_\cK}\big)=\textup{Tr}\big(g \mid T_\ell(B)(-1)\big)+\textup{Tr}\big(g \mid \mathfrak{X}_*(T)\big).
\end{equation*}
Consequently, it suffices to consider the trace of $g$ on $T_\ell(B)(-1)$. The Weil pairing on $B$ gives an isomorphism of $G\times W_\cK$-modules $T_\ell(B)(-1)\iso \textup{Hom}\big(T_\ell(B^\vee),\mathbb{Z}_\ell\big),$ where $\tau\in G$ acts on $T_\ell(B^\vee)$ as $(\tau^{-1})^\vee$. If $\sigma$ acts on $\bar{k}$ as a non-negative power of the geometric Frobenius, then $g^{-1}$ acts on $T_\ell(B^\vee)$ as a geometric endomorphism. Thus, by \cite[Proposition 12.9]{MR861974}, the characteristic polynomial of $g^{-1}$ on  $T_\ell(B^\vee)$ has integer coefficients independent of $\ell$. This implies the result. 
\end{proof}

\begin{proposition}\label{cor:twistedtrace}
Let $A$ be a principally polarised abelian variety over a non-archimedean local field~$\cK$, let $G$ be a finite group acting on $A$ by $\cK$-automorphisms, and let $\ell$ be a prime different from the residue characteristic of $\cK$.
\begin{enumerate}[leftmargin=*]
\item $$\Tr (\Frob_\cK \> | \> H_{\ell}^1(A^\tau)^{I_\cK}) = \frac{1}{|G|}\sum_{g\in G} \Tr (g^{-1}|\tau) \Tr (\Frob_\cK \cdot g|H^{1}_\ell(A)^{I_\cK}).$$
\item For $\alpha\in\Gal(\Q(\tau)/\Q)$ the characteristic polynomials of $\Frob_\cK$ on $H^1_\ell(A^\tau)^{I_\cK}$ and on  \linebreak $H^1_\ell(A^{\tau^\alpha})^{I_\cK}$ are $\alpha$-conjugate. In other words, if \hbox{$\det(t-\Frob_\cK | H_{\ell}^1(A^\tau)^{I_\cK}) = \sum_{i=0}^{n} a_i t^i$}, then \linebreak \hbox{$\det(t-\Frob_\cK | H_{\ell}^1(A^{\tau^\alpha})^{I_\cK}) = \sum_{i=0}^{n} \alpha(a_i) t^i$}.
\item If $A/\cK$ is semistable then $H^1_\ell(A^\tau)/H^1_\ell(A^\tau)^{I_\cK}\simeq \Hom_G(\tau, \dcg_A\otimes\overline{\Q}_\ell)\otimes\chicyc^{-1}$ as a $\Gal(\overline{\cK}/\cK)$-module.
\end{enumerate}
\end{proposition}

\begin{proof}

(1) For irreducible $\tau$ this follows from Lemma \ref{lem:traceVtau}(2) below, applied to $V=H^1_\ell(A)^{I_\cK}$, $H=\Gal(\overline{\cK}/\cK)$, $h=\Frob_\cK$, and observing that $H^1_\ell(A^\tau)^{I_\cK}\simeq \Hom_G(\tau,H^1_\ell(A)^{I_\cK})$ as $\Gal(\overline{\cK}/\cK)$-modules. The general case then follows from Proposition \ref{prop:artinh1}(1).

(2) The coefficients of the characteristic polynomial are symmetric functions in the eigenvalues of Frobenius, and hence can be expressed as universal $\Z$-linear combinations of $\Tr \Frob_\cK^i$ for $i=0, \ldots n$. It therefore suffices to show that $\Tr (\Frob_\cK^i | H^1_\ell(A^{\tau^\alpha})^{I_\cK})=\alpha \Tr(\Frob_\cK^i|H^1_\ell(A^{\tau})^{I_\cK})$. This follows from (1), since by definition $\alpha \Tr (g^{-1}|\tau) = \Tr (g^{-1}|\tau^{\alpha})$ and as $\Tr (\Frob_\cK \cdot g|H^{1}_\ell(A)^{I_\cK})\in\Q$ by Proposition \ref{thm:grothendieckadam}(2).

(3) This follows from Proposition \ref{thm:grothendieckadam}(1).
\end{proof}

\begin{lemma}\label{lem:traceVtau}
Let $V$ be a representation of $G \times H$ with $G$ finite. For a $G$-representation $\tau$ define $V^\tau=\Hom_G(\tau, V)$. 
Then, for an irreducible $G$-representation~$\tau$
\begin{enumerate}
\item $V^{(\tau^{\oplus \dim\tau})}$ is isomorphic to the $\tau$-isotypic component of $V$, and
\item for $h\in H$,
$$
\Tr (h \> | \> V^\tau) = \frac{1}{|G|}\sum_{g\in G} \Tr (g^{-1}|\tau) \Tr (h\cdot g|V).
$$
\end{enumerate}
\end{lemma}

\begin{proof}
(1) The map which sends $V$ to its $\tau$-isotypic component is isomorphic, as a functor, to $V \mapsto \tau \otimes \textup{Hom}(\tau,V)$. Since $H$ acts trivially on $\tau$, then $\tau \otimes \textup{Hom}(\tau,V) \iso V^{\tau^{\oplus \dim\tau}}$ as an $H$-representation. 

(2) $\frac{\dim\tau}{|G|}\sum \Tr(g^{-1}|\tau) g \in\C[G]$ acts as the projector to the $\tau$-isotypic component. \qedhere
\end{proof}

\begin{theorem}\label{thm:indep_l}
Let $A$ be a principally polarised abelian variety over a local field $\cK$ of characteristic 0 with finite residue field $\F_q$. Let $G$ be a finite group acting on $A$ by $\cK$-automorphisms, and let $\tau$ be a representation of $G$.
The local Weil--Deligne representation associated to $H^1_\ell(A^\tau)$ is independent of $\ell$, weight-monodromy compatible and Frobenius-semisimple; more precisely, it can be written in the form
$$
 \rho_1 \oplus (\rho_2\otimes Sp(2)),
$$
where $\rho_i$ are continuous complex representations of the Weil group that are independent of the choice of $\ell\nmid q$, Frobenius acts semisimply on the $\rho_i$ with eigenvalues of absolute value $|q|^{-1+\frac{i}{2}}$, and $Sp(2)$ denotes the 2-dimensional special representation.
\end{theorem}

\begin{proof}

It suffices to prove the result for irreducible $\tau$, as the general case then follows on taking direct sums (Proposition \ref{prop:artinh1}(1)).
As $A$ is an abelian variety, the Weil-Deligne representation associated to $H^1_\ell(A)$ admits as decomposition of the above form (see e.g. \cite{Sab} Proposition 1.10).
By Lemma \ref{lem:traceVtau}(1), $H^1_\ell(A^\tau)^{\oplus \dim\tau}$ is isomorphic to a direct summand of $H^1_\ell(A)$, and therefore also admits such a decomposition (which is also Frobenius-semisimple and weight-monodromy compatible, but may not be independent of $\ell$). Hence so does the Weil-Deligne representation associated to $H^1_\ell(A^\tau)$.

It remains to prove independence of $\ell\nmid q$.
For a fixed integer $d\ge 0$ and finite extension of the coefficient field $F/\Q_\ell$, \cite{DDWeil2} Theorem 7 and Corollary 8 show that there is a finite list of finite extensions $\mathcal{F}_i/\mathcal{K}$ with the property that
\begin{enumerate}
\item Every abelian variety $B/\mathcal{K}$ with $\dim B\le d$ has semistable reduction over each $\mathcal{F}_i$, and
\item Every Frobenius-semisimple weight-monodromy compatible continuous $\ell$-adic representation $\rho:\Gal(\overline{\mathcal{K}}/\mathcal{K})\to \GL_n(F)$ for $n\le 2d$ is uniquely determined by the set of traces $\Tr \rho^{I_{\mathcal{F}_i}}(\Frob_{\mathcal{F}_i})$ for all $i$.
\end{enumerate}
Taking $d=\dim A$ and $\mathcal{F}=\Q_\ell(\tau)$, we deduce that the local Weil--Deligne representation associated to $H^1_\ell(A^\tau)$ is determined by the traces of $\Frob_{\mathcal{F}_i}$ acting on $H^1_\ell(A^\tau)^{I_{\mathcal{F}_i}}$. These traces do not depend on the choice of $\ell$ by Proposition \ref{thm:grothendieckadam}(2) and Proposition \ref{cor:twistedtrace}(1).
\end{proof}

\begin{corollary}\label{cor:indep_l}
Let $A$ be a principally polarised abelian variety over a number field $K$, $G$ be a finite group acting on $A$ by $K$-automorphisms, and $\tau$ a representation of $G$.
Then $H^1_\ell(A^\tau)$ form a compatible system of $\ell$-adic representations, in the sense that for every prime $v$ of $K$, its associated local Weil--Deligne representation is independent of the choice of $\ell$ for $v\nmid \ell$.
\end{corollary}

\subsection{$L$-functions and root numbers of $A^\tau$}\label{ss:Lw}

We refer to Tate's \cite{TateNTB} and Deligne's \cite{DeligneValeurs} for the general definition of $L$-functions and $\epsilon$-factors.

\begin{definition}\label{def:twistedroot}
Let $A$ be a principally polarised abelian variety over a local field $\cK$, let $G$ be a finite group acting on $A$ by $\cK$-automorphisms, let $\tau$ be a representation of $G$ and $\ell$ a prime different from the residue characteristic of $\cK$.

When $\cK$ is non-archimedean the local polynomial is defined by the usual formula, 
$$
P(A^\tau/\cK,T)=\det (1- \Frob_\cK^{-1}T | H^1_\ell(A^\tau)^{I_\cK}).
$$ 
The local root number $w(A^\tau/\cK)$ is defined as the local root number associated to the Galois representation $H^1(A^\tau)$ via the theory of $\epsilon$-factors, 
$$
 w(A^\tau/\cK)=\frac{\epsilon(H^1_\ell(A^\tau), \mu,\psi)}{|\epsilon(H^1_\ell(A^\tau), \mu,\psi)|}
$$ 
for some choice of Haar measure $\mu$ and nontrivial additive character $\psi$ on $\cK$. 
In view of Theorem~\ref{thm:indep_l}, the choice of $\ell$ is immaterial, at least if $\cK$ has characteristic 0 (we will not discuss the case of equal characteristic here).

When $\cK$ is archimedean, the local $\epsilon$-factor (and the $\Gamma$-factor in the functional equation) is defined in terms of the corresponding Hodge structure (see \cite{DeligneValeurs} (5.3)).
We will only use that the Hodge structure for $A^\tau$ is non-zero only in degrees $(0,1)$ and $(1,0)$ and that the dimension satisfies
$$
\dim (H^{(1,0)}(A^\tau) \oplus H^{(0,1)}(A^\tau)) = \dim H^1_\ell(A^\tau).
$$
(Like the $\ell$-adic representation, the Hodge structure is constructed as a $\Hom$ from that of $\tau$ to that of $A$. The former is concentrated in degree $(0,0)$ where it is just the underlying vector space of the representation with its $G$-action. The latter, by the functoriality of the Hodge decomposition (see e.g.  \cite[Section 12]{Ara2012}), satisfies
$$
 H^{(1,0)}(A)\oplus H^{(0,1)}(A) \simeq H^1_\ell(A)
$$
as $G$-representations. The dimension formula follows on applying $\Hom_G(\tau, \bullet).$)

See Proposition \ref{prop:rootnumber} below for an explicit formula for the local root number when $A/K$ is semistable or $K$ is archimedean. 

\end{definition}

\begin{lemma}\label{lem:windep}
The local root number $w(A^\tau/\cK)$ does not depend on the choice of measure $\mu$.
If $\tau$ is orthogonal then $w(A^\tau/\cK)$ does not depend on the choice of additive character $\psi$ and, moreover, $w(A^\tau/\cK)\in \{\pm 1\}$. 
\end{lemma}
\begin{proof}
Independence of $\mu$ is standard and applies to local root numbers generally, as scaling the measure merely scales the $\epsilon$-factor by a positive real number; see \cite{TateNTB} (3.4.3) for the case of Weil representations and (4.1.6) for the general case. Independence of $\psi$ and that the local root number is $\pm 1$ generally holds for representations with positive determinant by \cite{TateNTB} (3.4.4) and (4.1.6), and by \cite{DeligneValeurs} (5.5.1), respectively. This applies to $H^1_\ell(A^\tau)$ by Proposition \ref{prop:artinh1}(5).
\end{proof}

\begin{definition} \label{def:twistedLfn}
Let $A$ be a principally polarised abelian variety over a number field $K$, let $G$ be a finite group acting on $A$ by $K$-automorphisms and $\tau$ a representation of $G$.
The $L$-function and global root number are defined as usual by
$$
L(A^\tau,s) = \prod_{v\nmid \infty} \frac{1}{P(A^\tau/K_v, q_v^{-s})}, \qquad\qquad w(A^{\tau})=\prod_v w(A^\tau/K_v),
$$
where $q_v$ is the size of the residue field at $v$, and the two products are taken over all the non-archimedean places of $K$ and over all the places of $K$, respectively.
(In the present article we will be interested in the representations $\tau_{\Theta,p}$, which are always orthogonal. 
We will thus not discuss to what extent the root number is well-defined for general $\tau$.)
\end{definition}

\begin{remark}
The roots of $P(A^\tau/K_v,T)$ are a subset of those of the local polynomial for $A/K_v$. A standard argument then shows that the $L$-series for $A^\tau$ converges on $\Re(s)>\frac{3}{2}$.
\end{remark}

\begin{lemma}\label{lem:galoisequivariance}
Let $A$ be a principally polarised abelian variety over a number field $K$ and let $G$ be a finite group acting on $A$ by $K$-automorphisms. For a representation $\tau$ of $G$ the coefficients of the $L$-series 
$$
L(A^\tau,s)=\sum_{0\neq\mathbf{n}\triangleleft \mathcal{O}_K} \frac{a_\mathbf{n}}{N_{K/\Q}(n)^s}
$$ 
have $a_{\mathbf{n}}\in \Q(\tau)$. 
Moreover, for $\alpha\in\Gal(\Q(\tau)/\Q)$ the $L$-series for $A^{\tau^\alpha}$ is given by 
$$
L(A^{\tau^\alpha},s)=\sum_{0\neq\mathbf{n}\triangleleft \mathcal{O}_K} \frac{\alpha(a_\mathbf{n})}{N_{K/\Q}(n)^s}.
$$
\end{lemma}

\begin{proof}
This is a direct consequence of Proposition \ref{cor:twistedtrace}(3), which shows the corresponing Galois equivariance property for each Euler factor.
\end{proof}

Standard conjectures on $L$-functions and Deligne's conjecture on the ``Galois equivariance'' properties of $L$-functions (see \cite{DeligneValeurs} \S5.2 and Conjecture 2.7), together with the above Lemma, imply the following conjecture. We will not make the terms $a,b$ explicit as they will not be used in this paper; see \cite{TateNTB} (3.4.7), (4.1.6), (4.2.4) and \cite{DeligneValeurs} 5.3 for formulae relating them to the associated condutor and the appropriate powers of $2$ and $\pi$.

\begin{conjecture}\label{conj:twistedlfunctions}
Let $A$ be a principally polarised abelian variety over a number field $K$ and let $G$ be a finite group acting on $A$ by $K$-automorphisms. For a representation $\tau$ of $G$
\begin{enumerate}[leftmargin=*]
\item $L(A^\tau,s)$ has an analytic continuation to $\C$;

\item $L(A^\tau,s)$ satisfies 
$$
\mathbf{L}(A^\tau,s) = w(A^\tau) \cdot (ab^s)\cdot \mathbf{L}(A^{\tau^*},2-s),
$$
where 
$\mathbf{L}(A^\tau,s)=L(A^\tau,s)\Gamma(s)^{\frac{[K:\Q]}{2}\dim H^1_\ell(A^\tau)}$ and $a, b>0$ are constants that depend on $A^\tau$ and $K$, but not on $s$;
\item $\ord_{s=1} L(A^{\tau^\alpha},s)=\ord_{s=1} L(A^{\tau},s)$ for $\alpha\in\Gal(\Q(\tau)/\Q)$.
\end{enumerate}
\end{conjecture}

These $L$-functions and root numbers satisfy the usual ``Artin formalism''. This follows from Proposition \ref{prop:artinh1} and standard properties of $L$-functions and root numbers under direct sums and induced representations:

\begin{proposition}\label{prop:artinformalism}
Let $A$ be a principally polarised abelian variety over a number field $K$ and let $G$ be a finite group acting on $A$ by $K$-automorphisms.
\begin{enumerate}[leftmargin=*]
\item For $G$-representations $\tau, \tau'$,
$$
 L(A^{\tau\oplus\tau'}\!,s)=L(A^{\tau}\!,s)L(A^{\tau'}\!,s) \quad\text{ and }\quad w(A^{\tau\oplus\tau'})=w(A^{\tau}) w(A^{\tau'}).
$$
\item If $H\le G$ and $\rho$ is a representation of $H$,
$$
 L(A^{\Ind_H^G\rho},s)=L(A^{\rho},s) \quad \text{ and }\quad w(A^{\Ind_H^G\rho})=w(A^{\rho}).
$$
\item If $\tau$ is a representation of $G$ that factors through $G/N$ for $N\triangleleft G$ then 
$$
 L(A^\tau,s)=L(A_N^\tau) \quad\text{ and }\quad w(A^\tau)=w(A_N^\tau).
$$
\item If $H\le G$ then
$$
 L(A^{\Ind_H^G\mathds{1}},s)=L(A_H,s) \quad \text{ and }\quad w(A^{\Ind_H^G\mathds{1}})=w(A_H).
$$
\end{enumerate}
\end{proposition}

Finally, we record the following explicit formula for $w(A^\tau/\cK)$ in the case when $\cK$ is archimedean or $A/\cK$ is semistable.

\begin{proposition}\label{prop:rootnumber}
Let $A$ be a principally polarised abelian variety over a local field $\cK$, $G$ a finite group acting on $A$ by $\cK$-automorphisms, and $\tau$ an orthogonal $G$-representation. 
\begin{enumerate}
\item If $\cK$ is archimedean, then 
$$
 w(A^\tau/\cK) = (-1)^{\frac{1}{2}\dim H^1_\ell(A^\tau)}.
$$
\item If $\cK$ is non-archimedean and $A/\cK$ is semistable, then
$$
 w(A^\tau/\cK) = (-1)^{\langle \tau, (\dcg_A\otimes\C)^{\Frob_\cK} \rangle}.
$$
In particular,  $w(A^\tau/\cK)=1$ if $A$ has good reduction.
\end{enumerate}
\end{proposition}

\begin{proof}
(1) This follows from Definition \ref{def:twistedroot} and \cite{DeligneValeurs}
$\mathparagraph$5.3.

(2)
By Proposition \ref{cor:twistedtrace}(3),
$$
 H^1(A^\tau)/H^1(A^\tau)^{I_\cK}\simeq \Hom_G(\tau, \dcg_A\otimes\overline\Q_\ell)\otimes\chicyc^{-1}.
$$
By \cite{TateNTB} (4.2.4), the root number of an $\ell$-adic representation $V$ is related to that of its semisimplification 
$V_{ss}$ by $w(V/\cK)=w(V_{ss}/\cK) \frac{\sgn \det(-\Frob^{-1}_\cK|V_{ss}^{I_\cK})}{\sgn \det(-\Frob^{-1}_\cK|V^{I_\cK})} $.
Applying this to $H^1_\ell(A^\tau)$,
$$
 w(A^\tau/\cK) = w(H^1(A^\tau)_{ss}/\cK)
  \cdot \sgn \det(-\Frob^{-1}_\cK|  \Hom_G(\tau, \dcg_A\otimes\overline\Q_\ell)\otimes\chicyc^{-1} )  ).
$$
Both $H^1(A^\tau)^{I_\cK}$ and $H^1(A^\tau)/H^1(A^\tau)^{I_\cK}$ are unramified, hence so is $H^1(A^\tau)_{ss}$.
By \cite{TateNTB} (3.2.6.1), $ w(H^1(A^\tau)_{ss}/\cK)=1$ 
(here we choose the additive character $\psi$ to have $n_\psi=0$; by Lemma \ref{lem:windep} 
$w(A^\tau/\cK)$ is independent of this choice).
Thus
$$
 w(A^\tau/\cK) =  \det(-\Frob^{-1}_\cK|  \Hom_G(\tau, \dcg_A\otimes\overline\Q_\ell)).  
$$
As $\tau$ is orthogonal and $\dcg_A$ is a lattice, both can be realised over $\R$, and hence the eigenvalues of $\Frob_\cK^{-1}$ on $\Hom_G(\tau, \dcg_A\otimes\overline{\Q}_\ell)$ must be real or come in complex conjugate pairs. As Frobenius acts by an element of finite order on $\dcg_A$ and trivially on $\tau$, we deduce that
$$
 w(A^\tau/\cK) = 
  (-1)^{\dim \Hom_G(\tau,\dcg_A\otimes_\Z\overline\Q_\ell)^{\Frob_\cK}} =
  (-1)^{\langle \tau, (\dcg_A\otimes_\Z\C)^{\Frob_\cK} \rangle},
$$
as required.
\end{proof}

\subsection{Arithmetic conjectures}

We can now explain our analogues of the Shafarevich--Tate conjecture, the Birch--Swinnerton-Dyer conjectural rank formula, the parity conjecture and the $p$-parity conjecture for $A^\tau$.

\begin{conjecture}\label{conj:wishful}
Let $A$ be a principally polarised abelian variety over a number field $K$, let $G$ be a finite group acting on $A$ by $K$-automorphisms, and let $\tau$ be a representation of $G$. Then
\begin{enumerate}
\item For every prime $p$, $\langle \tau, \mathcal{X}_p(A/K)_\C\rangle = \langle \tau,  A(K)_\C \rangle$;
\item $\ord_{s=1}L(A^\tau,s) = \langle \tau,  A(K)_\C \rangle$;
\item If $\tau$ is self-dual, then $w(A^\tau)=(-1)^{\langle \tau,  A(K)_\C \rangle}$;
\item If $\tau$ is self-dual, then $w(A^\tau)=(-1)^{\langle \tau, \mathcal{X}_p (A/K)_\C \rangle}$.
\end{enumerate}
\end{conjecture}

\begin{theorem}\label{thm:wishful}
Let $A$ be a principally polarised abelian variety over a number field $K$ and let $G$ a finite group acting on $A$ by $K$-automorphisms. 
\begin{enumerate}[leftmargin=*]
\item Conjecture \ref{conj:wishful} (1) follows from the Shafarevich--Tate conjecture for $A/K$; 
\item Conjecture \ref{conj:wishful} (2) follows from Conjecture \ref{conj:twistedlfunctions} (1,3) for $L(A^\tau,s)$ for all representations $\tau$ of $G$ and from the Birch--Swinnerton-Dyer conjecture for $\rk A_H$ for all $H\le G$;
\item Conjecture \ref{conj:wishful} (3) follows from Conjecture \ref{conj:wishful} (2) and Conjecture \ref{conj:twistedlfunctions} (2);
\item Conjecture \ref{conj:wishful} (4) follows from Conjecture \ref{conj:wishful} (1) and (3).
\end{enumerate}
\end{theorem}

\begin{proof}
(1) 
The inclusion $A(K)\otimes_\Z\Q_p\to\mathcal{X}_p(A/K)^*$ is $G$-equivariant, by functorialilty of the connecting maps in Galois cohomology. Since $\mathcal{X}_p(A/K)$ is self-dual by Theorem \ref{thm:introgaloisdescent}(5), if $\sha(A)[p^\infty]$ is finite, then
$\mathcal{X}_p(A/K)_\C\simeq A(K)_\C$ as a $G$-module.

(2) If $\tau=\Ind_H^G\mathds{1}$ for some $H\le G$, then
$$
\ord_{s=1}L(A^\tau,s) = \ord_{s=1}L(A_H,s) = \rk A_H = \dim A(K)_\C^{H}= \langle \tau,  A(K)_\C \rangle,
$$
by Proposition \ref{prop:artinformalism}(iv), the Birch--Swinnerton-Dyer rank formula for $A_H$, Remark \ref{rmk:agGD} and Frobenius reciprocity.

If $\tau$ has rational character, then one can write $\tau^{\oplus n}\oplus\bigoplus_i \Ind_{H_i}^G\mathds{1} = \bigoplus_j \Ind_{H_j'}^G\mathds{1}$ for some $n\ge 1$ and subgroups $H_i, H_j' \le G$, as the Burnside ring has finite index in the rational representation ring (see e.g. \cite[Theorem 2.1.3]{SN94}). The result then follows from the previous case and multiplicativity of $L$-functions (Proposition  \ref{prop:artinformalism}(i)).

Finally, for general $\tau$ let $\rho=\oplus_{\alpha\in\Gal(\Q(\tau)/\Q)}\tau^{\alpha}$. As $A(K)_\C$ is a rational representation, $\langle \tau, A(K)_\C\rangle = \langle \tau^\alpha, A(K)_\C\rangle$ for all $\alpha\in\Gal(\Q(\tau)/\Q)$. The result now follows from Conjecture~\ref{conj:twistedlfunctions} 
and the previous case applied to $\rho$, which has rational character:
$$
\ord_{s=1}L(A^\tau,s) = \frac{\ord_{s=1} L(A^\rho,s)}{|\Gal(\Q(\tau)/\Q)|}  = \frac{\langle\rho, A(K)_\C \rangle}{|\Gal(\Q(\tau)/\Q)|} = \langle\tau, A(K)_\C \rangle.
$$

(3) Clear, since, by the functional equation, the parity of $\ord_{s=1}L(A^\tau,s)$ is determined by $w(A^\tau)$ whenever $\tau=\tau^*$.

(4) Clear.
\end{proof}

\section{Pseudo Brauer relations and regulator constants}\label{Sec_rep_theory}
\label{sec_rep_theory}

We now define pseudo Brauer relations and their regulator constants, extending the notion of regulator constants for Brauer relations considered in \cite{tamroot}.  Most key properties are retained in this expanded framework. Much of \S \ref{subsec_pseudo_brauer_relations}--\S\ref{subsec_tau_theta_p} will be familiar to readers experienced with these concepts.

Throughout this section, $\L$ is a field of characteristic $0$ with a fixed embedding $\L \hookrightarrow \mathbb{C}$, $G$ is a finite group, $\mathcal{H}$ is a set of representatives of the subgroups of $G$ up to conjugacy  and $\langle,\rangle$ is the standard inner product on characters. All representations are assumed finite dimensional.

\begin{subsection}{Pseudo Brauer relations} \label{subsec_pseudo_brauer_relations}

\begin{definition} \label{def:pseudo Brauer_relations}  
Let $\mathcal{V}$ be an $L[G]$-representation. 
An element $\Theta = \sum_{i} H_{i} - \sum_{j} H_{j}' \in \mathbb{Z}[\mathcal{H}]$ is a \textit{pseudo Brauer relation relative to $\mathcal{V}$} if there are $\mathbb{C}[G]$-representations $\rho_{1}$ and $\rho_{2}$, satisfying $\langle \rho_{1},  \mathcal{V}  \rangle  = \langle \rho_{2} , \mathcal{V} \rangle  =0$, such that
\begin{equation*}\label{eq:pseudo_brauer_equation_defining}
 \rho_1\oplus \bigoplus_{i} \mathbb{C}[G/H_{i}] \iso \rho_2\oplus \bigoplus_{j} \mathbb{C}[G/H_{j}'] .
 \end{equation*}
\end{definition}
\begin{remark} \label{remark_Brauer_relations} When $\mathcal{V}=L[G]$, we necessarily have $\rho_{1}=\rho_{2}=0$. In this case, Definition \ref{def:pseudo Brauer_relations} coincides with the existing notion of a Brauer relation. 
\end{remark}

\begin{remark}
The choice of representatives in $\mathcal{H}$ will be immaterial in practice (see Remark \ref{remark:representatives_for_vectorspaces} for a thorough discussion). When specific choices are required (for instance in \S \ref{Sec_ellcurves}), these will be explicitly stated.
\end{remark}

The set of all pseudo Brauer relations relative to  $\mathcal{V}$ forms an abelian subgroup of $\mathbb{Z}[\mathcal{H}]$. The following result describes the rank of this subgroup.

\begin{proposition}  
Let $\mathcal{V}$ be an $L[G]$ representation and let $\textup{Irr}_{\mathbb{Q}}(G)$ be the set of isomorphism classes of irreducible representations of $G$ over $\mathbb{Q}$. Then,$$\rk_\Z\textup{PBR}(\mathcal{V}) = \#\{\textup{conj. classes of non-cyclic }H\leq G \}+ \#\{\rho \in \textup{Irr}_{\mathbb{Q}}(G) : \langle \rho, {\mathcal{V}} \rangle =0\},$$
where $\textup{PBR}(\mathcal{V})\subseteq \mathbb{Z}[\mathcal{H}]$ denotes the subgroup of pseudo Brauer relations relative to $\mathcal{V}$.
 \end{proposition}
 
 \begin{proof}
Denote by $BR\subseteq \mathbb{Z}[\mathcal{H}]$ the subgroup of Brauer relations. It's well known that the rank of $BR$ is equal to the number of conjugacy classes of non-cyclic subgroups of $G$. Indeed, denoting by $R(G)$ the rational representation ring, we have a natural linear map $\alpha:\mathbb{Q}[\mathcal{H}]\rightarrow R(G)\otimes \mathbb{Q}$  sending $\sum_{i}n_iH_i$ to $\sum_i n_i \textup{Ind}_{H_{i}}^{G} \mathds{1}$. The kernel of $\alpha$ is $BR\otimes \mathbb{Q}$, the dimension of $R(G)\otimes \mathbb{Q}$ is equal to the number of conjugacy classes of cyclic subgroups of $G$ by \cite[\S 13.1, Cor. 1]{serre_scott_2014}, and $\alpha$ is surjective by the induction theorem \cite[\S 13.1, Theorem 30]{serre_scott_2014}.

The restriction of $\alpha$ to $PBR(\mathcal{V}) \otimes \mathbb{Q}$ is readily seen to have image contained in the subspace of $R(G)\otimes \mathbb{Q}$ spanned by irreducible rational representations $\rho$ with $\left \langle \rho,\mathcal{V}\right \rangle=0$. To complete the proof we wish to show that, conversely, any such $\rho$ lies in the image of $\alpha$. This, again, is a consequence of the induction theorem \cite[\S 13.1, Theorem 30]{serre_scott_2014}. 
\end{proof}

\subsection{Regulator constants}
\begin{notation}\label{pairing_1_2_notat}
 Let $\Theta = \sum_{i}H_{i} - \sum_{j} H_{j}'\in \mathbb{Z}[\mathcal{H}]$ be a pseudo Brauer relation relative to a self-dual $\L[G]$-representation $\mathcal{V}$. Given a non-degenerate, $G$-invariant,  $\L$-bilinear pairing   $\Langle,\Rangle$  on  $\mathcal{V}$,   we denote by $\langle ,\rangle_{1}$ the  pairing
\[\langle ,\rangle_{1}=\bigoplus_i \frac{1}{|H_i|}\Langle,\Rangle\quad \textup{on the vector space}\quad\bigoplus_{i} \mathcal{V}^{H_{i}},\]
and define the pairing $\left \langle,\right \rangle_2$ on $\bigoplus_{j} \mathcal{V}^{H_{j}'}$ similarly.  Given a basis $\mathcal{B}=\{v_i\}$  for $\bigoplus_{i} \mathcal{V}^{H_{i}}$, we  denote by $\langle \mathcal{B},\mathcal{B}\rangle_1$ the matrix  with $(i,j)$th entry $\langle v_{i},v_{j} \rangle_1 $, and define  $\langle \mathcal{B}',\mathcal{B}'\rangle_2$ for a basis $\mathcal{B}'$ of  $\bigoplus_{j} \mathcal{V}^{H_{j}'}$ similarly. By \cite[Lemma 2.15]{tamroot}, both $\langle ,\rangle_{1}$ and $\langle ,\rangle_{2}$ are non-degenerate.
\end{notation}

\begin{definition} \label{def:regulator_constants_for_pseudo_brauer}
Let $\Theta=\sum_{i} H_{i} - \sum_{j} H_{j}' \in \mathbb{Z}[\mathcal{H}]$  be a pseudo Brauer relation relative to a self-dual $\L[G]$-representation $\mathcal{V}$, and let $\Langle,\Rangle$ be a non-degenerate, $G$-invariant,    $\L$-bilinear pairing on   $\mathcal{V}$ taking values in some field extension $\L'$ of $\L$.  Given bases $\mathcal{B}$ for $\bigoplus_{i} \mathcal{V}^{H_{i}}$ and $\mathcal{B}'$ for $\bigoplus_{j} \mathcal{V}^{H_{j}'}$, we define
\[ \mathcal{C}^{\mathcal{B},\mathcal{B}'}_{\Theta}(\mathcal{V}) =  \frac{ \text{det}\langle \mathcal{B},\mathcal{B} \rangle_{1}}{\text{det} \langle \mathcal{B}',\mathcal{B}' \rangle_{2}}\in \L'^\times. \] 

 We then define the \textit{regulator constant of $\mathcal{V}$ relative to $\Theta$}, denoted $\mathcal{C}_{\Theta}(\mathcal{V}) $, to be the class of  $\mathcal{C}^{\mathcal{B},\mathcal{B}'}_{\Theta}(\mathcal{V}) $ in $\L'^{\times}/ \L^{\times 2}$ for any choice of bases $\mathcal{B}$, $\mathcal{B}'$ (the result being independent of this choice).
\end{definition}

Many properties of regulator constants associated to Brauer relations \cite[\S 2.ii]{tamroot} continue to hold for the  pseudo Brauer relations of  Definition \ref{def:regulator_constants_for_pseudo_brauer}. Specifically, we have the following:

\begin{theorem}
\label{Theorem: Independence of pairing}
Let $\L$ be a field of characteristic $0$, $G$ be a finite group and $\mathcal{V}, \mathcal{V}_1, \mathcal{V}_2$ be finite dimensional self-dual $\L[G]$-representations. Then,
\begin{enumerate}
\item  given a pseudo Brauer relation  $\Theta$ relative to $\mathcal{V}$, $\mathcal{C}_{\Theta}(\mathcal{V})$ is independent of the choice of pairing $\Langle,\Rangle$, and  takes values in  $\L^{\times}/\L^{\times 2} $,
\item given pseudo Brauer relations $\Theta_1$ and $\Theta_2$ relative to $\mathcal{V}$, we have
\[  \mathcal{C}_{\Theta_1+\Theta_2}(\mathcal{V}) = \mathcal{C}_{\Theta_1}(\mathcal{V})\, \mathcal{C}_{\Theta_2}(\mathcal{V}),\]
\item  if $\Theta$ is a pseudo Brauer relation relative to both $\mathcal{V}_1$ and $\mathcal{V}_2$, then 
\begin{equation*}
\mathcal{C}_{\Theta}( \mathcal{V}_{1}\oplus \mathcal{V}_{2}) =
 \mathcal{C}_{\Theta}(\mathcal{V}_{1}) \mathcal{C}_{\Theta}(\mathcal{V}_{2}).
 \end{equation*}
  In particular, if  $\mathcal{V} \iso \bigoplus_{i} \mathcal{V}_{i}^{n_i}$ is a decomposition into self-dual $L[G]$-representations, then
\[ \mathcal{C}_{\Theta}(\mathcal{V}) = \prod_{i} \mathcal{C}_{\Theta}(\mathcal{V}_{i})^{n_i}. \]
\end{enumerate} 
\end{theorem}

We will prove Theorem \ref{Theorem: Independence of pairing} after introducing  an alternative description of regulator constants in \S \ref{sec:alt_desc}. We note  that an analogue of (1) also holds for $\mathcal C_{\Theta}^{\mathcal{B},\mathcal{B}'}(\mathcal{V})$, see Remark \ref{rem:reg_const_bases_indep}.

\begin{remark}
\label{remark_descend_representations}
We note that $\mathcal{C}_{\Theta}$ is compatible with extension of scalars: if $K/L$ is a field extension, then $\mathcal{C}_\Theta(\mathcal{V}) = \mathcal C_\Theta(\mathcal{V} \otimes_L K)$ in $K^\times/K^{\times 2}$. On the other hand, if $\mathcal{V}$ descends to a $K[G]$-representation $\mathcal{W}$ for a subfield $K\subseteq L$, then $\mathcal{W}$ is unique up to $K[G]$-isomorphism. In particular, we can associate to $\mathcal{V}$  a well-defined regulator constant   $\mathcal{C}_{\Theta}(\mathcal{W}) \in K^\times/K^{\times 2}$. We will often omit $\mathcal{W}$ from the notation  and simply view $\mathcal{C}_\Theta(\mathcal{V})$ as an element of $K^\times/K^{\times 2}$ without comment. 
\end{remark}

Theorem \ref{Theorem: Independence of pairing} represents a generalisation of \cite[Theorem 2.17 \& Corollary 2.18]{tamroot}. In addition, the following result generalises \cite[Corollary 2.25 \& Lemma 2.26]{tamroot}.

\begin{lemma} \label{lem:symplectic_rep_regulator_constant}
Let $\Theta = \sum_{i} H_{i} - \sum_{j} H_{j}'$ be a pseudo Brauer relation relative to  $\mathcal{V}$. If either
\begin{enumerate}
\item $\mathcal{V}$ is symplectic, or
\item $\langle \mathcal{V}, L[G/H_i] \rangle = \langle \mathcal{V}, L[G/H_j'] \rangle = 0$ for all $i, j$,
\end{enumerate}
then $\mathcal{C}_{\Theta}(\mathcal{V}) \equiv 1 \textup{ mod } L^{\times 2}$.
\end{lemma}
\begin{proof}
See \cite[Corollary 2.25, Lemma 2.26]{tamroot}.
\end{proof}
We use this to show that $\mathcal{C}_{\Theta}(\mathcal{V})$ can be taken to be a positive real number.
\begin{lemma} \label{reg_constant_is_positive}
Let $\Theta$ be a pseudo Brauer relation relative to a self-dual $\L[G]$-representation $\mathcal{V}$. Then, there exists a positive real number in the same class as $\mathcal{C}_{\Theta}(\mathcal{V})$ in $\L^{\times}/\L^{\times 2}$. (Recall that we fixed an embedding $L\hookrightarrow \mathbb{C}$ at the start of the section.)
\end{lemma}
\begin{proof}
Write $\mathcal{V}\iso \psi_1 \oplus \psi_2$ where $\psi_1$ (resp. $\psi_2$) is an orthogonal (resp. symplectic) respresentation.
Then $\mathcal{C}_{\Theta}(\mathcal{V})=\mathcal{C}_{\Theta}(\psi_1) \mathcal{C}_{\Theta}(\psi_2)= \mathcal{C}_{\Theta}(\psi_1)$ by Theorem \ref{Theorem: Independence of pairing}(3) and Lemma \ref{lem:symplectic_rep_regulator_constant}. Now $\psi_1$ is realisable over $\mathbb{R}$, say on a real vector space $\mathcal{W}$, which necessarily admits a positive definite $G$-invariant pairing (take the average, over $g\in G$, of any inner product). The associated pairings $\langle,\rangle_1, \langle,\rangle_2$ (as in Notation \ref{pairing_1_2_notat}) are then positive definite. Computing $\mathcal{C}_{\Theta}(\mathcal{W})$ with respect to these pairings we obtain a positive real number. By Remark \ref{remark_descend_representations}, this gives the result. 
\end{proof}

\subsection{The representations $\tau_{\Theta,p}$} \label{subsection_tau_theta_p}
\label{subsec_tau_theta_p}
We now introduce a special class of representations $\tau_{\Theta,p}$ in the case $L=\mathbb{Q}_{p}$, which feature throughout the paper.

\begin{definition}
\label{Def: p-adically nontrivial reps}
Let $\Theta$ be a pseudo Brauer relation relative to a self-dual $\mathbb{Q}_{p}[G]$-representation~$\mathcal{V}$. 
Consider the set $\mathcal{R}_{\mathcal{V}}$ of all self-dual $\mathbb{Q}_{p}[G]$-representations $\tau$ all of whose irreducible constituents appear in $\mathcal{V}$, that is $\tau$ for which $\langle \rho,\mathcal{V}\rangle =0$ implies $\langle \rho,\tau\rangle =0$ for all $\mathbb{Q}_{p}[G]$-representations $\rho$. 
We define $\tau_{\Theta,p}$ to be (any choice of) self-dual $\mathbb{C}[G]$-representation, all of whose complex irreducible constituents are orthogonal, that satisfies  
$$
\langle \tau_{\Theta,p},\rho\rangle  \equiv \ord_{p} \mathcal{C}_{\Theta}(\rho) \mod 2, \qquad\qquad \forall \rho\in \mathcal{R}_{\mathcal{V}}.
$$ 
(For the purposes of this paper, the choice of $\tau_{\Theta,p}$ will be immaterial.)
\end{definition} 

\begin{remark} 
A choice for $\tau_{\Theta,p}$ always exists by Lemma \ref{lem:symplectic_rep_regulator_constant}(1). For instance, let $\{\tau_i\}_i$ be the set of all self-dual and $\mathbb{Q}_p$-irreducible representations of $G$ with $\ord_{p} \cC_{\Theta}(\tau_i)$ odd. Then we can take
\[ \tau_{\Theta,p}=\bigoplus_{i} \big{(} \textup{any }\mathbb{C}\textup{-irreducible constituent of }\tau_i\big{)}.\]
In particular, when all  irreducible representations of $G$ are realisable over $\Q$, we can take
$$
  \tau_{\Theta,p} = \bigoplus_\tau (\ord_{p} \cC_{\Theta}(\tau)) \tau,
$$
where the sum ranges over all irreducible representations of $G$, and $\ord_{p} \cC_{\Theta}(\tau)$ is taken in $\{0,1\}$. 
\end{remark}

\begin{remark} If $\Theta$ is a Brauer relation (see Remark \ref{remark_Brauer_relations}), then Definition \ref{Def: p-adically nontrivial reps} coincides with \cite[Definition 2.50]{tamroot}. 
\end{remark}

\begin{example}[see \cite{tamroot} Examples 2.3, 2.4, 2.20, 2.22] \label{rem:tautable}
The following table describes generators for the group of Brauer relations in $C_2\!\times\!C_2$, $S_3$ and $D_{2p}$, along with the associated representations $\tau_{\theta,p}$ for the relevant primes $p$. For $G = S_3$ and $C_2\times C_2$, these choices of $\tau_{\Theta,p}$ appeared in Examples \ref{ex:S3part1} and \ref{ex:C2C2rankrecipe}.
$$\begin{array}{| c||c|c|c|}
\hline
G&\Theta&p&\tau_{\Theta,p}\\
\hline
C_2\times C_2&C_2^a + C_2^b + C_2^c - 2C_2\!\!\times\!\! C_2 - \{1\}&2&\mathds 1 \oplus\epsilon_a\oplus\epsilon_b\oplus\epsilon_c\\
\hline
S_3&2C_2 + C_3 - 2S_3 - \{1\}&3&\mathds 1  \oplus \epsilon \oplus \rho\\
\hline
D_{2p}&2C_2 + C_p - 2D_{2p} - \{1\}&p&\mathds 1  \oplus \epsilon \oplus \rho\\
\hline
\end{array}
$$ 
Here $C_2^a$, $C_2^b$, $C_2^c$ are the order 2 subgroups and $\epsilon_a$, $\epsilon_b$, $\epsilon_c$ the non-trivial order 2 characters of $C_2\times C_2$, $\rho$ is a (choice of) 2-dimensional irreducible representation of $S_3$ or $D_{2p}$ and $\epsilon = \det \rho$. 
\end{example}

\subsection{Alternative description of regulator constants} \label{sec:alt_desc}
The following reinterpretation of regulator constants is based on expositions, in the case of Brauer relations, given in \cite[\S 3]{bartel} and \cite[Lemma 3.2]{dokchitser_dokchitser_2009}. We begin with some notation.  

\begin{notation}\label{notat:G_mod_homs_invars}
Let $M$ be a $\mathbb{Z}[G]$-module (below we will take  $M=\mathcal{V}$ to be a self-dual $\L[G]$-representation, but this greater generality will be useful in later sections). For each subgroup $H$ of $G$ we have an isomorphism 
\begin{equation*}\label{eq:invariants_as_hom}
 \textup{Hom}_G(\mathbb{Z}[G/H],M)\overset{\sim}{\longrightarrow}M^{H} \tag{$*$}
\end{equation*}
given by evaluating homomorphisms at the trivial coset. Given subgroups $H_1,\ldots,H_n$ and $H_1',\ldots,H_{m}'$ of $G$, define $G$-sets 
$S=\bigsqcup_{i=1}^nG/H_i$ and $S'=\bigsqcup_{j=1}^{m}G/H_j' .$ 
Taking $\mathbb{Z}[S]$ and $\mathbb{Z}[S']$ to be the corresponding permutation modules, 
($*$)
induces isomorphisms 
\[\textup{Hom}_G(\mathbb{Z}[S],M)\iso \bigoplus_{i}M^{H_i}\quad \textup{ and }\quad  \textup{Hom}_G(\mathbb{Z}[S'],M)\iso \bigoplus_{j}M^{H_j'}.\]    
 Consequently, given $\Phi\in \textup{Hom}_G(\mathbb{Z}[S],\mathbb{Z}[S'])$, the  map from $\textup{Hom}_G(\mathbb{Z}[S'],M)$ to $\textup{Hom}_G(\mathbb{Z}[S], M)$ sending $f$ to $f\circ \Phi$ induces  a homomorphism
\begin{equation*} \label{eq:induced_G_hom}
 \Phi^*  : \bigoplus_{j}M^{H_j'}\rightarrow \bigoplus_{i}M^{H_i}.
\end{equation*}
The $G$-module $\mathbb{Z}[S]$ (resp. $\mathbb{Z}[S']$) is canonically self-dual, via the pairing making the elements of $S$ (resp. $S'$) an orthonormal basis. Given $\Phi\in\textup{Hom}_G(\mathbb{Z}[S],\mathbb{Z}[S'])$ we denote  by $\Phi^\vee$ the corresponding dual homomorphism $\Phi^\vee:\mathbb{Z}[S']\rightarrow \mathbb{Z}[S]$.

\end{notation}

\begin{definition} \label{Def: G-isogeny realising pseudo Brauer} 
Let $\Theta = \sum_{i} H_{i} - \sum_{j} H_{j}' \in \mathbb{Z}[\mathcal{H}]$ be a pseudo Brauer relation relative to a self-dual $\L[G]$-representation $\mathcal{V}$.
We say that a $\mathbb{Z}[G]$-module homomorphism $ \Phi:\bigoplus_{i} \mathbb{Z}[G/H_{i}] \to  \bigoplus_{j} \mathbb{Z}[G/H_{j}']$ \textit{realises}  $\Theta$ if the induced map
\[\Phi^{*}: \bigoplus_{j} \mathcal{V}^{H_{j}'} \to \bigoplus_{i} \mathcal{V}^{H_{i}}\]  
is an isomorphism. 
\end{definition}

\begin{lemma} \label{isomorphism_pseudo_brauer} Let $\Theta = \sum_{i} H_{i} - \sum_{j} H_{j}'$ be a pseudo Brauer relation relative to a self-dual $\L[G]$-representation $\mathcal{V}$. Then, there exists a $G$-module homomorphism $\Phi$ realising $\Theta$. 
\end{lemma}

\begin{proof} 
By Definition \ref{def:pseudo Brauer_relations} there are $\C G$-representations $\rho_1, \rho_2$ such that 
$\rho_1\oplus \bigoplus_{i} \mathbb{C}[G/H_{i}]$ is isomorphic to $\rho_2\oplus \bigoplus_{j} \mathbb{C}[G/H_{j}']$.
We may assume that $\rho_1$ and $\rho_2$ have no common irreducible constituents, so that, in particular, they are both realisable over $\Q$.
We can then find  free $\mathbb{Z}[G]$-modules $V_1$ and $V_2$  such that $V_{1} \otimes_{\mathbb{Z}} \mathbb{C} \iso \rho_{1}$ and $V_{2} \otimes_{\mathbb{Z}} \mathbb{C} \iso \rho_{2}$, and a $G$-module homomorphism \[\phi: V_1\oplus \bigoplus_{i} \mathbb{Z}[G/H_{i}] \to V_2\oplus \bigoplus_{j} \mathbb{Z}[G/H_{j}'] \ \] with finite kernel and cokernel. Denoting by $\iota$ and $\pi$ the inclusion/projection in/out of the permutation modules, one checks that $\Phi=\pi\circ \phi \circ \iota$ realises the pseudo Brauer relation $\Theta$. \end{proof}
\begin{remark}
If $\Theta= \sum_{i} H_{i} - \sum_{j} H_{j}'$ is a Brauer relation, then a $G$-map realising it is precisely a $G$-injection $\Phi: \bigoplus_{i} \mathbb{Z}[G/H_{i}] \to \bigoplus_{j} \mathbb{Z}[G/H_{j}']$ with finite cokernel. 
\end{remark}

\begin{notation} \label{Notation:Matrices}
Given finite dimensional $\L$-vector spaces $V, W$, with bases $\mathcal{B}_1=\{v_i\}_i,\mathcal{B}_2=\{w_j\}_j$ respectively, and given an   $\L$-linear map $T: V \to W$, we write $[T]_{\mathcal{B}_{1}}^{\mathcal{B}_{2}}$ to denote the matrix of $T$ relative to  $\mathcal{B}_{1}$  and $\mathcal{B}_{2}$. Similarly to Notation \ref{pairing_1_2_notat}, given a pairing $\Langle,\Rangle$ between $V$ and $W$,   we denote by $\Langle \mathcal{B}_1,\mathcal{B}_2\Rangle$ the matrix with $(i,j)$th entry $\Langle v_i,w_j\Rangle$.
\end{notation}

The following proposition (along with Corollary \ref{Regulator Constants and pullback basis}) gives the promised alternative description of regulator constants. Analogues for Brauer relations appear as  \cite[Theorem 3.2]{bartel}, \cite[Lemma 3.2]{dokchitser_dokchitser_2009}. 

We highlight that, while we have not yet shown that  $ \mathcal{C}_{\Theta}^{\mathcal{B},\mathcal{B}'}(\mathcal{V})$ is independent of the choice of pairing $\Langle ,\Rangle$ made in its definition, the proof of the proposition applies regardless of the choice made. As a by-product, this will prove the sought independence.

\begin{proposition}\label{comparison_determinant_regulator_constant}

Let $\Theta=\sum_{i} H_{i} - \sum_{j} H_{j}'$ be a pseudo Brauer relation relative to a self-dual $\L[G]$-representation $\mathcal{V}$.
For any $G$-module homomorphism $\Phi$ realising $\Theta$, we have 
\[    \mathcal{C}_{\Theta}^{\mathcal{B},\mathcal{B}'}(\mathcal{V}) = \frac{\textup{det} [( \Phi^{\vee})^{*}]_{\mathcal{B}}^{\mathcal{B}'}}{\textup{det}[\Phi^{*}]_{\mathcal{B}'}^{\mathcal{B}}}\] where $\mathcal{B}$, $\mathcal{B}'$ are bases for $\bigoplus_{i} \mathcal{V}^{H_{i}}$, $\bigoplus_{j} \mathcal{V}^{H_{j}'}$ respectively.
\end{proposition}

\begin{proof}
Fix a non-degenerate $G$-invariant pairing $\Langle,\Rangle$ on $\mathcal{V}$. We follow the proof of \cite[Lemma 3.2]{dokchitser_dokchitser_2009}. For a finite $G$-set $T$ we define the pairing $(,)_T$ on $\textup{Hom}_G(\mathbb{Z}[T],\mathcal{V})$ by setting
\[(f_1,f_2)_T=\frac{1}{|G|}\sum_{t\in T}\Langle f_1(t),f_2(t)\Rangle.\]
Take $S=\bigsqcup_iG/H_i$ and $S'=\bigsqcup_jG/H_j'$.  After identifying $\textup{Hom}_G(\mathbb{Z}[S],\mathcal{V})$ with $\bigoplus_{i}\mathcal{V}^{H_i}$ as in Notation \ref{notat:G_mod_homs_invars}, the pairing $(,)_S$ identifies with the pairing $\langle,\rangle_1$ of Notation \ref{pairing_1_2_notat}. Similarly, the pairing $(,)_{S'}$ identifies with the pairing $\langle,\rangle_2$. An easy computation then shows that $\Phi^*$ and $(\Phi^\vee)^*$ are adjoint for the pairings $\langle, \rangle_1$ and $\langle,\rangle_2$. We now compute
\begin{equation*}
\mathcal{C}_{\Theta}^{\mathcal{B},\mathcal{B}'}(\mathcal{V})= \frac{ \text{det}\langle \mathcal{B},\mathcal{B} \rangle_{1}}{\text{det} \langle \mathcal{B}',\mathcal{B}' \rangle_{2}}
= \frac{\text{det}\langle \mathcal{B},\Phi^*\mathcal{B}' \rangle_{1}}{\textup{det}  [\Phi^{*} ]_{\mathcal{B}'}^{\mathcal{B}}}\cdot \frac{\textup{det} [( \Phi^{\vee})^{*}]_{\mathcal{B}}^{\mathcal{B}'}}{\text{det} \langle (\Phi^\vee)^*\mathcal{B},\mathcal{B}'\rangle_{2}}. \qedhere
\end{equation*}
\end{proof}

\begin{corollary} \label{Regulator Constants and pullback basis} Let $\Theta=\sum_{i} H_{i} - \sum_{j} H_{j}'$ be a pseudo Brauer relation relative to a self-dual $\L[G]$-representation $\mathcal{V}$. Let $\Phi$ be a $G$-module homomorphism  realising $\Theta$ and $\mathcal{B}$ a basis for $\bigoplus_{i} \mathcal{V}^{H_{i}}$. 
Then
$$
\mathcal{C}_{\Theta}^{\mathcal{B}, \Phi^{\vee} \mathcal{B}}(\mathcal{V}) = \frac{1}{\textup{det} (\Phi^{\vee}   \Phi)^{*}}.
$$
In particular, $\mathcal{C}_{\Theta}^{\mathcal{B}, \Phi^{\vee} \mathcal{B}}(\mathcal{V})$ is independent of $\mathcal{B}$.  
\end{corollary}

\begin{proof} Immediate from Proposition \ref{comparison_determinant_regulator_constant}. 
\end{proof}
 \begin{proof}[Proof of Theorem \ref{Theorem: Independence of pairing}]
 Parts (2) and (3) follow readily from part (1). To prove (1), pick some $\Phi$ realising the pseudo Brauer relation $\Theta$, and pick a pairing $\Langle,\Rangle$ as in the definition of $\mathcal{C}_\Theta(\mathcal{V})$. The proof of Proposition \ref{comparison_determinant_regulator_constant} then shows that
 \[\mathcal{C}_\Theta(\mathcal{V})\equiv \textup{det} (\Phi^{\vee}   \Phi)^{*} \mod \L^{\times 2}.\]
 Since the right hand side is an element of $\L^\times$ which is independent of $\Langle,\Rangle$, the result follows. 
 \end{proof}
 
 \begin{remark} \label{rem:reg_const_bases_indep}
 By the same argument, for any bases $\mathcal{B}$ and $\mathcal{B}'$ as in Definition \ref{def:regulator_constants_for_pseudo_brauer}, the quantity $\mathcal{C}_{\Theta}^{\mathcal{B},\mathcal{B}'}(\mathcal{V})$ lies in $L^\times$, and is independent of the choice of pairing $\Langle,\Rangle$ used to define it.
 \end{remark}

 \begin{remark} \label{remark:representatives_for_vectorspaces}
The choice of representatives in $\mathcal{H}$ made at the start of this section does not affect the above results in any meaningful way. For instance, if $\Theta = \sum_{i} H_{i} - \sum_{j} H_{j}' \in \mathbb{Z}\mathcal{H}$ is a pseudo Brauer relation for $\mathcal{V}$, and we set $M_i=g_iH_ig_i^{-1}$, $M_i'=g_i'H_i'g_i'^{-1}$  for some $g_i, g_i'\in G$, then we have canonical isomorphisms   $\bigoplus_{i}\mathcal{V}^{H_i}\simeq \bigoplus_i\mathcal{V}^{M_i}$ and $\bigoplus_{i} \mathbb{Z}[G/M_i] \simeq\bigoplus_{i} \mathbb{Z}[G/H_i]$ given by $(x_i) \mapsto (g_i x_i)$ and $(y_iM_i) \mapsto (y_ig_iH_i)$ respectively. Similarly for $H_j'$. 
We conclude that the calculation of a regulator constant from Definition  \ref{def:regulator_constants_for_pseudo_brauer} and the notion of realising a pseudo Brauer relation from Definition \ref{Def: G-isogeny realising pseudo Brauer} are consistent even after a change from $\Theta$ to $\Theta'=\sum_{i}M_{i} - \sum_{j} M_{j}'$. In addition, a change from $\Theta$ to $\Theta'$ does not affect Proposition \ref{comparison_determinant_regulator_constant} or  Corollary \ref{Regulator Constants and pullback basis}.
\end{remark}
\end{subsection}

\section{Isogenies induced from pseudo Brauer relations}\label{sec:isogenies}
Let $X$ be a curve defined over a field $K$ of characteristic $0$ and let $G$ be a finite subgroup of $\Aut_K(X)$. In what follows, we will exploit a consequence of \cite[Theorem 4.14]{Etale_paper}, presented as Theorem \ref{isogeny_from_pseudo_brauer} below, which allows us to associate certain isogenies to   pseudo Brauer relations relative to the $\ell$-adic Tate module $V_{\ell}(\textup{Jac}_{X})$. 

\begin{definition}\label{def:pseudo_brauer_geometry}
Let $\ell$ denote any prime. We say that $\Theta$ is a pseudo Brauer relation for $G$ and $X$ if $\Theta$ is a pseudo Brauer relation relative to $V_{\ell}(\textup{Jac}_{X})$ in the sense of Definition \ref{def:pseudo Brauer_relations}. When there is no ambiguity, we  call these pseudo Brauer relations for $X$.
\end{definition}
\begin{remark}
 Having fixed an embedding $K\hookrightarrow \mathbb{C}$, we have  a $\mathbb{Q}_{\ell}[G]$-isomorphism \[H_1(\textup{Jac}_{X}(\mathbb{C}),\mathbb{Z}) \otimes \mathbb{Q}_{\ell} \iso V_{\ell}(\textup{Jac}_{X})\] for every prime $\ell$. Thus the notion of a pseudo Brauer relation  for $G$ and $X$ is independent of $\ell$.
\end{remark}

\begin{theorem} \label{isogeny_from_pseudo_brauer}
Let $X$ be a curve over a field $K$ of characteristic $0$, and $G$ be a finite subgroup of $\Aut_K(X)$. Let $\sum_{i} H_{i} - \sum_{j} H_{j}'$ be a pseudo Brauer relation for $X$ realised by $\Phi$ (in the sense of Definition \ref{Def: G-isogeny realising pseudo Brauer}). Then, the following hold.
\begin{enumerate}
\item  The $G$-map $\Phi$ induces a $K$-isogeny
\[f_{\Phi}: \prod_{j} \textup{Jac}_{X/H_{j}'} \to \prod_{i} \textup{Jac}_{X/H_{i}}.\]
\item If $\Phi'$ realises a pseudo Brauer relation $\sum_{j} H_{j}' - \sum_{k} {H}_{k}''$ for $X$, then the composition $\Phi'\Phi$ realises $\sum_{i} H_{i} - \sum_{k} {H}_{k}''$ and $f_{\Phi'\Phi} = f_{\Phi} f_{\Phi'}.$
\item  The dual homomorphism $\Phi^{\vee}$ (as in Notation \ref{notat:G_mod_homs_invars}) realises the pseudo Brauer relation $\sum_{j} H_{j}' - \sum_{i} H_{i}$ for $X$, and we have $f_{\Phi^{\vee}}=(f_{\Phi})^{\vee}$ where  $(f_{\Phi})^{\vee}$ denotes the dual of $f_{\Phi}$ with respect to the canonical principal polarisations.  
\end{enumerate}
\end{theorem}
\begin{proof}
(1) follows from \cite[Theorem 4.14(1,3)]{Etale_paper}. For (2), it is clear that $\Phi'\Phi$ realises the pseudo Brauer relation $\sum_{i} H_{i} - \sum_{k} {H}_{k}''$, while the equailty $f_{\Phi'\Phi} = f_{\Phi} f_{\Phi'}$ follows from \cite[Theorem 4.14(2)]{Etale_paper}. (3) follows from \cite[Theorem 4.14(1)]{Etale_paper}.
\end{proof}

The next result concerns the degree of the isogeny $f_{\Phi}$.

\begin{proposition} \label{prop:deg=regulator_constant}  Suppose that $\Omega^{1}(\JX)$ is a self-dual $G$-representation. Fix any basis $\mathcal{B}_{1}$ for $\Omega^{1}(\prod_{i} \Jac_{X/H_{i}})$ and let $\Phi^{\vee}\mathcal{B}_{1}$ be the basis for $\Omega^{1}(\prod_{j} \Jac_{X/H_{j}'})$ obtained by applying $(\Phi^{\vee})^{*}$ to the basis vectors in $\mathcal{B}_{1}.$  
\begin{enumerate}
\item We have \[\mathcal{C}_{\Theta}^{\mathcal{B}_1, \Phi^{\vee} \mathcal{B}_1}(\Omega^{1}(\JX))^{-1}= \pm{\textup{deg}(f_{\Phi})}.\]
\item If $\Omega^{1}(\JX)$ is orthogonal, then \[\mathcal{C}_{\Theta}^{\mathcal{B}_1, \Phi^{\vee} \mathcal{B}_1}(\Omega^{1}(\JX))^{-1}= {\textup{deg}(f_{\Phi})}.\]
\item For any $G$-maps  $\Phi_1$, $\Phi_2$ realising $\Theta$,  ${\textup{deg}(f_{\Phi_1})}/{\textup{deg}(f_{\Phi_2})}$ lies in $\mathbb{Q}^{\times 2}$.
\end{enumerate} 
\end{proposition}
\begin{proof}
For (1), we compute
\begin{eqnarray*}
\mathcal{C}_{\Theta}^{\mathcal{B}_1, \Phi^{\vee} \mathcal{B}_1}(\Omega^{1}(\JX))^{-2} & \overset{\substack{\textup{Cor. \ref{Regulator Constants and pullback basis} \&}\\ \textup{\cite[Lem. 5.10(1)]{Etale_paper}}}}{=}  & \textup{det}\big{(}(\Phi^{\vee}  \Phi)^{*} \ | \ V_{\ell}(\prod\textup{Jac}_{X/H_{i}}) \big{)} \\ 
&\overset{\textup{\cite[Lem. 4.28]{Etale_paper}}}{=} & \textup{det}\big{(}f_{\Phi^{\vee}  \Phi} \ | \ V_{\ell}(\prod\textup{Jac}_{X/H_{i}}) \big{)} \\ 
&\overset{\textup{\cite[Prop. 12.9]{MR861974}}}{=}& \textup{deg}(f_{\Phi^\vee\Phi})\\&\overset{\textup{Thm. \ref{isogeny_from_pseudo_brauer} (2),(3)}}{=}&\textup{deg}(f_\Phi)^2.
\end{eqnarray*} 
 For (2), the $G$-representation $\Omega^{1}(\textup{Jac}_{X})$ is realisable over $\mathbb{R}$, say on an $\mathbb{R}$-vector space $W_{\mathbb{R}}$. By evaluating $\mathcal{C}_{\Theta}^{\mathcal{B}_1, \Phi^{\vee} \mathcal{B}_1}(\Omega^{1}(\JX))$ with respect to an $\mathbb{R}$-basis $\mathcal{B}_1$ for  $\prod_{i} (W_{\mathbb{R}})^{H_i}$, Lemma \ref{reg_constant_is_positive} combined with Theorem \ref{Theorem: Independence of pairing}(1) tell us that $\mathcal{C}_{\Theta}^{\mathcal{B}_1, \Phi^{\vee} \mathcal{B}_1}(\Omega^{1}(\JX))$ is a positive real number. Since the equality in $(1)$ holds for any $\mathcal{B}_1$, we obtain the result.
  For (3), note that $\Omega^{1}(\JX)^{\oplus 2}$ is realisable over $\mathbb{Q}$ by \cite[Lemma 5.10(2)]{Etale_paper}. If follows from \cite[Theorem 32.15]{Splitting-fields-CSA} that there exist infinitely many distinct quadratic extensions $L/\mathbb{Q}$ such that $\Omega^{1}(\textup{Jac}_{X})$ is realisable over $L$. For each such, we deduce from part (1) and Theorem \ref{Theorem: Independence of pairing}(1) that 
$\pm\tfrac{\textup{deg}(f_{\Phi_1})}{\textup{deg}(f_{\Phi_2})}   \in L^{\times 2}.$
  This is only possible if $\tfrac{\textup{deg}(f_{\Phi_1})}{\textup{deg}(f_{\Phi_2})}   \in \mathbb{Q}^{\times 2}$. 
\end{proof}
\section{Local formulae for Selmer rank parities}  \label{sec:Rank_Parity} 
Let $X/K$ be a curve defined over a number field $K$, and let $G$ be a finite subgroup of $\Aut_K(X)$. The main result of this section expresses the valuation of the regulator constant of the dual $p^\infty$-Selmer group, ${\mathcal X}_p(\JX)$, in terms of   explicit local invariants. 

\begin{theorem} \label{Theorem: Local Formula}
Let $X$ be a curve over a number field $K$, and $G$ be a finite subgroup of $\Aut_K(X)$. Let $\Theta$ be a pseudo Brauer relation for $X$, and suppose that $\Omega^{1}(\JX)$ is a self-dual $K[G]$-representation. Then \[  \ord_p \mathcal{C}_{\Theta} (\mathcal{X}_{p}(\JX)) \equiv \sum_{v \textup{ place of }K} \ord_p\Lambda_{\Theta}(X/K_{v})\mod  2,\]
where the local invariant $\Lambda_{\Theta}(X/K_{v})$ is as in Definition \ref{def:lemma:local_invariants}. In particular, for any choice of $\tau_{\Theta,p}$ as in Definition \ref{Def: p-adically nontrivial reps},
$$
\langle \tau_{\Theta,p}, {\mathcal X}_p(\JX)\rangle \equiv \sum_{v \textup{ place of }K} \ord_p\Lambda_\Theta(X/K_v) \mod 2.
$$
\end{theorem}

To prove this theorem, we begin by defining the following local invariant 
$$\lambda_{\Theta,\Phi}(X/K_{v}) =\frac{\#\textup{coker} f_{\Phi}(K_{v}) }{\#\textup{ker} f_{\Phi}(K_{v})}\cdot \frac{\prod_{i}\mu_v({X/H_{i}})}{\prod_{j}\mu_v( {X/H_{j}'})},$$
where $\Phi$ is a $G$-map realising $\Theta$, $\textup{coker}f_\Phi(K_v)$ and $\textup{ker}f_\Phi(K_v)$ are the induced maps on $K_v$-points and $\mu_v$ encodes whether a curve is deficient at a place $v$ (see \S \ref{subsec_deficiency} below). We prove an analogue of the above formula (see Theorem \ref{Theorem: Local_Formula_dependent_on_differnetials}) with $\lambda_{\Theta, \Phi}$ in place of $\Lambda_{\Theta}$. Unfortunately, $\lambda_{\Theta, \Phi}$ depends on the choice of $\Phi$. To remove this dependence (see Lemma \ref{lemma_tilde_lambda}(1)), we introduce a revised invariant $$\tilde \lambda_{\Theta}(X/K_{v})= \lambda_{\Theta,\Phi}(X/K_{v})\cdot \bigl{|}\sqrt{\textup{deg}(f_{\Phi})}\bigr{|}_{v}.$$
The drawback now is that $\ord_p \tilde \lambda_{\Theta} \equiv 0, \,\frac{1}{2},\, 1\textup{ or }\, \frac{3}{2} \bmod 2$. We finally fix this (see Theorem \ref{Thm:main_local_inv_thm}(1)) by defining
$${\Lambda}_{\Theta}(X/K_{v}) =\tilde \lambda_{\Theta}(X/K_{v}) \cdot \Big{|} \sqrt{\ct(\Omega^{1}(\JX))}\Big{|}_{v},$$
 where $\ct\in \mathbb{N}$ denotes the square-free part of $\deg(f_{\Phi})$, calculated with respect to any $\Phi$. We caution that whilst $\tilde\lambda_\Theta$ is multiplicative in $\Theta$, in general $\Lambda_\Theta$ is not.

\begin{remark}
When $K_v/\Q_\ell$ is a finite extension and $p\neq 2 \textup{ or } \ell$, we will additionally show (see Theorem \ref{Theorem:Properties of local invariant}(3)) that
$$\textup{ord}_{p} \, \Lambda_{\Theta}(X/K_v) = \textup{ord}_{p}\, \frac{\prod_{i} c_{v}(\textup{Jac}_{X/H_{i}})}{\prod_{j} c_{v}(\textup{Jac}_{X/H_{j}'})}$$ where $c_v(A)$ denotes the Tamagawa number of an abelian variety $A/K_v$. 
Lemma \ref{local_term_neron_basis} details the analogue of this formula for an arbitrary prime $p$.
\end{remark}

\begin{remark}By Theorem \ref{thm:introgaloisdescent}(5),(7), $\mathcal{X}_{p}(\textup{Jac}_{X})$ is a self-dual $G$-representation for all $p$, and any pseudo Brauer relation for $X$ is a pseudo Brauer relation relative to $\mathcal{X}_{p}(\JX)$.  Thus, the left-hand side of the formula in Theorem \ref{Theorem: Local Formula} is well-defined.  \end{remark}

\subsection{Deficiency} \label{subsec_deficiency} Here we remind the reader of the notion of deficiency for curves over local fields. It is well-known that when the curve arises via base-change from a number field, this concept controls the size, modulo rational squares, of the $2$-primary part of the Shafarevich--Tate group of its Jacobian (see \cite[\S 8]{poonen1999cassels} for geometrically connected curves, and  \cite[\S 5.5]{Etale_paper} for an  extension to   the curves of Convention \ref{convention_curves}). 

\begin{definition}[\cite{Etale_paper}, Definition 5.13] \label{def:deficiency}  Let $\mathcal{K}$ be a local field of characteristic 0.
 
 A geometrically connected curve $X/\mathcal K$ of genus $g$ is called {\em deficient} if it has no $\mathcal{K}$-rational divisor of degree $g-1$. For such $X$, we define 
 \[\mu_\mathcal{K}(X)= \begin{cases} 
2 & \text{if $X$ is deficient,} \\
1 & \textup{otherwise}.
\end{cases} \]

Suppose that $X/\mathcal K$ is connected (but not necessarily geometrically connected). We write $\mathcal{L}$ for the minimal field of definition of one of its geometric components,  $Y$ say,  and define $\mu_{\mathcal{K}}(X)=\mu_\mathcal{L}(Y)$.

Finally, writing $X=\bigsqcup_iX_i$ as a disjoint union of connected components, we define 
$\mu_\mathcal{K}(X)=\prod_{i}\mu_{\mathcal K}(X_{i}).$

If the field $\mathcal{K}$ is clear from context, we will often omit it from the notation.
Further, when $\mathcal{K} = K_v$ is the completion of a number field $K$ at a place $v$, we write $\mu_v(X)$ in place of $\mu_{K_v}(X)$. 
\end{definition}

\subsection{A local formula in $\lambda_{\Theta, \Phi}(X/\mathcal{K})$}\label{sec:littlelambda}
Here we prove the analogue of Theorem \ref{Theorem: Local Formula} obtained by replacing $\Lambda_\Theta$ with the local invariant $\lambda_{\Theta, \Phi}$  defined below. 

\begin{notation}\label{not:C} For an abelian variety $A$ over a local field $\mathcal K$ of characteristic $0$, and a choice of a non-zero exterior form $\omega$ on $A$, we write
$$C(A, \omega) = \begin{cases}c(A) \cdot |\omega/\omega^{0}| & \textup{when $\mathcal K/\Q_p$ is finite},\\
\int_{A(\mathcal K)} |\omega| &\textup{when ${\mathcal K}=\mathbb{R}$},\\ 2^{\text{dim}(A)}\int_{A(\mathcal K)} |\omega \wedge \bar{\omega}|& \textup{when ${\mathcal K}=\mathbb{C}$}. \end{cases}$$
Here $\omega/\omega^{0}\in \mathcal K^{\times}$ is such that $\omega = (\omega/\omega^{0})\cdot \omega^{0}$, where $\omega^{0}$ is a N\'eron exterior form on $A$.
\end{notation}

\begin{definition} \label{def:local_invariants} 
Let $\mathcal{K}$ be a local field of characteristic $0$, $X/\mathcal{K}$ be a curve, $G$ be a finite subgroup of $\Aut_\K(X)$ and $\Theta=\sum_{i} H_{i} - \sum_{j} H_{j}'$  be a pseudo Brauer relation for $X$. Fix bases $\mathcal{B}_{1}$, $\mathcal{B}_{2}$ for $\Omega^{1}(\prod_{i}  \Jac_{X/H_i})$, $\Omega^{1}(\prod_{j} \Jac_{X/H_j'})$ and write $\omega(\mathcal{B}_{1})$, $\omega(\mathcal{B}_{2})$ for the exterior forms given by the wedge product of the elements in $\mathcal{B}_{1}$, $\mathcal{B}_{2}$ respectively. We define
\[ \lambda_{\Theta}^{\mathcal{B}_{1},\mathcal{B}_{2}}(X/\mathcal{K})= \frac{C( \prod_{i} \Jac_{X/H_{i}},\omega(\mathcal{B}_{1})) \ }{C( \prod_{j} \Jac_{X/H_{j}'},\omega(\mathcal{B}_{2})) } \cdot \frac{\prod_{i}\mu({X/H_{i}})}{\prod_{j}\mu( {X/H_{j}'})} \] 
where $C$ is given in Notation \ref{not:C} and $\mu$ is as in Definition \ref{def:deficiency}.

Given a $G$-map $\Phi$ realising $\Theta$, write $\Phi^{\vee}\mathcal{B}_{1}$ for the basis obtained by applying $(\Phi^{\vee})^{*}$ to the elements of $\mathcal{B}_{1}$ (cf.  Notation \ref{notat:G_mod_homs_invars}). We additionally define
$$\lambda_{\Theta, \Phi}(X/\mathcal K) = \lambda_{\Theta}^{\mathcal{B}_{1}, \Phi^{\vee}\mathcal{B}_{1}}(X/\mathcal{K}).$$ 

We write $f_{\Phi}(\mathcal{K})$ for the map on $\mathcal{K}$-points induced by the isogeny $f_{\Phi}$. 
\end{definition}

\begin{lemma} \label{pullback-basis-lemma}   We have
\[ \lambda_{\Theta, \Phi}(X/\mathcal K)  = \frac{\#\textup{coker} f_{\Phi}(\mathcal{K}) }{\#\textup{ker} f_{\Phi}(\mathcal{K})}\cdot \frac{\prod_{i}\mu({X/H_{i}})}{\prod_{j}\mu( {X/H_{j}'})}.\]
In particular, $\lambda_{\Theta, \Phi}(X/\mathcal K)\in \mathbb{Q}^\times$ is independent of the choice of $\mathcal B_1$.
\end{lemma}
\begin{proof}
By applying \cite[Remark 4.29]{Etale_paper} with $F=\Omega^1(-)$, and taking exterior powers, we deduce that $f_{\Phi}^{*}\omega(\mathcal{B}_{1}) = \omega(\Phi^{\vee}\mathcal{B}_{1})$.
As in the proof of  \cite[Theorem 7.3]{MR2261462}, we have    \[ \frac{C(\prod_{i} \Jac_{X/H_{i}}, \omega(\mathcal{B}_{1}))}{C(\prod_{j} \Jac_{X/{H_{j}'}}, \omega(\Phi^{\vee}\mathcal{B}_{1}))}  = \frac{\#\textup{coker} f_{\Phi}(\mathcal{K}) }{\#\textup{ker} f_{\Phi}(\mathcal{K})},\] giving the required result. \qedhere
\end{proof}

\begin{theorem}  \label{Theorem: Local_Formula_dependent_on_differnetials}
Let $X$ be a curve over a number field $K$ and $G$ be a finite subgroup of $\Aut_K(X)$. Let $\Theta$ be a pseudo Brauer relation for $X$, $\Phi$ a $G$-map realising $\Theta$ and $p$ a prime. Then, 
 \[  \ord_p \mathcal{C}_{\Theta}(\mathcal{X}_{p}(\JX)) \equiv \sum_{v \textup{ place of }K} \ord_p \lambda_{\Theta,\Phi}(X/K_{v})  \mod 2.\]
\end{theorem} 

\begin{proof}
For a $K$-isogeny $f: A\to B$, we write
\begin{equation*}  \begin{split} Q(f) = \#\text{coker}(f:A(K)/A(K)_{\text{tors}} \to B(K)/B(K)_{\text{tors}} )
\cdot \# \text{ker}(f: \Sha(A)_{\text{div}} \to \Sha(B)_{\text{div}}),\end{split} \end{equation*}
where $\Sha_{\textup{div}}$ denotes the divisible part of $\Sha$. Thus, if $f$ is a self-isogeny of $A$, then $\ord_{p}Q(f) = \ord_{p} \textup{det}( f | \mathcal{X}_{p}(A))$. For each place $v$ of $K$, let $\mathcal{B}_{v}$ be a basis for $\Omega^{1}(\prod_{i}\textup{Jac}_{X/{H_{i}}}/K_{v})$ and write 
$\omega(\mathcal{B}_{v})$ for the exterior form on $\prod_{i} \textup{Jac}_{X/H_{i}}/K_{v}$ obtained by taking the wedge product of elements in $\mathcal{B}_{v}$, and similarly for $\omega(\Phi^{\vee}\mathcal{B}_{v})$. As in the proof of Lemma \ref{pullback-basis-lemma}, $\omega(\Phi^{\vee}\mathcal{B}_{v}) = f_{\Phi}^{*}\omega(\mathcal{B}_{v})$ and so

\begin{eqnarray*}
\frac{{Q(f_{\Phi})}}{{Q(f_{\Phi}^{\vee})}} &\overset{\textup{\cite[Thm. 4.3]{MR2680426}}}{\equiv}& \frac{\prod_{v}C_{v}(\prod_{i} \textup{Jac}_{X/{H_{i}}}, \omega(\mathcal{B}_{v}))}{\prod_{v}C_{v}(\prod_{j} \textup{Jac}_{X/{H_{j}'}}, f_{\Phi}^{*}\omega(\mathcal{B}_{v}))} \cdot \frac{\prod_{i}\#\Sha_{0}(\Jac_{X/H_{i}})[2^{\infty}]}{\prod_{j}\#\Sha_{0}(\Jac_{X/H_{j}'})[2^{\infty}]}\\
&\overset{\text{\cite[Prop. 5.14]{Etale_paper}}}{\equiv}& \prod_{v} \lambda_{\Theta,\Phi}(X/K_{v}) \mod \mathbb{Q}^{\times2}\end{eqnarray*} where $C_v(A,\omega)$ denotes $C(A,\omega)$ for $A/K_v$. On the other hand, by parts (2) and (3) of Theorem \ref{isogeny_from_pseudo_brauer}, we have \[\frac{{Q(f_{\Phi})}}{{Q(f_{\Phi}^{\vee})}} ~\equiv~ {{Q(f_{\Phi})}}{{Q(f_{\Phi}^{\vee})}}~=~Q(f_{\Phi} f_{\Phi}^{\vee})~=~ Q(f_{\Phi} f_{\Phi^{\vee}})~=~Q(f_{\Phi^{\vee}  \Phi}) \mod \mathbb{Q}^{\times2}.\]  

\noindent Putting everything together, we see that
\begin{eqnarray*}
\ord_{p}\displaystyle\prod_{v} \lambda_{\Theta,\Phi}(X/K_{v}) 
&\equiv& \ord_{p} \det\big{(}f_{\Phi^{\vee} \Phi} \ | \ \mathcal{X}_{p}\big{(}\prod_{i} \Jac_{X/H_{i}}\big{)}\big{)}\\[-0.5em]
&\overset{\textup{\cite[Lem. 4.28]{Etale_paper}}}{\equiv} &\ord_{p} \det\big{(}(\Phi^{\vee} \Phi)^{*} \ |\ \mathcal{X}_{p}\big{(}\prod_{i}\Jac_{X/H_{i}} \big{)}\big{)}\\[-0.5em]
&\overset{\textup{Cor. \ref{Regulator Constants and pullback basis}}}{\equiv}&\ord_{p}\cC_{\Theta}(\mathcal{X}_{p}(\JX)) \mod 2
\end{eqnarray*} 
as claimed.
\end{proof}

\subsection{The local invariant $\tilde{\lambda}_{\Theta}(X/\mathcal{K})$}\label{subsec_tilde_lambda}
Assuming that $\Omega^{1}(\textup{Jac}_{X})$ is a self-dual $\mathcal{K}[G]$-representation, we now define a revised version of the local invariant $\lambda_{\Theta,\Phi}$.

\begin{definition} \label{temporary_lambda}
Let $\mathcal{K}$ be a local field of characteristic $0$, $X/\mathcal{K}$ be a curve, $G$ be a finite subgroup of $\Aut_\K(X)$ and $\Theta=\sum_{i} H_{i} - \sum_{j} H_{j}'$  be a pseudo Brauer relation for $X$.
Fix bases $\mathcal{B}_{1}$, $\mathcal{B}_{2}$ for $\Omega^{1}(\prod_{i}  \Jac_{X/H_i})$, $\Omega^{1}(\prod_{j} \Jac_{X/H_j'})$. Assuming that $\Omega^{1}(\textup{Jac}_{X})$ is self-dual as a $\mathcal{K}[G]$-representation, we define
\begin{equation*} \label{Temporary_constant}
\tilde{\lambda}_{\Theta}(X/\mathcal{K})=\frac{\lambda_{\Theta}^{\mathcal{B}_{1}, \mathcal{B}_{2}}(X/\mathcal{K})}{\Big|\sqrt{\mathcal{C}_{\Theta}^{\mathcal{B}_1,\mathcal{B}_{2}}(\Omega^{1}(\JX)) } \Big|_{\mathcal{K}}},
\end{equation*}
where $\mathcal{C}_{\Theta}^{\mathcal{B}_{1},\mathcal{B}_{2}}$ is the regulator constant evaluated with respect to $\mathcal{B}_{1}$, $\mathcal{B}_{2}$ as in Definition \ref{def:regulator_constants_for_pseudo_brauer}.
\end{definition}

\begin{remark}
By Theorem \ref{thm:introgaloisdescent}(7), any pseudo Brauer relation $\Theta$ for $X$ is also a pseudo Brauer relation for $\Omega^{1}(\JX)$. Thus, $\mathcal{C}_{\Theta}^{\mathcal{B}_1,\mathcal{B}_{2}}(\Omega^{1}(\JX)) $ is well-defined. 
\end{remark}

\begin{lemma} \label{lemma_tilde_lambda}  Suppose that $\Omega^{1}(\JX)$ is a self-dual $\mathcal{K}[G]$-representation.  
\begin{enumerate}
\item  $\tilde{\lambda}_{\Theta}(X/\mathcal{K})$ is independent of the bases $\mathcal{B}_{1}$, $\mathcal{B}_{2}$.
\item $\tilde{\lambda}_{\Theta}(X/\mathcal{K})$ is independent of how $\Theta$ is expressed as a formal linear combination of conjugacy classes of subgroups of $G$. That is, for all subgroups $H$ of $G$, $$\tilde{\lambda}_{\Theta}(X/\mathcal{K}) = \tilde{\lambda}_{\Theta+(H-H)}(X/\mathcal{K}).$$
\item Given pseudo Brauer relations $\Theta_1$, $\Theta_2$ for $X$, \[\tilde{\lambda}_{\Theta_1+\Theta_2}(X/\mathcal{K}) = \tilde{\lambda}_{\Theta_1}(X/\mathcal{K})\tilde{\lambda}_{\Theta_2}(X/\mathcal{K}).\]
\item  For every $G$-map $\Phi$ realising $\Theta$, \[ \tilde{\lambda}_{\Theta}(X/\mathcal{K}) =\frac{\#\textup{coker} f_{\Phi}(\mathcal{K}) }{\#\textup{ker} f_{\Phi}(\mathcal{K})}\cdot \frac{\prod_{i}\mu({X/H_{i}})}{\prod_{j}\mu( {X/H_{j}'})} \cdot \bigl{|} \sqrt{\textup{deg}(f_{\Phi})}\bigr{|}_{\mathcal{K}}.\]
\end{enumerate}
\end{lemma}

\begin{proof} (1) holds since, using Notation \ref{Notation:Matrices},
 \begin{equation*} \label{Eqn: small Lambda and change of basis} \frac{\lambda_{\Theta}^{ \tilde{\mathcal{B}}_{1}, \tilde{\mathcal{B}}_{2}}(X/\mathcal{K})}{\lambda_{\Theta}^{\mathcal{B}_{1},\mathcal{B}_{2} }(X/\mathcal{K})}=  \frac{| \textup{det}\big{(}[\textup{Id}]_{\tilde{\mathcal{B}}_{1}}^{{\mathcal{B}}_{1}}\big{)}|_{\mathcal{K}}}{| \textup{det}\big{(}[\textup{Id}]_{\tilde{\mathcal{B}}_{2}}^{{\mathcal{B}}_{2}}\big{)}|_{\mathcal{K}}}\quad\textup{ and }\quad \frac{\mathcal{C}_{\Theta}^{\tilde{\mathcal{B}}_{1},\tilde{\mathcal{B}}_{2}}(\Omega^{1}(\JX))}{\mathcal{C}_{\Theta}^{\mathcal{B}_{1},\mathcal{B}_{2}}(\Omega^{1}(\JX))}= \frac{\textup{det}\big{(}[\textup{Id}]_{\tilde{\mathcal{B}}_{1}}^{{\mathcal{B}}_{1}}\big{)}^{2}}{\textup{det}\big{(}[\textup{Id}]_{\tilde{\mathcal{B}}_{2}}^{\mathcal{B}_{2}}\big{)}^{2}}.
 \end{equation*}  
For (2), given a basis $\mathcal{B}_H$ for $\Omega^{1}(\textup{Jac}_{X/H})$, we write ${\mathcal{B}}_1' = \mathcal{B}_{1} \sqcup \mathcal{B}_{H}$ for the corresponding basis for $\Omega^{1}(\prod_{i} \textup{Jac}_{X/H_i} \times \textup{Jac}_{X/H})$. Similarly, we write ${\mathcal{B}}_{2}' = \mathcal{B}_{2} \sqcup \mathcal{B}_H$. Evaluating $\tilde{\lambda}_{\Theta+(H-H)}(X/\mathcal{K})$ with respect to ${\mathcal{B}}_1'$, ${\mathcal{B}}_2'$ and $\tilde{\lambda}_{\Theta}(X/\mathcal{K})$ with respect to $\mathcal{B}_{1}$, $\mathcal{B}_{2}$, gives the desired equality, which 
holds for any choice of bases by (1). (3) follows in the same way as (2). (4) follows from part (1), Lemma \ref{pullback-basis-lemma} and Proposition \ref{prop:deg=regulator_constant}(1). 
\end{proof}

\begin{theorem}\label{thm:lambdatilde} With the same set up as in Theorem \ref{Theorem: Local_Formula_dependent_on_differnetials}, and supposing that $\Omega^{1}(\JX)$ is a self-dual ${K}[G]$-representation, we have
 \[  \ord_p \mathcal{C}_{\Theta}(\mathcal{X}_{p}(\JX)) \equiv  \sum_{v \textup{ place of }K}  \ord_p  \tilde{\lambda}_{\Theta}(X/K_{v})  \mod 2\mathbb{Z},\]
where we extend $\textup{ord}_{p}$ to $\mathbb{Q}(\sqrt{p})$ so that $\ord_p  \tilde{\lambda}_{\Theta}(X/K_{v}) \in \frac{1}{2}\mathbb{Z}.$
\end{theorem} 
\begin{proof} Let $\Phi$ be any $G$-map realising $\Theta$.
By {Lemma \ref{lemma_tilde_lambda}(4)}, $\tilde{\lambda}_{\Theta}(X/K_v)= {\lambda_{\Theta, \Phi}(X/\mathcal{K})}\cdot {\left|\textup{deg} (f_{\Phi})\right |^{1/2}_{v}}$ for all $v$ independently of $\Phi$. Since $\prod_{v} |\textup{deg}(f_{\Phi})|_v=1$, we deduce that \[\ord_{p} \prod_{v} \tilde{\lambda}_{\Theta}(X/K_v) = \ord_{p} \prod_{v} {\lambda}_{\Theta, \Phi}(X/K_v) \overset{\textup{Thm.  \ref{Theorem: Local_Formula_dependent_on_differnetials}}}\equiv \ord_{p} \mathcal{C}_{\Theta}(\mathcal{X}_{p}(\textup{Jac}_{X}))\mod 2. \qedhere\]
\end{proof}

Although $\tilde{\lambda}_{\Theta}(X/\mathcal{K})$ is independent of the underlying bases of differentials, it is difficult to work with in practice since it may not be rational.

\subsection{The local invariant $\Lambda_{\Theta}(X/\mathcal{K})$ and proof of Theorem \ref{Theorem: Local Formula}}\label{subsec_correction_term}   Here we detail how $\Lambda_{\Theta}$ is obtained from the local invariant $\tilde{\lambda}_{\Theta}(X/\mathcal{K})$ introduced in Definition \ref{temporary_lambda}. This involves introducing a correction term which we will denote by $\ct$.
We show that $\Lambda_{\Theta}$ is rational and conclude the section by proving Theorem \ref{Theorem: Local Formula}.

\begin{definition} \label{Def:correction_term} 
Let ${L}$ be a field of characteristic 0, $X/{L}$ be a curve, $G$ be a finite subgroup of $\Aut_L(X)$ and $\Theta$ be a pseudo Brauer relation for $X$ realised by a $G$-map $\Phi$. Assuming that $\Omega^{1}(\textup{Jac}_{X})$ is self-dual as an $L[G]$-representation, we define
$\ct(\Omega^{1}(\JX))$ to be the square-free integer equivalent to $\textup{deg}(f_{\Phi}) \bmod \mathbb{Q}^{\times 2}$. This is independent of $\Phi$ by Proposition \ref{prop:deg=regulator_constant}(3).
\end{definition}

When the $G$-representation $\Omega^{1}(\JX)$ is realisable over $\mathbb{Q}$, we view $\mathcal{C}_{\Theta}(\Omega^{1}(\JX))$ as an element of $\mathbb{Q}^{\times}/\mathbb{Q}^{\times 2}$ via  Remark \ref{remark_descend_representations}.

\begin{lemma} \label{lemma_correction_term_rational}
Suppose that $\Omega^{1}(\JX)$ is  realisable over $\mathbb{Q}$. Then, 
\[\ct(\Omega^{1}(\JX)) \equiv \mathcal{C}_{\Theta}(\Omega^{1}(\JX)) \bmod \mathbb{Q}^{\times 2}.\]
\end{lemma}

\begin{proof}
Let $\Phi$ be any $G$-map realising $\Theta$. Then,
\begin{equation*}\ct(\Omega^{1}(\JX)) \equiv  \textup{deg}(f_{\Phi}) \overset{\textup{Prop. \ref{prop:deg=regulator_constant}(2)}}\equiv {\mathcal{C}_{\Theta}^{\mathcal{B}, \Phi^{\vee}\mathcal{B}}(\Omega^{1}(\JX))} \!\overset{\textup{Thm. \ref{Theorem: Independence of pairing}(1)}}\equiv \! {\mathcal{C}_{\Theta}(\Omega^{1}(\JX))} \mod \mathbb{Q}^{\times 2}. \qedhere\end{equation*}
\end{proof}

\begin{remark} 
For the proof of Theorem \ref{thm:introjacobian} we will restrict to $G = C_2\times C_2$, $D_8$ or $D_{2p}$ for a prime $p$. In these cases $\Omega^{1}(\JX)$ is realisable over $\mathbb{Q}$ by \cite[Lem. 5.10(2)]{Etale_paper},  and so the description of $\ct(\Omega^{1}(\JX))$ given above applies.
\end{remark}

\begin{definition} \label{def:lemma:local_invariants} 
Let $\mathcal{K}$ be a local field of characteristic $0$, $X/\mathcal{K}$ be a curve, $G$ be a finite subgroup of $\Aut_\K(X)$ and $\Theta=\sum_{i} H_{i} - \sum_{j} H_{j}'$  be a pseudo Brauer relation for $X$.
Fix bases $\mathcal{B}_{1}$, $\mathcal{B}_{2}$ for $\Omega^{1}(\prod_{i}  \Jac_{X/H_i})$, $\Omega^{1}(\prod_{j} \Jac_{X/H_j'})$ and write $\omega(\mathcal{B}_{1})$, $\omega(\mathcal{B}_{2})$ for the exterior forms given by the wedge product of the elements in $\mathcal{B}_{1}$, $\mathcal{B}_{2}$ respectively. Assuming that $\Omega^{1}(\textup{Jac}_{X})$ is self-dual as a $\mathcal{K}[G]$-representation, we define
\[ \Lambda_{\Theta}(X/{\mathcal{K}})=  \frac{C( \prod_{i} \Jac_{X/H_{i}},\omega(\mathcal{B}_{1})) \ }{C( \prod_{j} \Jac_{X/H_{j}'},\omega(\mathcal{B}_{2})) } \cdot \frac{\prod_{i}\mu({X/H_{i}})}{\prod_{j}\mu( {X/H_{j}'})}\cdot  \
\Biggl{|}\sqrt{\frac{\ct(\Omega^{1}(\JX))}{\mathcal{C}_{\Theta}^{\mathcal{B}_1,\mathcal{B}_{2}}(\Omega^{1}(\JX))}}\Biggr{|}_{\mathcal{K}}.\]
For the definitions of $C$, $\mu$, $\ct$ and $\mathcal{C}_{\Theta}^{\mathcal{B}_1,\mathcal{B}_{2}}$ see Notation \ref{not:C}, and Definitions \ref{def:deficiency}, \ref{Def:correction_term}, \ref{def:regulator_constants_for_pseudo_brauer} respectively. In particular, $\Lambda_{\Theta}(X/\mathcal{K}) = \tilde{\lambda}_{\Theta}(X/\mathcal{K})\cdot {|} {\ct(\Omega^{1}(\JX))} {|}_{\mathcal{K}}^{{1}/{2}}.$
\end{definition}

\begin{theorem} \label{Thm:main_local_inv_thm}
\label{Theorem:Properties of local invariant} 
Let $X$ be a curve over a local field $\mathcal K$ of characteristic $0$, and $G$ be a finite subgroup of $\Aut_\K(X)$. Let $\Theta = \sum_i H_i - \sum_j H_j'$ be a pseudo Brauer relation for $X$, and suppose that $\Omega^{1}(\JX)$ is a self-dual $\mathcal K[G]$-representation. Then, the following hold.
\begin{enumerate}
\item $\Lambda_{\Theta}(X/\mathcal{K})$ is a rational number independent of $\mathcal{B}_{1}$, $\mathcal{B}_{2}.$
\item ${\Lambda}_{\Theta}(X/\mathcal{K})$ is independent of how $\Theta$ is expressed as a formal linear combination of conjugacy classes of subgroups of $G$. That is, for any subgroup $H$ of $G$, $${\Lambda}_{\Theta}(X/\mathcal{K}) = {\Lambda}_{\Theta+(H-H)}(X/\mathcal{K}).$$ 
\item For any $G$-map $\Phi$ realising $\Theta$,
\[\Lambda_{\Theta}(X/\mathcal{K})= \frac{\#\textup{coker} f_{\Phi}(\mathcal{K})  }{\#\textup{ker} f_{\Phi}(\mathcal{K}) } \cdot   \frac{\prod_{i}\mu (X/{H_{i}})}{\prod_{j}  \mu(X/{H_{j}'})}  \cdot \Bigl{|} \sqrt{\textup{deg}(f_{\Phi}) \cdot \ct(\Omega^{1}(\JX))}\Bigr{|}_{\mathcal{K}}. \]
\item If $\mathcal{K}=\mathbb{C}$, then $\Lambda_{\Theta}(X/\mathbb{C}) = \ct(\Omega^{1}(\JX))$.
\end{enumerate}
\end{theorem} 

\begin{proof} 
Independence of the bases $\mathcal{B}_{1}$, $\mathcal{B}_{2}$ follows from Lemma \ref{lemma_tilde_lambda}(1). (2) follows from Lemma \ref{lemma_tilde_lambda}(2), provided we show that $\mathcal{C}_{\Theta}^{\textup{sf}} = \mathcal{C}_{\Theta'}^{\textup{sf}}$ where $\Theta' = \Theta+(H-H)$. Let $\Phi$ be any $G$-map realising $\Theta$, and consider $\Phi' := \Phi \oplus \textup{id}$, where $\textup{id}$ is the identity on $\mathbb{Z}[G/H]$. Then, $\Phi'$ realises $\Theta'$ and $f_{\Phi'}= f_{\Phi} \times \textup{id}|_{\textup{Jac}_{X/H}}$. Therefore, $\mathcal{C}_{\Theta'}^{\textup{sf}} \equiv \textup{deg}(f_{\Phi'}) = \textup{deg}(f_{\Phi}) \equiv \mathcal{C}_{\Theta}^{\textup{sf}} \bmod \mathbb{Q}^{\times 2},$ and so they must be equal. (3) follows from Lemma \ref{lemma_tilde_lambda}(4), and (4) follows from (3) since $\textup{coker}f_{\Phi}(\mathcal{K})$ is trivial when $\mathcal{K} = \mathbb{C}$. It remains to show that $\Lambda_{\Theta}$ is rational, to complete the proof of (1). By Proposition \ref{prop:deg=regulator_constant}(3), $\textup{deg}(f_{\Phi}) \cdot \ct \in \mathbb{Q}^{\times 2}$ independently of $\Phi$, and so rationality follows from (3).
\end{proof}

We are now able to prove Theorem \ref{Theorem: Local Formula}.

\begin{proof}[Proof of Theorem \ref{Theorem: Local Formula}] 
We note that $\Lambda_{\Theta}(X/{K}_v) = \lambda_{\Theta,\Phi}(X/{K}_v) \cdot {|}\textup{deg}(f_{\Phi}) \cdot \ct(\Omega^{1}(\JX)){|}_{v}^{1/2}$, and so $\prod_{v} \Lambda_{\Theta}(X/K_v)=\prod_{v} \lambda_{\Theta, \Phi}(X/K_v)$ since  $\textup{deg}(f_{\Phi}) \cdot \ct \in \mathbb{Q}^{\times 2}$. By Theorem \ref{Theorem:Properties of local invariant}(1), $\ord_p  \prod_{v} \Lambda_{\Theta}(X/K_v) =\sum_v \ord_p \Lambda_{\Theta}(X/K_v)$ and the result then follows by Theorem \ref{Theorem: Local_Formula_dependent_on_differnetials}.
\end{proof}

Unfortunately, one drawback of the invariant $\Lambda_\Theta$ is that it is generally not multiplicative in $\Theta$ (unlike $\tilde \lambda_\Theta$, see Lemma \ref{lemma_tilde_lambda}(3)). This is the case since generally $\mathcal{C}_{\Theta_1+\Theta_2}^{\textup{sf}}\neq \mathcal{C}_{\Theta_1}^{\textup{sf}} \cdot \mathcal{C}_{\Theta_2}^{\textup{sf}}$. 

\begin{lemma} \label{lem:no_multiplicativity}
Let $\Theta, \Theta_1, \Theta_2$ be pseudo Brauer relations for $X$. Then,
\begin{enumerate}
\item $\frac{\Lambda_{\Theta_1+\Theta_2}(X/\mathcal{K})}{\Lambda_{\Theta_1}(X/\mathcal{K})\Lambda_{\Theta_2}(X/\mathcal{K})}  = \Biggl{|}\sqrt{\frac{\cC_{\Theta_1+\Theta_2}^{\textup{sf}}(\Omega^{1}(\JX))}{\cC_{\Theta_1}^{\textup{sf}}(\Omega^{1}(\JX))\cC_{\Theta_2}^{\textup{sf}}(\Omega^{1}(\JX))}}\Biggr{|}_{\mathcal{K}},$ 
\item $\Lambda_{\Theta}(X/\mathcal{K}) \Lambda_{-\Theta}(X/\mathcal{K})= |\ct(\Omega^{1}(\JX)) |_{\mathcal{K}}$.
\end{enumerate} 

\end{lemma}
\begin{proof} 
(1) follows from Lemma \ref{lemma_tilde_lambda}(3). For (2), let $\Psi$ be the trivial pseudo Brauer relation (i.e. one of the form $\sum_{i} H_{i} - \sum_{i} H_{i}$). Then, $\Lambda_{\Psi}(X/\mathcal{K})=\mathcal{C}_{\Psi}^{\textup{sf}}(\Omega^{1}(\JX))=1$. By part (1), we get
\[\Lambda_{\Theta}(X/\mathcal{K})\Lambda_{-\Theta}(X/\mathcal{K})=  \
\Bigl{|}\sqrt{{\ct(\Omega^{1}(\JX))\CT(\Omega^{1}(\JX))}}\Bigr{|}_{\mathcal{K}}. \] 
Let $\Phi$ be any $G$-map realising $\Theta$. Then $\Phi^{\vee}$ realises $-\Theta$ and $\ct \equiv \textup{deg}(f_{\Phi})= \textup{deg}(f_{\Phi}^{\vee}) \overset{\textup{Thm. \ref{isogeny_from_pseudo_brauer}(3)}}= \textup{deg}(f_{\Phi^{\vee}})  \equiv  \CT \textup{ mod } \mathbb{Q}^{\times 2}.$ Since $\ct$ and $\CT$ are square-free integers, they must be equal. 
\end{proof}

The following lemma provides a simplification of $\Lambda_{\Theta}(X/\mathcal{K})$ when $\mathcal{K}$ is non-archimedean. 
\begin{lemma} \label{local_term_neron_basis}
Let $\mathcal K$ be a non-archimedean local field of characteristic $0$. Let $\mathcal{J}_{H}$ be the N\'{e}ron model of $\Jac_{X/H}$ over $\mathcal{O}_{\mathcal{K}}$ and let $ \mathcal{B}(\mathcal{J}_{H})$ be a basis for the $\mathcal{O}_{\mathcal{K}}$-module $\Omega^{1}(\mathcal{J}_{H})$. Letting $\mathcal{N}_{1} = \bigsqcup_i \mathcal{B}(\mathcal{J}_{H_{i}})$ and $\mathcal{N}_{2} = \bigsqcup_j \mathcal{B}(\mathcal{J}_{H_{j}'})$, 
\[\Lambda_{\Theta}(X/\mathcal{K}) = 
    \frac{\prod_{i} c(\Jac_{X/H_i}) \cdot \mu (X/{H_{i}})}{\prod_{j} c(\Jac_{X/H_{j}'}) \cdot \mu(X/{H_{j}'})}  \cdot
\
\Biggl{|}\sqrt{\frac{\ct(\Omega^{1}(\JX))}{\mathcal{C}_{\Theta}^{\mathcal{N}_1,\mathcal{N}_{2}}(\Omega^{1}(\JX))}}\Biggr{|}_{\mathcal{K}}.\]
\end{lemma}
\begin{proof}
For a  N\'eron basis $\mathcal{B}(\mathcal{J}_{H})$, the exterior form $\omega(\mathcal{B}(\mathcal{J}_{H}))$ on $\Jac_{X/H}$ coincides with the exterior form on the N\'{e}ron model of $\mathcal{J}_{H}$ over $\mathcal{O}_{\mathcal{K}}$.
\end{proof}

\section{Parity conjecture for elliptic curves}\label{Sec_ellcurves} 

Here we compare the local terms appearing in Example \ref{ex:S3part1} to the corresponding root numbers. This gives a new proof of the parity conjecture for elliptic curves over number fields, assuming finiteness of Tate--Shafarevich groups. 

\begin{theorem}\label{thm:mainthm}
 Let  $E: y^2 = x^3+ax+b$ be an elliptic curve over a number field $K$ with $a\neq 0$.  Define the elliptic curve $E':y^2=x^3-27bx^2-27a^3x$. If $\Sha(E)$ has finite $3$-primary part and $\Sha(E')$ has finite $2$- and $3$-primary parts, then the parity conjecture holds for $E$.
\end{theorem}

\begin{remark}\label{rem:3isogeny-case} An elliptic curve $E: y^2 = x^3+b$ admits a $3$-isogeny and so, assuming that the $3$-primary part of $\Sha(E)$ is finite, the parity conjecture is known to hold by \cite[Theorem 1.8]{kurast}.
\end{remark}

The proof of Theorem \ref{thm:mainthm} uses the theory of covers of curves developed in \S\S\ref{sec:isogenies}-\ref{sec:Rank_Parity}, as opposed to the proof given in \cite[Theorem 1.2]{kurast} which comes from studying elliptic curves over extensions of number fields. We adhere to the following notation.

\begin{notation} Let $E:y^2 = f(x):= x^3+ax+b$ be an elliptic curve over a field $K$ with $a\neq0$. Let $g(y^2) = -27y^4+54by^2-(4a^3+27b^2)\in K[y]$ be the discriminant of $f(x)-y^2$. The {\em discriminant curve} is $D:\Delta^2 = g(y^2)$ and the {\em bihyperelliptic cover} of $E$ is  $B:\{y^2 = f(x),\,\Delta^2 = g(y^2)\}.$
\end{notation}

The coefficient $a$ being non-zero ensures that the genus of $D$ is 1 and that $K(B)/K(y)$ is an $S_3$-Galois extension (c.f.  Example \ref{ex:S3part1}). In particular, $S_3$ acts on $B$ by automorphisms.   
Example \ref{ex:S3part1}  asserts that $B$ has genus 3  (we call it `the bihyperelliptic cover of $E$' because it has a bihyperelliptic model), that $\textup{Jac}_D$ is the elliptic curve $E'$ appearing in the statement of Theorem \ref{thm:mainthm}, the existence of an isogeny involving $E$, $\Jac_D$ and $\Jac_B$, and a local formula for the parity of the rank of $E\times\Jac_D$.

\begin{theorem}\label{thm:localformula} Let $E:y^2 = x^3+ax+b$ be an elliptic curve over a number field $K$ with $a\neq 0$, let $D$ be its discriminant curve and let $B$ be its bihyperelliptic cover. Then,
\begin{equation*}
  \rk_3 E + \rk_3 \Jac_D ~\equiv~ \sum_{v\textup{ place of }K}\ord_3\Lambda_\Theta(B/K_v) \mod 2
\end{equation*}
where $\Theta = 2C_2 +C_3-2S_3 -\{1\}$.\end{theorem}

\begin{proof} 
This follows from Theorem \ref{thm:introgaloisdescent}(4) and Theorem \ref{Theorem: Local Formula}, via an identical argument to that given in Example \ref{ex:S3part1}.
\end{proof}

To deduce the $3$-parity conjecture for $E\times \Jac_D$ it remains to suitably relate $\Lambda_\Theta(B/K_v)$ to the product  of local root numbers  $w({E/K_v})w({\Jac_D/K_v})$.  

\begin{theorem}\label{thm:localthm} Let $E:y^2 = x^3+ax+b$ be an elliptic curve over a local field $\K$ of characteristic 0 with $a\neq 0$, let $D$ be its discriminant curve and let $B$ be its bihyperelliptic cover. 
Then, 
$$ (-1)^{\ord_3\Lambda_\Theta(B/\K)}=w({E/\K})w({\Jac_D/\K}) $$ where $\Theta = 2C_2 +C_3-2S_3 -\{1\}$.\end{theorem}

Before proving this relationship, we note some immediate consequences. For the following result only, we allow for the possibility that $a=0$.

\begin{theorem}\label{thm:3PC} Let $E:y^2 =x^3+ax+b$ be an elliptic curve over a number field $K$ and $D$ its discriminant curve. The $3$-parity conjecture holds for $E\times \Jac_D$.
\end{theorem}

\begin{proof}

If $a=0$ then $\Jac_D=0$ and the Theorem holds by Remark \ref{rem:3isogeny-case}.

If $a\neq 0$ 
then Theorem \ref{thm:localthm} says that $(-1)^{\ord_3\Lambda_\Theta(B/K_v)}=w({E/K_v})w({\Jac_D/K_v})$ for each place $v$ of $K$. Taking the product over all places and invoking Theorem \ref{thm:localformula} gives the result.
\end{proof} 

\begin{proof}[Proof of Theorem \ref{thm:mainthm}] 
As above,  $E'=\textup{Jac}_D$. By Theorem \ref{thm:3PC}, the $3$-parity conjecture holds for $E\times E'$. Since $E'$ has a $2$-isogeny over $K$, the $2$-parity conjecture holds for $E'$ by \cite[Theorem 1.8]{kurast}. Together, these statements give the result.
\end{proof}

\begin{remark}\label{rem:psicompat} Let $\Psi$ be the Brauer relation for $C_2\times C_2$ identified in Example \ref{ex:C2C2rankrecipe}. Then,
$$\prod_{v \textup{ place of }K} (-1)^{\ord_2\Lambda_\Psi(D/ K_v)}w(\Jac_D/ K_v) \overset{\textup{Ex. \ref{ex:C2C2rankrecipe}}}{=} (-1)^{\rk_2(\Jac_D)}w(\Jac_D)\overset{\textup{\cite[Thm. 1.8]{kurast}}}{=}  1.$$
The explicit relationship between $(-1)^{\ord_2\Lambda_\Psi(D/ K_v)}$ and $w(\Jac_D/ K_v)$ appears to be complicated.
\end{remark}

We now prove Theorem \ref{thm:localthm}.

\begin{notation}Let $\sigma$, $\tau\in S_3$ have orders $3$ and $2$ respectively, with $K(E) = K(B)^{\langle \tau\rangle}$ and $K(D) = K(B)^{\langle \sigma\rangle}$. Let $\Phi$ denote the injective $S_3$-module homomorphism
\begin{align*}
\Z[S_3/\langle \tau\rangle]y_1\oplus \Z[S_3/\langle \tau\rangle]y_2\oplus\Z[S_3/\langle \sigma\rangle]y_3&\longrightarrow  \Z x_1\oplus \Z x_2\oplus\Z[S_3]x_3
\end{align*}
given by $y_1 \mapsto  x_1+(1+\tau) x_3$, $y_2 \mapsto x_1+x_2+(\sigma+\tau\sigma) x_3$ and $y_3 \mapsto x_1 + x_2 +(1+\sigma+\sigma^2)x_3$.
By Theorem \ref{isogeny_from_pseudo_brauer}, there are   isogenies
\[  f_{\Phi}:\Jac_B \to E\times E\times \Jac_{D}\quad \textup{ and }\quad 
      f_{\Phi^\vee}=f_{\Phi}^\vee:E\times E\times \Jac_{D} \to \Jac_B.\]
 From the explicit construction of these isogenies given in \cite[\S4.3]{Etale_paper}, we see that
\[f_\Phi=(\pi_{E*}, \pi_{E*}\circ \sigma_*,\pi_{D*})\quad \textup{ and }\quad f_\Phi^\vee=\pi_E^*+\sigma^*\circ \pi_E^*+\pi_D^*,\]
    where $\pi_E:B\to E$, $\pi_D:B\to D$ denote the quotient maps.
\end{notation}

\begin{lemma}\label{lem:isogdeg} 
We have $\deg f_\Phi= \deg f_\Phi^\vee=9$ and
$\ker(f_{\Phi})=\pi_D^*(\textup{Jac}_D[3])$. The first projection $E\times E\times \textup{Jac}_B\rightarrow E$ induces an isomorphism  of Galois modules $\ker(f_{\Phi}^\vee)\iso  E[3]$.
\end{lemma}

\begin{proof} Let $P, Q\in E$ and $R\in \Jac_D$. Then $(f_\Phi\circ f_\Phi^\vee)(P,Q,R) = \big(2P-Q, 2Q-P, 3R\big)$. For the first entry, this follows using that ${\pi_E}_*\circ\pi_E^* = [2]$, ${\pi_E}_*\circ\sigma^*\circ\pi_{E}^*=[-1]$ and ${\pi_E}_*\circ\pi_D^* = 0$, and similarly for the second entry. For the third entry, we additionally use that ${\pi_D}_*\circ\pi_D^* = [3]$. We therefore deduce that \[\ker (f_\Phi\circ f_\Phi^\vee) = \{(P,-P,R): P\in E[3], \,R\in \Jac_D[3]\}\] and 
$\deg f_\Phi= \deg f_\Phi^\vee=9$. Further, the identities above show that $\pi_D^*(\textup{Jac}_D[3])$ is contained in $\ker(f_\Phi).$ By \cite[Lemma 6]{BP01}, since $B\rightarrow D$ has no non-trivial unramified subcover, $\pi_D^*:\textup{Jac}_D\rightarrow \textup{Jac}_B$ is injective. Since $\deg f_\Phi=9$ this gives $\ker(f_{\Phi})=\pi_D^*(\textup{Jac}_D[3])$. Finally, the description of $\ker (f_\Phi\circ f_\Phi^\vee)$ given above shows that the first projection induces a map $\alpha:\ker(f_\Phi^\vee) \rightarrow E[3]$ whose kernel is contained in $0\times 0\times \textup{Jac}_D[3]$. Since $\pi_D^*$ is injective,  $0\times 0\times \textup{Jac}_D[3]$ and $\ker(f_\Phi^\vee)$ have trivial intersection. Thus $\alpha$ is injective. Since $\deg(f_\Phi^\vee)=9$, $\alpha$ is an isomorphism.
\end{proof}

\begin{remark}
The proof of Lemma \ref{lem:isogdeg} shows that $\pi_D^*$ is injective. We conclude that $\ker(f_\Phi)\simeq \textup{Jac}_D[3]$ and $\ker(f_\Phi^\vee)\simeq E[3]$ as Galois modules. Since $E[3]$ is its own Cartier dual, we have an isomorphism of Galois modules $\textup{Jac}_D[3]\simeq E[3]$. 
\end{remark}

\begin{remark} \label{lem:Csf_example}
Since $\deg f_\Phi=9$, the `correction term' $\ct(\Omega^{1}(\Jac_B))$ of Definition \ref{Def:correction_term} is equal to $1$.
\end{remark}

\begin{proposition}\label{prop:S3C} 
Theorem \ref{thm:localthm} holds when $\K =\mathbb{C}$ or $\mathbb{R}$.
\end{proposition}

\begin{proof}
We have $w({E/\K})=w({\Jac_D/\K})=-1$. The case $\K=\mathbb{C}$ then follows from Theorem \ref{Thm:main_local_inv_thm} and Remark \ref{lem:Csf_example}. Henceforth, we suppose that $\K=\mathbb{R}$.

 In the notation of Definition \ref{def:local_invariants}, we have $\ord_3\#\textup{coker} f_{\Phi}^\vee(\mathcal{K}) = 0$ and $\#\textup{ker} f_{\Phi}^\vee(\mathcal{K}) = 3$ (by Lemma \ref{lem:isogdeg}). We deduce that
\begin{equation*}\ord_3\Lambda_{\Theta}(B/\mathcal{K})  \overset{ \textup{Lem. \ref{lem:no_multiplicativity}(2)  } }{=}-\ord_3\Lambda_{-\Theta}(B/\mathcal{K}) \overset{\substack{\textup{Thm. \ref{Thm:main_local_inv_thm}(3) \&}\\\textup{Lem. \ref{lem:isogdeg}}}}{=} 
-\ord_3(1) ~=~ 0.\qedhere
\end{equation*}
  \end{proof}

 \begin{notation}
In the computations that follow, we will exploit the existence of a degree $8$ isogeny 
$\eta:\Jac_B\to E\times\Jac_C $
where $C: z^2 = -(3x^2+4a)(x^3+ax+b)$ has genus $2$ (cf. \cite[\S2.2]{bruin}).  Briefly, making the change of variables $t=\Delta/(3x^2+a)$, one sees that $B$ can be given by the equations $\{y^2=f(x), t^2=-(3x^2+4a)\}$. We have a degree $2$ morphism $\pi_C:B\rightarrow C$ given by $(x,y,t)\mapsto (x,ty)$. Then $\eta=(\pi_{E*}, \pi_{C*})$ gives the sought isogeny (to compute its degree, note that $\eta^\vee=\pi_E^*+\pi_C^*$ and $\eta\circ \eta^\vee= 2$). 
\end{notation}

\begin{remark}
One can also construct the isogeny $\eta$ via Theorem \ref{isogeny_from_pseudo_brauer}, using the Brauer relation $\Psi$ of Example \ref{ex:C2C2rankrecipe} in the $C_2\times C_2$ extension $K(B)=K(x,y,t)/K(x)$. 
\end{remark}
 
Recall that $c(A)$ denotes the Tamagawa number of an abelian variety $A$ defined over a non-archimedean local field $\mathcal K$. 

\begin{proposition}\label{prop:S3Qpss}
Theorem \ref{thm:localthm} holds when $\K/\Q_p$ is finite with $p\neq 2,3$, and $E\times\Jac_D/\K$ is semistable.
\end{proposition}

\begin{proof} After a change of variable, we may assume that $a,b\in \mathcal O_\K$. 
Let $\pi$ be a uniformiser for $\K$, and write $v$ for the normalised valuation on $\overline\K$ (so that $v(\pi) = 1$).

Since the isogeny $\eta:\textup{Jac}_B\rightarrow E\times \textup{Jac}_C$ has degree coprime to $3$, it suffices to show that $$(-1)^{\ord_3(c(\Jac_D)c(E)c(\Jac_C))} = w(E/\K)w(\Jac_D/\K).$$
Indeed, $c({\Jac_B}) = 2^n c(E) c({\Jac_C})$ for some $n\in\Z$, by an easy generalisation of \cite[Lemma 6.2]{dokchitser2015local}. 

We remind the reader that an elliptic curve $\mathcal E/\K$ of type $I_n$ has $c(\mathcal E) = n$, $w(\mathcal E) = -1$ in the case of split multiplicative reduction, and $c(\mathcal E)=\textup{gcd}(n,2)$, $w(\mathcal E) = 1$ otherwise (see \cite{MR0393039}, \cite[\S 19]{Rohrlich}).
We will compute $c(\Jac_C)$ using the theory of cluster pictures (see \cite[Definitions 1.1 \& 1.13, Tables 1 \& 3]{DDMM18}).

As above, $\Jac_D$ is given by the Weierstrass equation $Y^2 = X^3-27bX^2-27a^3X$. Let $d :=4a^3+27b^2$, then $\Delta_E = -2^4\cdot d$, $\Delta_{\Jac_D} =2^4\cdot 3^9 \cdot a^6 d$ and $\Delta_C = 2^{12}\cdot 3^3 \cdot a  d^3$. Since $E$ is semistable, we may assume that if  $v(a)>0$  then $v(b)=0.$

\textbf{Suppose {$a,d$ are units}.} $E$, $\Jac_D$ and $\Jac_C$ have good reduction, so $c(\Jac_D)=c(E)=c(\Jac_C)=1$ and $w(E/\K) = w(\Jac_D/\K)= 1$.

\textbf{Suppose {$a\equiv 0 \pmod \pi$}}. Then $b$, $d$ are units, so $E$ has good reduction and $c(E) = 1$, $w(E/\K) = 1$. The reductions of $\Jac_D$ and $C$ are $Y^2 = X^2(X-27b)$ of type $I_{6v(a)}$, and $z^2 = -3x^2(x^3+b)$ with cluster picture for $C$ given by $\scalebox{0.8}{\clusterpicture            
  \Root[] {1} {first} {r2}
  \Root[] {1} {r2} {r3}
  \ClusterLD c1[\epsilon][v(a)] = (r2)(r3);
    \Root[] {} {c1} {r4}
  \Root[] {} {r4} {r5}
  \Root[] {} {r5} {r6}
  \ClusterD c3[0] = (c1)(r4)(r5)(r6);
\endclusterpicture}$. If $-3b$ is a square modulo $\pi$ then $\epsilon = +$ and $c(\Jac_D)=6v(a)$, $c(\Jac_C) = 2v(a)$, $w(\Jac_D/\K) = -1$. Otherwise, $\epsilon = -$ and $c(\Jac_D) =c(\Jac_C) = 2$, $w(\Jac_D/\K) = 1$. 

\textbf{Suppose {$d\equiv 0 \pmod \pi$}}. Then $a$, $b$ are units. The reductions of $E$, $\Jac_D$ and $C$ are $y^2 = (x-\frac{3b}{a})(x+\frac{3b}{2a})^2$, $Y^2 = X(X-\frac{27}{2}b)^2$ both of type $I_{v(d)}$, and $z^2 = -3(x-\frac{3b}{a})^2(x+\frac{3b}{2a})^2(x+\frac{3b}{a})$ with cluster picture for $C$ given by
$\scalebox{0.8}{\clusterpicture            
  \Root[] {1} {first} {r2}
  \Root[] {1} {r2} {r3}
  \ClusterLD c1[\epsilon][v(d)] = (r2)(r3);
    \Root[] {1} {c1} {r4}
  \Root[] {} {r4} {r5}
    \ClusterLD c2[\epsilon][2v(d)] = (r4)(r5);
  \Root[] {} {c2} {r6}
  \ClusterD c3[0] = (c1)(c2)(r6);
\endclusterpicture}$.
If $6b$ is a square modulo $\pi$ then $\epsilon = +$ and $c(E) = c(\Jac_D)=v(d)$, $c(\Jac_C)=8v(d)^2$, $w(E/\K) = w(\Jac_D/\K) = -1$. Otherwise, $\epsilon = -$ and $c(E) = c(\Jac_D)=\textup{gcd}(v(d),2)$, $c(\Jac_C)=4$, $w(E/\K) = w(\Jac_D/\K) = 1$. 
\end{proof}

When $\mathcal{K}$ is a finite extension of $\mathbb{Q}_3$, the following result expresses $\ord_3\Lambda_\Theta(B/\K)$ in terms of invariants associated to the genus $1$ curves $E$ and $D$, and the genus $2$ curve $C$. This is used to establish the local constancy of $\ord_3\Lambda_\Theta(B/\K)$ given in Lemma \ref{lem:invariantcontinuity} below. In the statement, we use \cite[Proposition 2.2]{MR861976} to identify regular differentials on a curve $X/\mathcal{K}$ with differentials on $\JX$. We write  $\omega^0_X$  for the N\'{e}ron exterior form, viewed on $X$ via this identification. 

\begin{lemma}\label{lem:3adicterm} Let $\K/\Q_3$ be finite. Then, $$\ord_3\Lambda_\Theta(B/\K) = \ord_3\Big(\frac{c(E)c({\Jac_D})}{c({\Jac_C})} \Big|\frac{\gamma}{\alpha\beta} \Big|_\K\Big)$$
where $\alpha$, $\beta$, $\gamma\in \K$ are such that $\omega_E^0 = \alpha \frac{dx}{y}$, $\omega_{D}^0 = \beta \frac{dy}{\Delta}$ and $\omega_{C}^0 = \gamma( \frac{dx}{z}\wedge x\frac{dx}{z})$. 
\end{lemma}

\begin{proof} Let $\mathcal B_1 = \big\{(\frac{dx}{y},0,0), (0,\frac{dx}{y},0), (0,0,\frac{dy}{\Delta})\big\}$ be a basis for $\Omega^1(E\times E\times \Jac_D)$. Evaluating $\Lambda_\Theta$ with $\mathcal B_2 = \Phi^\vee\mathcal B_1$ and using Proposition \ref{prop:deg=regulator_constant}(2), Remark \ref{lem:Csf_example} and Lemma \ref{lem:isogdeg}, gives 
$$\ord_3\Lambda_\Theta(B/\K) ~=~ \ord_3\Big(\frac{c(E)^2c({\Jac_D})}{c({\Jac_B})} \Bigg|\frac{\omega(\mathcal B_1)}{\omega_{E\times E\times \Jac_D}^0} \cdot \frac{\omega_{\Jac_B}^0}{\omega(\Phi^\vee\mathcal B_1)}\Bigg|_\K| 3|_\K\Big)$$
where $\omega(\mathcal B_1)$ and $\omega(\Phi^\vee\mathcal B_1)$ denote the exterior forms on $E\times E\times \Jac_D$ and $\Jac_B$ obtained by taking the wedge product of the elements in $\mathcal B_1$ and $\Phi^\vee\mathcal B_1$ respectively. 

As in the proof of Proposition \ref{prop:S3Qpss},  the   isogeny $\eta:\textup{Jac}_B\rightarrow E\times \textup{Jac}_C$ forces  $\ord_3(c({\Jac_B})) = \ord_3(c(E) c({\Jac_C}))$. Similarly, a straightforward generalisation of \cite[Lemma 4.3]{dokchitser2015local} gives
$$\Bigg|\frac{\omega_{\Jac_B}^0}{\omega(\Phi^\vee\mathcal B_1)}\Bigg|_\K  ~=~ \Bigg|\frac{\omega_{\Jac_B}^0}{\eta^*\omega_{E\times \Jac_C}^0}\cdot\frac{\eta^*\omega_{E\times \Jac_C}^0}{\omega(\Phi^\vee\mathcal B_1)}\Bigg|_\K ~=~ \Bigg|\frac{\eta^*\omega_{E\times \Jac_C}^0}{\omega(\Phi^\vee\mathcal B_1)} \Bigg|_\K,$$
where $\omega(\Phi^\vee\mathcal B_1)=\frac{dx}{y}\wedge \sigma^2\frac{dx}{y}\wedge \frac{dy}{\Delta}= \pm\frac{3}{4}(\frac{dx}{y}\wedge \frac{dx}{z}\wedge x\frac{dx}{z})$. Using that Néron models respect products, we see that $\big|{\eta^*\omega_{E\times \Jac_C}^0}/{\omega(\Phi^\vee\mathcal B_1)} \big|_\K = |\alpha\gamma/3|_\K$ and $\big|\omega(\mathcal B_1)/\omega_{E\times E\times \Jac_D}^0\big|_\K = |1/\alpha^2\beta|_\K.$ This completes the proof.
\end{proof}

\begin{lemma}\label{lem:invariantcontinuity}
Let $E:y^2 = x^3+ax+b$ an elliptic curve with $a\neq0$ over a finite extension $\K/\Q_p$. There is an $\epsilon>0$ such that changing $a$, $b$ to any $a'\neq 0$, $b'$ with $|a-a'|_\K$, $|b-b'|_\K<\epsilon$ does not change $w({E/\K})$, $w({\Jac_D/\K})$ and $\ord_3\Lambda_\Theta(B/\K)$.
\end{lemma}

\begin{proof}
Root numbers are functions of $V_\ell E = T_\ell E \otimes \Q_\ell$, so their local constancy can be seen from that of the Tate module \cite[p. 569]{MR1680032}. The same argument applies to the $3$-part of the Tamagawa number of an abelian variety $A/\K$ when $p\neq 3$, since $\ord_3c({A}) = \ord_3\#(\Phi_A[3^\infty])^{\textup{Frob}_{\mathcal{K}}}$ and $\Phi_A[3^\infty]\iso H^1(I_\K, T_3(A))_\textup{tors}$ by \cite[\S11]{grothendieck1971modeles} ($\Phi_A$ denotes the component group of the special fibre of the N\'eron model of $A$ over $\mathcal O_\K$).

Now consider $\ord_3\Lambda_\Theta(B/\K)$ when $p=3$. By Lemma \ref{lem:3adicterm}, we need to show that $\ord_3 c(E)$, $\ord_3 c(\Jac_D)$, $\ord_3 c(\Jac_C)$ and $\ord_3|\gamma/\alpha\beta|_\K$ are all locally constant. For the terms concerning $E$ and $\Jac_D$ this follows from Tate's algorithm \cite{MR0393039}. For the terms concerning $\Jac_C$ we argue as in the proof of \cite[Lemma 11.2]{DM2019}. Specifically, each term can be encoded in terms of the special fibre of the minimal regular model of $C$: for Tamagawa numbers use \cite[\S1]{bosliu} and for  N\'{e}ron exterior forms use \cite[\S3]{vanB22}. The result is now a consequence of Liu's algorithm \cite{liumin}. 
\end{proof}

\begin{proposition}\label{prop:S3Qp}Theorem \ref{thm:localthm} holds for any finite extension $\K/\Q_p$.
\end{proposition}

\begin{proof} 
We will deduce the remaining cases from known instances of the $3$-parity conjecture. To do this, we approximate $f(x)$ by a separable cubic $f_0(x) = x^3+a_0x+b_0\in \mathcal O_L[x]$ with $a_0\neq0$ where $L$ is a totally real field, subject to certain conditions, and let $E_0:y^2 = f_0(x)$ with $D_0$ its discriminant curve and $B_0$ its bihyperelliptic cover.

To begin, we first verify Theorem \ref{thm:localthm} for the single elliptic curve $E: y^2 = x^3 -\frac{1}{3}x+\frac{35}{108}$ over $\mathbb{Q}_2$. This can be done by explicit computation; see \cite[Lemma 5.2.5]{hollythesis} for details. 

Now suppose $\K$ is a finite extension of  $\Q_3$. Let $L$ be a totally real number field with a unique prime $\mathfrak q\mid3$, that further satisfies $L_{\mathfrak q} \iso \K$, and is such that   $L_{\mathfrak{r}}=\mathbb{Q}_2$ for every prime $\mathfrak{r}\mid 2$ (to see that such a field exists, if $\K = \Q_3[x]/(h(x))$ for some monic $h(x)\in\Q_3[x]$, then approximate $h(x)$ by some $\tilde h(x) \in \Q[x]$ that splits completely over $\mathbb{R}$ and $\mathbb{Q}_2$; take $L = \Q[x]/(\tilde h(x))$). 

With $L$ fixed as above, choose $a_0$ (resp. $b_0$) in $O_L$ to be $\mathfrak q$-adically close to $a$ (resp. $b$), and $\mathfrak r$-adically close to $-\frac{1}{3}$ (resp. $\frac{35}{108}$)  for all $\mathfrak r\mid 2$. 
For primes $ \mathfrak p\nmid2,3$, we do this in such a way that $\mathfrak p\nmid b_0$ whenever $\mathfrak p\mid a_0$ (to arrange this, first pick $a_0$ which is suitable at primes over $2,3$, then pick a suitable $b_0$ given this). 

Combining Lemma \ref{lem:invariantcontinuity} with Propositions \ref{prop:S3C}  and \ref{prop:S3Qpss}, we see that Theorem \ref{thm:localthm} holds for $E_0/L_v$ whenever $v\neq \mathfrak{q}$. 
Since the 3-parity conjecture holds for elliptic curves over totally real fields by  \cite[Theorem E]{Nekovar3}, we have
\begin{eqnarray*}
1 &=& (-1)^{\rk_3E_0+\rk_3\Jac_{D_0}}w(E_0)w(\Jac_{D_0})\\[0.5em]
&\overset{\textup{Thm. \ref{thm:localformula}}}{=}&\prod_{v\textup{ place of }L}(-1)^{\ord_3\Lambda_\Theta(B_0/L_v)}w({E_0/L_v})w({\Jac_{D_0}/L_v})\\[0.5em]
&=&(-1)^{\ord_3\Lambda_\Theta(B_0/L_{\mathfrak q})}w({E_0/L_{\mathfrak q}})w({\Jac_{D_0}/L_{\mathfrak q}})\\[0.5em]
&\overset{\textup{Lem. \ref{lem:invariantcontinuity}}}{=}&(-1)^{\ord_3\Lambda_\Theta(B/\K)}w({E/\K})w({\Jac_{D}/\K}).
\end{eqnarray*}

Next, suppose that $\K/ \Q_2$ is finite. This time, take $L$ be a totally real number field with a unique prime $\mathfrak q\mid2$, which further satisfies $L_{\mathfrak q} \iso \K$. 
Choose $a_0,b_0\in \mathcal O_L$ to be $\mathfrak q$-adically close to $a,b$ respectively. 
For primes $\mathfrak p\nmid2,3$, ensure that $\mathfrak p\nmid b_0$ whenever $\mathfrak p\mid a_0$.
 Arguing as above, and using the newly proven case at primes over $3$, we see that   Theorem \ref{thm:localthm} holds for $E/\K$.

Finally, suppose that $\K/ \Q_p$ is finite. Repeating the argument for $p=2$ but replacing $2$ by $p$ and the condition $\mathfrak p \nmid 2,3$ by $\mathfrak p \nmid 2,3,p$ we conclude that Theorem \ref{thm:localthm} holds for $E/\K$.
\end{proof}

\begin{proof}[Proof of Theorem \ref{thm:localthm}] Combine Proposition \ref{prop:S3C} with Proposition \ref{prop:S3Qp}. 
\end{proof}

\section{Parity of ranks of Jacobians}\label{Sec_applications}
 
The aim of this section is to define the ``arithmetic analogue of root numbers'' $w_{\text{arith}}(X/\mathcal{K})$ for curves over local fields and prove Theorem \ref{thm:introjacobian}.

In \S\ref{sec:Rank_Parity} we explained how to compute the parity of the multiplicity of certain representations $\tau_{\Theta,p}$ inside Selmer groups of Jacobians using local data (see Definition \ref{Def: p-adically nontrivial reps} \& Theorem \ref{Theorem: Local Formula}). In this section, we describe how to apply this result to determine the parity of the rank of Jacobians of curves. This generalises the argument for elliptic curves given in Theorem \ref{thm:recipe}. We also comment on how our construction relates to the parity conjecture and prove Theorem \ref{thm:PCconj}.

\subsection{Regulator constants in $S_n$}

Theorem \ref{thm:snbrauer} below shows that there is a large supply of representations of the form $\tau_{\Theta,p}$. The proof relies on understanding these representations in dihedral groups and on an induction theorem for permutation representations. 

For ease of notation, we will write the result in terms of characters of representations. Recall that a generalised character is a formal $\Z$-linear combination of characters of representations. By a permutation character we mean a $\Z$-linear combination of characters of permutation representations.

\begin{lemma}\label{lem:RCdihedral}
Suppose $G=C_2\!\times\! C_2$, $D_8$ (in which case set $p\!=\!2)$ or $D_{2p}$ for an odd prime $p$. For every character $\tau$ of $G$ of degree 2, there is a Brauer relation $\Theta$ such that $\tau_{\Theta,p}=\tau-\triv-\det\tau$.
\end{lemma}
\begin{proof}
See \cite[Examples 2.53, 2.54]{tamroot}.
\end{proof}

\begin{theorem}\label{thm:vind}
Let $G$ be a finite group. Every permutation character $\rho$ of $G$ of degree 0 and trivial determinant can be written in the form 
$$ 
  \rho = \sigma + \bar{\sigma} + \sum_i k_i \Ind_{H_i}^G (\tau_i - \triv_{H_i} - \det \tau_i),
$$
for some generalised degree 0 character $\sigma$ and degree 2 characters $\tau_i$ of subgroups $H_i$ that factor through quotients of $H_i$ isomorphic to $C_2\times C_2$, $D_8$ or $D_{2p_i}$ for some odd prime $p_i$, and $k_i\in\Z$.
\end{theorem}
\begin{proof}
See \cite[Theorem 1.2]{VInduction}.
\end{proof}

\begin{theorem}\label{thm:snbrauer}
Let $G=S_n$. Let $\rho$ be the irreducible $n\!-\!1$ dimensional character of $G$ that is the natural permutation character on $n$ points minus the trivial character, and $\epsilon=\det\rho$ the sign character. 
There are primes $p_i\le n$ and Brauer relations $\Theta_i$ in $G$ such that
$$
 -n\triv +  \epsilon + \rho \quad = \quad 2\sigma + \sum_i \tau_{\Theta_i, p_i} 
$$
for some character $\sigma$.
\end{theorem}
 
\begin{proof}
Note that $-n\triv +  \epsilon + \rho$ is a permutation character of degree zero and trivial determinant. By Theorem \ref{thm:vind}, it can be written in the form
$$
-n\triv +  \epsilon + \rho\quad =\quad \sigma + \bar\sigma + \sum_{i} k_i \Ind_{G_i}^G (\tau_i - \triv_{G_i} - \det\tau_i),
$$
for some $k_i\in\Z$, $G_i\le G$ and characters $\tau_i$ of $G_i$ of degree 2 that factor through a $C_2\!\times\! C_2$-, $D_8$- or $D_{2p_i}$-quotient of $G_i$ for an odd prime $p_i\le n$; we set $p_i=2$ in case of a $C_2\!\times\! C_2$- or $D_8$-quotient.

As all characters of $S_n$ are real-valued, $\sigma=\bar\sigma$, so that $\sigma+\bar\sigma=2\sigma$.

By Lemma \ref{lem:RCdihedral} 
there is a $G_i$-Brauer relation $\Theta_i'=\sum_j m_{j,i}H_{j,i}$ (with $m_{j,i}\in\Z$ and $H_{j,i}\le G_i$) such that $\tau_{\Theta_i',p_i}=\tau_i - \triv_{G_i} - \det\tau_i$. Taking $n_{j,i}=k_i m_{j,i}$, by \cite[Theorem 2.56(3)]{tamroot} the $G$-Brauer relation $\Theta_i=\sum_j n_{j,i}H_{j,i}$ has
$$
 \tau_{\Theta_i, p_i} = k_i \Ind_{G_i}^G (\tau_i - \triv_{G_i} - \det\tau_i).
$$
The result follows.
\end{proof}

The rough strategy for determining the parity of the rank of $\Jac_Y$ for a general curve $Y$ is now as follows.
Suppose for this discussion that $\sha$ is always finite. Write $Y$ as a degree $n$ cover of $\P^1$ with Galois group $S_n$, and let $X$ correspond to its Galois closure, e.g. $E$ and $B$ in Example \ref{ex:S3part1}. Taking the inner product of the formula in Theorem \ref{thm:snbrauer} with $\Jac_X(K)\otimes\Q$ and using Theorems \ref{thm:introgaloisdescent} and \ref{LocalFormula:Intro} gives an expression of the form
$$
 \rk \Jac_Y + \rk \Jac_{X/A_n} \equiv \text{(local data)} \bmod 2.
$$
This reduces the problem to determining the parity of $\rk \Jac_{X/A_n}$. As $A_n$ has index 2 in $S_n$, $X/A_n$ is a hyperelliptic curve, say $y^2=h(x)$ for $h$ of degree $m$. Applying this construction again to $X/A_n$ viewed as a degree $m$ cover (!) of $\P^1$ reduces the problem to understanding the parity of the rank of the Jacobian of another hyperelliptic curve which, this time, is always of the form $y^2=g(x^2)$ (like the curve $D$ in Example \ref{ex:S3part1}). The latter curve has $C_2\times C_2 $ in its automorphism group, and the parity of the rank of its Jacobian can be determined using Example \ref{ex:C2C2rankrecipe}.

We now formalise this construction.

\subsection{Arithmetic analogue of root numbers}

\begin{notation}
Let $L$ be a field.
We write $\P^1_x$ for $\P^1$ (over $L$) together with a choice of an element $x$ that generates its function field over $L$.
\end{notation}

\begin{definition} \label{definition:discriminant_curve}
Throughout this definition we consider curves over a field $L$ of characteristic 0.

(i) For a cover $Y\to \P^1$ of degree $n$ we write $X_Y$ for its $S_n$-closure. This is a Galois cover $X_Y\to \P^1$ with automorphism group $S_n$ and $X_Y/S_{n-1}\simeq Y$; see \cite{Etale_paper}, Section 2.5.
We write $\Delta_Y=X_Y/A_n$ for the associated ``discriminant curve'' (cf. \cite{Etale_paper}, Lemma 2.7).

(ii) For a cover $Y\to\P^1_x$, the function field of $\Delta_Y$ is a quadratic extension of $K(x)$, and so it can be written as $K(x)(\sqrt{a g(x)})$ for a unique monic squarefree $g(x)\in K[x]$ and an element $a \in K^\times$ that is well-defined up to multiplication by $K^{\times 2}$.
Strictly speaking, throughout this definition we work with \'etale algebras. That is by $K(x)(\sqrt{h(x)})$ we mean the algebra $K(x)[t]/(t^2-h(x))$, which is of degree 2 over $K(x)$ even when $h(x)$ is a perfect square.

Define $\Delta^1_Y$ to be the double cover of $\P^1_x$ with function field $K(x)(\sqrt{g(x)})$; it is the quadratic twist of the hyperelliptic curve $\Delta_Y$ by $a$. Writing $y\!=\!\sqrt{g(x)}$, $\Delta^1_Y: y^2=g(x)$ comes with a well-defined cover $\Delta^1_Y\to \P^1_y$, which only depends on $Y\to \P^1_x$.

Define $\Delta^2_Y$ to be the $C_2\!\times\! C_2$-cover of $\P^1_x$ with function field $K(x)(\sqrt{g(x)}, \sqrt{a})$.

(iii) For a cover $Y\to\P^1_x$ of the form $x^2=f(z)$, the curve $\Delta_Y$ is endowed with an action of $C_2\!\times\!C_2$ with quotient $\mathbb P^1_{x^2}$. Indeed, the discriminant (in $x$) of $f(z)-x^2$ is visibly a polynomial in $x^2$, so that $\Delta_Y$ is given by $y^2=ah(x^2)$ for some polynomial $h(x)$. Thus $C_2\!\times\!C_2$ acts on $\Delta_{Y}$ via the automorphisms $x \mapsto \pm x$ and $y \mapsto \pm y$.
\end{definition}

\begin{notation}
For every $n\ge 1$, fix a list of primes $p_i=p_i^{(n)}$ and of Brauer relations $\Theta_i=\Theta_i^{(n)}$ in $S_n$ that satisfy the conclusion of Theorem \ref{thm:snbrauer}.
\end{notation}

\begin{definition}\label{def:alpha} 
Let $\cK$ be a local field of characteristic $0$.

(i)
For a cover $Y\to \P^1$ of degree $n$ over $\cK$ define
$$
 \alpha(Y\to \P^1/\cK) = \sum_i \ord_{p_i^{(n)}} \Lambda_{\Theta_i^{(n)}}(X_Y/\cK).
$$

(ii)
For a cover $Y\to\P^1_x$ over $\cK$ define
$$
 \beta(Y\to \P^1_x/\cK) = \ord_2 \Lambda_{\Theta}(\Delta^2_Y/\cK),
$$
where $\Theta = 2C_2\!\times\! C_2 - C_2^{a} - C_2^{b}-C_2^{c}+ \{1\}$ is a $C_2\!\times\! C_2$-Brauer relation.

(iii)
For a cover $Y\to\P^1_x$ of the form $x^2\!=\!f(z)$ over $\cK$, define
$$
 \gamma(Y\to \P^1_x/\cK) = \ord_2 \Lambda_{\Theta}(\Delta_Y/\cK),
$$
where $\Theta = 2C_2\!\times\! C_2 - C_2^{a} - C_2^{b}-C_2^{c}+ \{1\}$ is a $C_2\!\times\! C_2$-Brauer relation of $\Delta_Y\to\P^1_{x^2}$.

(iv)
For a cover $Y\to\P^1_x$ over $\cK$ we define the ``arithmetic analogue of the local root number'' (with respect to the cover $Y\to \P^1_x$)  $w_{\rm{arith}}(Y/\cK)\in\{\pm 1\}$ as
$$
  w_{\rm{arith}}(Y/\cK) = (-1)^{\alpha(Y\to\P^1_{x}/\cK)} (-1)^{\beta(Y\to\P^1_{x}/\cK)} (-1)^{\alpha(\Delta^1_Y\to\P^1_{y}/\cK)} (-1)^{\gamma(\Delta^1_Y\to\P^1_{y}/\cK)}.
$$
\end{definition}

\begin{remark}
The constructions of curves and automorphism groups given in Definition \ref{definition:discriminant_curve} have been set up so as to commute with extensions of scalars. Thus, for example, if $Y\to \P^1$ is defined over a number field $K$ and $v$ is a place of $K$, then ``$X_Y/K_v$'' in Definition \ref{def:alpha} can be obtained by constructing $X_Y$ over $K$ from $Y/K$. This is essential in order to apply global results on Selmer groups (Theorem \ref{Theorem: Local Formula}).
\end{remark}

\begin{remark}
The invariant $w_{\rm{arith}}(Y/\cK)$ depends on the map $Y\to \P^1_x$, although we have suppressed it in the notation.
It also depends on the choice of Brauer relations in symmetric groups used in Theorem \ref{thm:snbrauer}. The latter is a purely group theoretic choice. It is applied to $S_n$ and $S_m$ where $n$ and $m$ are the degrees of $Y\to\P^1_x$ and $\Delta^1_Y\to \P^1_{y}$, respectively.
\end{remark}

\subsection{Parity of ranks of Jacobians}

\begin{theorem}\label{thm:snparity}
Let $Y\to \P^1$ be a cover of degree $n$ over a number field $K$. Let $X_Y\to\P^1$ be its $S_n$ closure and $\Delta_Y\!=\! X_Y/A_n$ its associated discriminant curve. 

(i) If the $p$-primary part of $\sha(\Jac_{X_Y})$ is finite for all primes $p \le n$, then
$$
   \rk \Jac_Y + \rk \Jac_{\Delta_Y} \equiv \sum_{v\textup{ place of }K} \alpha(Y\to\P^1/K_v) \mod 2.
$$

(ii) If Conjecture \ref{conj:brauerwishful} holds, then
$$
  w(\Jac_Y)w(\Jac_{\Delta_Y})= \prod_{v\textup{ place of }K} (-1)^{\alpha(Y\to\P^1/K_v)}
.$$
\end{theorem}

\begin{proof}
Write $\rho$ for the irreducible $n\!-\!1$ dimensional character of $G$ that is the natural permutation character on $n$ points minus the trivial character.

(i) We have
$$
\langle \Jac_{X_Y}(K), \triv\rangle = \rk \Jac_{\P^1} =0,
$$
$$
\langle \Jac_{X_Y}(K), \rho\rangle = \rk \Jac_{X_Y/S_{n-1}} = \rk \Jac_Y, 
$$
$$
\langle \Jac_{X_Y}(K), \det\rho\rangle = \rk \Jac_{X_Y/A_n} = \rk \Jac_{\Delta_Y}. 
$$
By Theorem \ref{Theorem: Local Formula} and the assumed finiteness of the $p_i^{(n)}$-primary part of $\sha$,
$$
\langle \Jac_{X_Y}(K),  \tau_{\Theta_i^{(n)},p_i^{(n)}} \rangle \equiv \sum_v \ord_{p_i^{(n)}} \Lambda_{\Theta_i^{(n)}}(Y/K_v) \mod 2.
$$
Taking the sum over $i$, we get
$$
  \rk \Jac_{\Delta_Y}+ \rk \Jac_Y = \langle \Jac_{X_Y}(K), -n\triv+\rho+\det\rho\rangle \equiv \sum_v \alpha(Y\to\P^1/K_v) \mod 2.
$$

(ii) 
The proof is analogous to (i). Using Proposition \ref{prop:artinformalism} we see that 
$$
w(X_Y^{\triv})\!=\!1, \qquad w(X_Y^{\rho})\!=\!w(Y) \qquad \text{and}\qquad w(X_Y^{\det\rho})=w(\Delta_Y).
$$
 By Conjecture \ref{conj:brauerwishful}, $w(X_Y^{\tau_{\Theta_i^{(n)}, p_i^{(n)} }})= \prod_v (-1)^{\ord_{p_i^{(n)}}\Lambda_{\Theta_i^{(n)}}(Y/K_v)}$. Taking the product over $i$ gives the result.
\end{proof}

\begin{lemma}\label{lem:delta2parity}
Let $Y\to \P^1_x$ be a cover over a number field $K$. 

(i) If $\sha(\Jac_{\Delta_Y})$ and $\sha(\Jac_{\Delta^1_Y})$ have finite 2-primary part, then
$$
   \rk \Jac_{\Delta_Y} + \rk \Jac_{\Delta^1_Y}  \equiv  \sum_{v\textup{ place of }K} \beta(Y\to\P^1_x/K_v) \mod 2.
$$

(ii) If Conjecture \ref{conj:brauerwishful} holds, then
$$
  w(\Jac_{\Delta_Y})w(\Jac_{\Delta^1_Y}) = \prod_{v\textup{ place of }K} (-1)^{\beta(Y\to\P^1_x/K_v)}.
$$
\end{lemma}

\begin{proof}
(i) This is Theorem \ref{Theorem: Local Formula} applied to the $C_2\!\times\! C_2$-Brauer relation of $\Delta_Y^2\to\P^1_x$ with $p=2$. 

(ii) This is Proposition \ref{prop:artinformalism} combined with Conjecture \ref{conj:brauerwishful} for the same Brauer relation.
\end{proof}

\begin{lemma}\label{lem:deltadeltaparity}
Let $Y\to \P^1_x$ be a cover of the form $x^2\!=\!f(z)$ over a number field $K$.

(i) If $\sha(\Jac_{\Delta_Y})$ has finite 2-primary part, then
$$
 \rk \Jac_{\Delta_Y}  \equiv  \sum_{v\textup{ place of }K} \gamma(Y\to\P^1_{x}/K_v) \mod 2.
$$

(ii) If Conjecture \ref{conj:brauerwishful} holds, then
$$
  w(\Jac_{\Delta_Y}) = \prod_{v\textup{ place of }K} (-1)^{\gamma(Y\to\P^1_{x}/K_v)}.
$$
\end{lemma}

\begin{proof}
(i) This is Theorem \ref{Theorem: Local Formula} applied to the $C_2\!\times\! C_2$-Brauer relation of $\Delta_Y\to\P^1_{x^2}$ with $p=2$.
(ii) This is Proposition \ref{prop:artinformalism} combined with Conjecture \ref{conj:brauerwishful} for the same Brauer relation.
\end{proof}

\begin{theorem}\label{thm:PCanalogue}
Let $Y\to \P^1_x$ be a degree $n$ cover defined over a number field $K$, and let $X_Y$ be its $S_n$-closure.
Let $\Delta^1_Y\to \P^1_{y}$ be the associated cover $y^2=g(x)$, let $m$ be its degree and $X_{\Delta^1_Y}$ its $S_m$-closure.

(i) If $\sha(\Jac_{X_Y})$ and $\sha(\Jac_{X_{\Delta^1_Y}})$
have finite $p$-primary part for $p\le n$ and $p\le m$, 
respectively, then
$$   
(-1)^{\rk \JY} =  \prod_{v\textup{ place of }K} w_{\rm{arith}}(Y/K_v).
$$

(ii) If Conjecture \ref{conj:brauerwishful} holds, then
$$
 w(\JY) = \prod_{v\textup{ place of }K} w_{\rm{arith}}(Y/K_v).
$$
\end{theorem}

\begin{proof}
Note that if the $p$-primary part of $\sha(\Jac_{X_Y})$ (respectively of $\sha(\Jac_{X_{\Delta^1_Y}})$) is finite, then so is that of $\sha(\Jac_{\Delta_Y})$ (respectively $\sha(\Jac_{\Delta_{\Delta^1_Y}})$) by Theorem \ref{thm:introgaloisdescent}(6).

Now combine Theorem \ref{thm:snparity} for $Y\to \P^1_x$ and for $\Delta^1_Y\to\P^1_{y}$ together with Lemma \ref{lem:delta2parity} for $Y\to \P^1_x$ and with Lemma \ref{lem:deltadeltaparity} for $\Delta^1_Y\to\P^1_{y}$.
\end{proof}

\begin{corollary}\label{cor:howtoproveparity}
If the Shafarevich--Tate conjecture and Conjecture \ref{conj:brauerwishful} hold for all Jacobians of curves over a fixed number field $K$, then so does parity conjecture, that is
$$
  (-1)^{\rk \JY} = w(\JY)
$$
for every curve $Y$ defined over $K$.
\end{corollary}

\section{On the parity conjecture for Jacobians}\label{Sec_ss_parity}  

We end with a discussion on what one would need in order to prove the parity conjecture for general Jacobians, assuming finiteness of Tate--Shafarevich groups.
Recall from Corollary \ref{cor:howtoproveparity} that the supply of Brauer relations is sufficiently large as to control the parity of ranks of all Jacobians using the local invariants $\Lambda$.
What is required is the connection to local root numbers given by Conjecture \ref{conj:brauerwishful}. Recall that this asserts that
$$
\prod_{v\textup{ place of }K} w(X^{\tau_{\Theta,p}}/K_v) (-1)^{\ord_p\Lambda_\Theta(X/K_v)}=1
$$
for every curve $X$ with a finite group $G$ acting by automorphisms, every prime $p$ and every Brauer relation $\Theta$ for $X$.

As the following theorem illustrates, the local terms $w(X^{\tau_{\Theta,p}}/K_v)$ and $(-1)^{\ord_p\Lambda_\Theta(X/K_v)}$ often match. Part (1) of the theorem is a special case of the root number formula given in Proposition~\ref{prop:rootnumber}. We will not prove part (2) here; it will appear in forthcoming work \cite{morganlattice}. 

\begin{theorem}\label{thm:compatibility}
Let $X$ be a curve over a local field $\cK$ of characteristic $0$ and $G$ a finite group of $\cK$-automorphisms of $X$, such that $\Omega^1(\Jac_X)$ is self-dual as a $G$-representation.
For an orthogonal representation $\tau$ of $G$:
\begin{enumerate}
\item[(1a)]
 $w(X^\tau/\cK) = (-1)^{\langle\tau,\Omega^1(\Jac_X)\rangle}$ if $\cK$ is archimedean; 
\item[(1b)]
 $w(X^\tau/\cK) = (-1)^{\langle \tau, (\dcg_{\Jac_X}\otimes\C)^{\Frob_\cK} \rangle}$ if $\cK$ is non-archimedean, $X/\cK$ is semistable.
\end{enumerate}
For every odd prime $p$ and pseudo Brauer relation $\Theta$ for $X$:
\begin{enumerate}
\item[(2a)] 
$\ord_p \Lambda_\Theta(X/\cK) = \langle \tau_{\Theta,p}, \Omega^1(\Jac_X) \rangle$ if $\cK=\C$;
\item[(2b)] 
$\ord_p \Lambda_\Theta(X/\cK) = \langle \tau_{\Theta,p}, (\dcg_{\Jac_X}\otimes\C)^{\Frob_\cK} \rangle$ if $X/\cK$ is semistable and either $\cK$ non-archimedean of residue characteristic different from $p$ or if $\cK$ is an extension of $\Q_p$ of even degree.
\end{enumerate}
\end{theorem} 

Since $\tau_{\Theta,p}$ is always orthogonal, 
under the assumptions of the theorem, \[w(X^{\tau_{\Theta,p}}/\cK)=(-1)^{\ord_p\Lambda_\Theta(X/\cK)}.\] In particular, Conjecture \ref{conj:brauerwishful} holds for curves over number fields that satisfy these local constraints at all places. In view of Theorem \ref{Theorem: Local Formula}, we deduce the following case of the $p$-parity conjecture for $X^{\tau_{\Theta,p}}$ (see Conjecture \ref{conj:wishful}):

\begin{corollary} \label{cor:mostgeneralPC}
Let $X$ be a curve over a number field $K$, $p$ an odd prime and $G$ a finite group of $K$-automorphisms of $X$ such that $\Omega^1(\JX)$ is self-dual as a $G$-representation.
Suppose moreover that $X$ is semistable, $K$ is totally complex and that $[K_v:\mathbb{Q}_p]$ is even for all $v\mid p$. 
Then for all pseudo Brauer relations $\Theta$ for $X$ 
$$
 (-1)^{\langle\tau_{\Theta,p}, \mathcal{X}_p(\JX) \rangle} =  w(X^{\tau_{\Theta,p}}).
$$
\end{corollary}

A more general version is proved in \cite{morganlattice}, which removes some of the assumptions at $\ell=p$ and allows real places. The local terms no longer agree in such generality. However, the discrepancy can be expressed in terms of a certain Artin symbol, which, when ones takes the product over all places of a number field, vanishes by the product formula. The case $p\!=\!2$ is more difficult, although not without progress. In particular, the parity conjecture for hyperelliptic curves with suitably good local behaviour should be within reach (see \cite[\S1.3.3]{hollythesis}).

%
%

\end{document}